\documentclass{article}
\usepackage{graphicx} 
\usepackage{fullpage}
\usepackage{amsfonts}
\usepackage{amsmath}
\usepackage{amsthm}
\usepackage{amssymb}
\usepackage{color}
\usepackage{xcolor}
\usepackage[colorlinks,linkcolor=blue, citecolor=blue]{hyperref}
\usepackage[section]{algorithm}
\usepackage{algpseudocode}
\usepackage{framed}
\usepackage{footnote}
\usepackage{geometry}
\usepackage{xfrac}
\usepackage{todonotes}
\usepackage{tabularx,booktabs}
\usepackage{svg,cite}

\usepackage{psfrag}
\usepackage{pdflscape}
\usepackage{multirow}
\usepackage{blindtext} 
\usepackage{enumitem}
\usepackage{subcaption}
\usepackage[capitalise]{cleveref}
\crefname{subsection}{Subsection}{Subsections}
\crefname{framework}{Framework}{Frameworks}

\usepackage{abs}

\newtheorem{theorem}{Theorem}[section]

\newtheorem{lemma}[theorem]{Lemma}

\newtheorem{assumption}{Assumption}[section]

\newtheorem{condition}{Condition}[section]
\newtheorem{corollary}[theorem]{Corollary}

\newtheorem{remark}[theorem]{Remark}

\numberwithin{algorithm}{section}
\newfloat{framework}{htbp}{loa}
\floatname{framework}{Framework}
\makeatletter
\let\c@framework\c@algorithm
\makeatother

\newcommand{\papertitle}{Retrospective Approximation Sequential Quadratic Programming for Stochastic Optimization with General Deterministic Nonlinear Constraints}

\newcommand{\paperauthora}{Albert S. Berahas}
\newcommand{\paperauthoraaffiliation}{Department of Industrial and Operations Engineering, University of Michigan}
\newcommand{\paperauthoraemail}{albertberahas@gmail.com}
\newcommand{\paperauthorb}{Raghu Bollapragada}
\newcommand{\paperauthorbaffiliation}{Operations Research and Industrial Engineering Program, University of Texas at Austin}
\newcommand{\paperauthorbemail}{raghu.bollapragada@utexas.edu}
\newcommand{\paperauthorc}{Shagun Gupta}

\newcommand{\paperauthorcemail}{shagungupta@utexas.edu}

\begin{document}
\title{\papertitle}

\author{\paperauthora\footnotemark[1]
    \and \paperauthorb\footnotemark[2]\ \footnotemark[3]
    \and \paperauthorc\footnotemark[2]}

\maketitle

\renewcommand{\thefootnote}{\fnsymbol{footnote}}
\footnotetext[1]{\paperauthoraaffiliation. (\url{\paperauthoraemail})}
\footnotetext[2]{\paperauthorbaffiliation. (\url{\paperauthorbemail},\url{\paperauthorcemail})}
\footnotetext[3]{Corresponding author.}
\renewcommand{\thefootnote}{\arabic{footnote}}

\begin{abstract}
    In this paper, we propose a framework based on the Retrospective Approximation (\texttt{RA}) paradigm to solve optimization problems with a stochastic objective function and general nonlinear deterministic constraints. This framework sequentially constructs increasingly accurate approximations of the true problems which are solved to a specified accuracy via a deterministic solver, thereby decoupling the uncertainty from the optimization. Such frameworks retain the advantages of deterministic optimization methods, such as fast convergence, while achieving the optimal performance of stochastic methods without the need to redesign algorithmic components. For problems with general nonlinear equality constraints, we present a framework that can employ any deterministic solver and analyze its theoretical work complexity. We then present an instance of the framework that employs a deterministic Sequential Quadratic Programming (\texttt{SQP}) method and that achieves optimal complexity in terms of gradient evaluations and linear system solves for this class of problems. For problems with general nonlinear constraints, we present an \texttt{RA}-based algorithm that employs an \texttt{SQP} method with robust subproblems. Finally, we demonstrate the empirical performance of the proposed framework on multi-class logistic regression problems and benchmark instances from the CUTEst test set, comparing its results to established methods from the literature.
\end{abstract}
\section{Introduction} \label{sec:Introduction}

In this paper, we propose a framework for solving optimization problems with a stochastic objective function and general nonlinear deterministic constraints of the form,
\begin{equation} \label{eq:intro_problem}
    \min_{x \in \Rmbb^n} \; f(x) \quad \text{s.t.} \quad c_E(x) = 0, \quad c_I(x) \leq 0,
\end{equation}
where $f : \Rmbb^n \rightarrow \Rmbb$, $c_E : \Rmbb^n \rightarrow \Rmbb^{m_E}$ and $c_I : \Rmbb^n \rightarrow \Rmbb^{m_I}$ are continuously differentiable functions. We consider two forms for the objective function,
\newcounter{eqnA}
\newcounter{eqnB}
\refstepcounter{equation} 
\setcounter{eqnA}{\value{equation}} 
\label{eq:intro_deter_error_obj}
\refstepcounter{equation} 
\setcounter{eqnB}{\value{equation}} 
\label{eq:intro_stoch_error_obj}
\begin{align*}
  f(x) = \frac{1}{|\Scal|} \sum_{\xi \in \Scal} F(x, \xi) \qquad \qquad \text{(\ref{eq:intro_deter_error_obj})}, 
  \qquad \qquad \text{and} \qquad \qquad
   f(x) = \Embb [F(x, \xi)], \qquad \qquad \text{(\ref{eq:intro_stoch_error_obj})}
\end{align*}
where the objective function \eqref{eq:intro_deter_error_obj} results in a deterministic finite-sum problem over a finite dataset $\Scal = \{\xi_1, \xi_2, \dots\}$ with $F : \Rmbb^n \times \Scal \rightarrow \Rmbb$, and the objective function \eqref{eq:intro_stoch_error_obj} results in an expectation minimization problem over the random variable $\xi$ with the associated probability space $(\Xi, \Omega, \mathcal{P})$, $F : \Rmbb^n \times \Xi \rightarrow \Rmbb$ and $\Embb[\cdot]$ denotes the expectation with respect to $\Pcal$. Such problems arise in various applications including machine learning \cite{le2022survey,mehrabi2021survey}, computational physics for climate modeling \cite{zhao2019physics, kotary2021end}, portfolio optimization \cite{perold1984large}, multistage optimization \cite{shapiro2021lectures}, and optimal power flow \cite{summers2015stochastic}.

Methods for solving \eqref{eq:intro_problem} with a deterministic objective function have been extensively studied; see e.g., \cite{nocedal2006numerical, bertsekas2009convex}.
A key challenge in solving problem \eqref{eq:intro_problem} is balancing the two possibly competing goals of minimizing the objective function and satisfying the constraints at each step of the method.
Recently, stochastic methods have been developed for minimizing \eqref{eq:intro_deter_error_obj} and \eqref{eq:intro_stoch_error_obj}, although many of these methods are limited to problems with only equality constraints.
These include penalty-based techniques \cite{ravi2019explicitly,nandwani2019primal,roy2018geometry,chen2018constraint,wang2017penalty}, projection methods \cite{lan2012optimal,lan2020first,ghadimi2016mini}, Sequential Quadratic Programming (\texttt{SQP}) \cite{berahas2021sequential,berahas2022adaptive,na2023adaptive,curtis2021inexact,curtis2023sequential,qiu2023sequential,curtis2024worst,o2024two, fang2024fully,na2023inequality}, and Interior Point methods \cite{curtis2023stochastic,curtis2024single,curtis2025interior}.
These methods retain the search direction computation mechanism of the deterministic approaches by using an approximation of the objective function gradient. However, these approaches redesign several components of the deterministic algorithms to make them suitable for stochastic settings and, as a result, lose some of the associated benefits such as automatic step size selection, and result in slower convergence.

The Retrospective Approximation (\texttt{RA}) framework is an alternate approach for solving stochastic optimization problems \cite{pasupathy2011introspective,chen1994retrospective}. The framework iteratively constructs deterministic subsampled approximations of increasing accuracy and solves each one to a certain specified accuracy. As a result, the approach decouples the uncertainty from the optimization process, enabling the use of deterministic methods to solve stochastic problems. This framework has been successfully employed in various areas, including stochastic root finding \cite{pasupathy2010choosing,chen2001stochastic}, two-stage stochastic optimization \cite{pasupathy2021adaptive}, simulation optimization \cite{jalilzadeh2016eg,deng2009variable}, and optimal design problems \cite{vondrak2009adaptive,royset2006optimal,phelps2016optimal,polak2008efficient}. More recently, it has also been proposed for unconstrained stochastic optimization \cite{newton2023retrospective,newtonretrospective}.
As compared to deterministic methods, these methods are more efficient due to their use of less expensive subsampled function (and gradient) approximations in the initial stages, and as compared to fully stochastic methods, these methods achieve the optimal iteration and work complexity, can attain high accuracy solutions and can take advantage of developments in deterministic methods.

In this work, we propose frameworks based on the \texttt{RA} paradigm to solve stochastic optimization problems with general deterministic nonlinear constraints, and employ the \texttt{SQP} method, a state-of-the-art technique for optimization problems with general nonlinear constraints \cite{nocedal2006numerical}, as our deterministic optimization method. We propose: $(1)$ a framework for problems with only equality constraints that can incorporate any deterministic method, and an instance that employs a line search \texttt{SQP} method \cite{powell2006fast,han1977globally}; and, $(2)$ an \texttt{RA} based algorithm for general nonlinear constraints that utilizes the robust-\texttt{SQP} method \cite{burke1989robust}.
We analyze the theoretical convergence of the two frameworks and evaluate their empirical performance on regularized logistic regression multi-class classification problems as well as the S2MPJ CUTEst problem set \cite{gratton2024s2mpj}, comparing them with established methods from the literature.

\subsection{Constrained Stochastic Optimization} \label{sec:stoch_constrained}

In this section, we present the fundamentals of constrained optimization for smooth problems, review the literature on constrained stochastic optimization, introduce the \texttt{SQP} method, and outline our contributions relative to existing work.
We define the Lagrangian for \eqref{eq:intro_problem} with $\lambda_E \in \Rmbb^{m_E}$ and $\lambda_I \in \Rmbb^{m_I}$ as $\Lcal(x, \lambda_E, \lambda_I) = f(x) + c_E(x)^T \lambda_E + c_I(x)^T \lambda_I$.
The Karush-Kuhn-Tucker (KKT) conditions state that $x^*$ is a first-order stationary point if there exist $\lambda^*_E$ and $\lambda^*_I$ such that,
\begin{equation} \label{eq:KKT_conditions}
    \nabla_x \Lcal (x^*, \lambda^*_E, \lambda^*_I) = 0, \quad
    c_E(x^*) = 0, \quad c_I(x^*) \leq 0, \quad c_I(x^*) \odot \lambda^*_I = 0, \quad \lambda^*_I \geq 0,
\end{equation}
where $c_I(x^*) \odot \lambda^*_I$ denotes the element-wise product of the two vectors.

While many powerful methods for obtaining a (or an approximate) KKT point in the deterministic optimization setting have existed for decades, see e.g., \cite[Chapters 12-19]{nocedal2006numerical}, counterparts for constrained problems with stochastic objective functions are significantly fewer.
Penalty-based methods \cite{ravi2019explicitly,nandwani2019primal,roy2018geometry,chen2018constraint,wang2017penalty} were among the first methods extended to this problem class due to their intuitive design and ease of implementation. However, the empirical performance of penalty-based methods is sensitive to the choice of penalty function and parameter.
Projection methods \cite{lan2012optimal,lan2020first,ghadimi2016mini} have also been proposed for this problem class for instances where the projection mapping onto the feasible region is easily computable. Moreover, stochastic interior points methods \cite{curtis2023stochastic,curtis2024single,curtis2025interior,dezfulian2024convergence} have also been recently developed. While these methods have sound theoretical results and show promising empirical performance, they are primarily applicable to bound constraints and require strong assumptions to handle the general nonlinear constraint setting. Finally, \texttt{SQP} methods have been proposed and are discussed in detail below.

\texttt{SQP} methods \cite[Chapter 18]{nocedal2006numerical} are iterative methods that update the current estimate of the solution by computing a search direction that minimizes a quadratic model of the objective function subject to a linear approximation of the constraints, i.e., at each iterate $x$,
\begin{equation} \label{eq:SQP_quad_model}
    \begin{aligned}
        \min_{d \in \Rmbb^n} & \; \tfrac{1}{2} d^T H d + \nabla   f(x)^T d   \\
        \text{s.t.} \,\, & \;c_E(x) + \nabla c_E(x)^T d = 0 \\
        & \; c_I(x) + \nabla c_I(x)^T d \leq  0,  
    \end{aligned}
\end{equation}
where $d \in \Rmbb^{n}$ is the search direction and $H \in \Rmbb^{n \times n}$ is an approximation of the Hessian of the Lagrangian.
Several methods have been developed for constrained stochastic optimization settings that incorporate an approximation of the gradient of the objective function in \eqref{eq:SQP_quad_model}.
Early methods were primarily designed for problems with only equality constraints \cite{berahas2021sequential,o2024two,berahas2025sequential,curtis2024worst}. Building upon these methods, subsequent works in the equality-constrained setting leveraged features from deterministic \texttt{SQP} methods such as inexact solutions to \eqref{eq:SQP_quad_model} \cite{curtis2021inexact,berahas2022adaptive}, adaptive test to control the gradient approximation accuracy \cite{berahas2022adaptive,na2023adaptive}, and line search and trust-region variants \cite{fang2024fully,berahas2025sequential}. All aforementioned methods use, modify and adapt mechanisms from the deterministic \texttt{SQP} methodology. That said, to overcome the challenges that arise due to the stochastic nature of \eqref{eq:intro_problem}, most advanced features of deterministic \texttt{SQP} methods such as quasi-Newton Hessian approximations, inexact search direction solutions, and line search mechanisms are omitted or used only in solitude, or strong assumptions on the behavior of the algorithm or problem are imposed. By incorporating a deterministic \texttt{SQP} solver within the \texttt{RA} framework, one can use all the advanced \texttt{SQP} features in harmony when solving constrained stochastic optimization problems. 

When solving equality-constrained stochastic problems using \texttt{SQP}, computing a search direction via \eqref{eq:SQP_quad_model} reduces to solving a linear system of equations. As a result, the two main computational costs of \texttt{SQP} methods in this setting are the number of gradient evaluations and the number of linear systems to be solved. Previously proposed methods such as \cite{berahas2021sequential,o2024two,berahas2025sequential,curtis2024worst} achieve the optimal complexity for gradient evaluations $\Ocal(\epsilon^{-4})$ but exhibit high complexity in terms of the number of linear system solves $\Ocal(\epsilon^{-4})$. In contrast, the method in \cite{berahas2022adaptive} achieves the optimal complexity in terms of the number of linear system solves $\Ocal(\epsilon^{-2})$ but incurs a higher complexity in terms of gradient evaluations $\Ocal(\epsilon^{-2(1 + \nu)})$, where $\nu > 1$. Our proposed \texttt{RA}-based \texttt{SQP} method for stochastic optimization with deterministic equality constraints achieves optimal complexity in terms of both costs.

The general equality and inequality constrained setting presents additional challenges in the design of \texttt{SQP} methods primarily, due to the difficulty in solving the \texttt{SQP} subproblem \eqref{eq:SQP_quad_model} while ensuring feasibility. Few \texttt{SQP}-based methods have been developed to handle both constraint types with a stochastic objective function. These include an active-set method \cite{na2023inequality}, a decomposition approach \cite{curtis2023sequential}, and a two-step method \cite{qiu2023sequential}. The integration of deterministic \texttt{SQP} methods within the \texttt{RA} framework allows for the incorporation of general nonlinear constraints. However, the choice of method and other aspects of \texttt{RA} must be carefully considered to ensure strong empirical performance in the presence of inequality constraints, as will be discussed in Section \ref{sec:Inequality_constrained}.

\subsection{Retrospective Approximation (\texttt{RA})} \label{sec:Retrospective_Approximation}

In this section, we describe the \texttt{RA} framework and our approach for applying it to constrained stochastic optimization.
The framework consists of nested outer and inner loops. 
In each outer iteration, a deterministic subsampled problem is constructed as a sample average approximation of the objective function in \eqref{eq:intro_problem},
\begin{align*}
    &\min_{x \in \Rmbb^n} F_{S}(x) \quad s.t. \quad c_E(x) = 0, \quad c_I(x) \leq 0, \label{eq:subsampled_problem} \numberthis \\
    \text{where} \quad & F_{S}(x) = \frac{1}{|S|} \sum_{\xi \in S} F(x, \xi) \quad \text{and} \quad g_{S}(x) = \frac{1}{|S|} \sum_{\xi \in S} \nabla F(x, \xi),
\end{align*}
$S \subseteq \Scal$ for \eqref{eq:intro_deter_error_obj} and $S$ is a set of independent and identically distributed (i.i.d.) samples drawn from the probability space $(\Xi, \Omega, \mathcal{P})$ for \eqref{eq:intro_stoch_error_obj}.
The deterministic subsampled problem \eqref{eq:subsampled_problem} is then solved to a specified accuracy in the inner loop.
The \texttt{RA} approach assumes greater control over the samples compared to other works in the constrained stochastic optimization literature that are in the fully stochastic regime and use only noisy gradient approximations \cite{curtis2023stochastic, berahas2021sequential, o2024two}. The approach requires objective function (and gradient) evaluations using the same sample set across iterates and the ability to control the sample size (accuracy), which is feasible in most practical applications discussed earlier. Once the inner loop terminates, a new batch of samples is obtained, and the formulated subsampled problem is warm-started at the output of the previous outer iteration. The framework is detailed in \cref{alg:RA_base} and the following remark.

\begin{framework}[H]
    \caption{Retrospective Approximation for Constrained Stochastic Optimization}
    \textbf{Inputs}: Initial iterate $x_{0,0}$, batch size sequence $\{|S_k|\}$, termination test sequence $\{\Tcal_k\}$. 
    \begin{algorithmic}[1]
        \For{$k = 0, 1, 2, \dots$}
            \State Construct subsampled problem \eqref{eq:subsampled_problem} from sample set $S_k$ \label{RA_base:subsample}
            \For{$j = 0, 1, 2, \dots$}
                \State \textbf{if} $\Tcal_k$ is satisfied; $N_k = j$, \textbf{break} \label{RA_base:termination}
                \State Update $x_{k, j + 1}$ \label{RA_base:update}
            \EndFor
            \State Set $x_{k+1, 0} = x_{k, N_k}$ \label{RA_base:last_iterate}
        \EndFor
    \end{algorithmic}
    \label{alg:RA_base}
\end{framework}

\begin{remark}
    We make the following remarks about \cref{alg:RA_base}.
    \begin{enumerate}
        \item \textbf{Indexing}: For quantities with double indexing, e.g., $x_{k, j}$, the first index represents the outer iteration and the second denotes the inner iteration. Quantities with a single index correspond solely to an outer iteration.
        \item \textbf{Outer Loop}: One obtains a sample set $S_k$ of pre-specified batch size and formulates the deterministic subsampled problem in Line \ref{RA_base:subsample}, independent of the current iterate $x_{k, 0}$. The subsampled problem is solved using a deterministic solver initialized at $x_{k, 0}$ in the inner loop.
        \item \textbf{Inner Loop (Lines \ref{RA_base:termination}-\ref{RA_base:update})}: The inner loop terminates when the specified accuracy for the subsampled problem is achieved, as tested in Line \ref{RA_base:termination} using the termination criterion $\Tcal_k$ for the outer iteration $k$. If the criterion is not met, a new iterate $x_{k, j+1}$ is obtained in Line \ref{RA_base:update}. Since the goal is to solve a deterministic subsampled problem to a prescribed accuracy, one can employ deterministic optimization methods to update the iterate within the inner loop.
    \end{enumerate}
\end{remark}

\cref{alg:RA_base} requires the specification of three key components: $(1)$ the batch size sequence $\{|S_k|\}$, $(2)$ the termination test $\{\Tcal_k\}$, and $(3)$ the update mechanism in the inner loop. The sequence $\{|S_k|\}$ is an increasing sequence. In previous work \cite{newtonretrospective,curtis2023sequential,royset2006optimal,newton2023retrospective}, this sequence has been predefined. In this work, an adaptive strategy tailored to constrained optimization settings is employed to determine the batch size at the start of each outer iteration.
In earlier works, the termination test $\Tcal_k$ either used a predefined sequence or adaptively determined $N_k$, the number of inner loop iterations, at the start of each outer iteration \cite{vondrak2009adaptive,royset2006optimal,phelps2016optimal,polak2008efficient}. Our goal is to adaptively set the required solution accuracy for each subsampled problem, similar to \cite{newtonretrospective,newton2023retrospective}.
The main challenge lies in determining how to measure solution accuracy. For unconstrained smooth optimization, the subsampled gradient norm is the natural choice, as done in \cite{newtonretrospective,newton2023retrospective}.
However, for constrained optimization, this estimation is more complex, with various options depending on the type of constraints (equality or general) and the update method in the inner loop.
In this work, we examine these options across different scenarios. Due to the adaptive nature of the batch size and termination criterion, we are also able to provide detailed convergence rates and complexity results for the number of gradient evaluations and inner loop iterations, unlike previous works \cite{vondrak2009adaptive, royset2006optimal, phelps2016optimal}.
Lastly, we use deterministic \texttt{SQP} methods in the inner loop to update the iterate, unlike earlier work that employed projection-based methods to implicitly ensure feasibility \cite{vondrak2009adaptive,royset2006optimal,phelps2016optimal,polak2008efficient}, and solve problems with general nonlinear constraints. 

\subsection{Contributions} \label{sec:Contributions}

We propose frameworks for solving optimization problems with a stochastic objective function and deterministic general nonlinear constraints \eqref{eq:intro_problem} based on \texttt{RA}. Our contributions can be summarized as follows.
\begin{enumerate}
    \item We propose a framework for solving stochastic optimization problems with general nonlinear equality constraints that can incorporate any deterministic solver (see \cref{alg:Equality_Constrained_RA}). Within this framework, we introduce adaptive conditions for selecting the batch size sequence and the termination test sequence, which lead to fast linear convergence of the outer iterates (\cref{th:EQ_outer_iter_complexity}). Furthermore, assuming that the deterministic solver exhibits sublinear convergence, we establish complexity guarantees for the finite-sum objective \eqref{eq:intro_deter_error_obj} that match those of deterministic methods applied to the full objective (\cref{th:EQ_work_complexity}). For the expectation-based problem \eqref{eq:intro_stoch_error_obj}, the framework achieves optimal complexity with respect to the number of gradient evaluations, $\Ocal(\epsilon^{-4})$, and the number of inner loop iterations, $\Ocal(\epsilon^{-2})$ (\cref{th:EQ_work_complexity}).
    \item We propose a variant of the framework for equality constraints that uses the deterministic \texttt{SQP} method (see \cref{alg:Equality_Constrained_RA_SQP}) in the inner loop for problems with equality constraints. This variant introduces alternate customized 
    conditions for the batch size sequence and termination criteria that leverage information from the \texttt{SQP} method to enhance empirical performance.
    The algorithm achieves optimal theoretical complexity results for number of gradient evaluations, $\Ocal(\epsilon^{-4})$, and number of \texttt{SQP} linear system solves, $\Ocal(\epsilon^{-2})$ (\cref{th:EQ_SQP_well_posed}), and also allows for the addition of advanced techniques from deterministic methods to enhance empirical performance in stochastic settings. These include employing inexact \texttt{SQP} subproblem solutions and the use of quasi-Newton Hessian approximations.
    \item We propose an \texttt{RA} based algorithm for stochastic problems with general (equality and inequality) nonlinear constraints (see \cref{alg:Inequality_Constrained_RA_SQP}), an area with relatively limited literature \cite{na2023inequality, curtis2023sequential, qiu2023sequential}. Our proposed algorithm employs deterministic robust-\texttt{SQP} subproblems \cite{burke1989robust} within the inner loop, and we present two specific instances of such subproblems. Both lead to similarly fast linear convergence in the outer iterations (\cref{th:outer_complexity_ineq}), but differ in practical performance due to variations in computational cost and subproblem accuracy.
    \item We illustrate the empirical performance of the proposed methods (\cref{alg:Equality_Constrained_RA_SQP} and \cref{alg:Inequality_Constrained_RA_SQP}) on regularized logistic regression multi-class classification problems and the S2MPJ CUTEst problem set \cite{gratton2024s2mpj}. We show the performance benefits of utilizing deterministic solvers and the use of techniques such as quasi-Newton Hessian approximations and inexact subproblem solves in harmony with the stochasticity. We also compare with algorithms from the literature.
\end{enumerate}

\subsection{Notation} \label{sec:Notation}
Let $\Rmbb$ denote the set of real numbers, $\Rmbb^n$ denote the set of $n$ dimensional real vectors, and $\Rmbb^{n \times m}$ denote the set of $n$ by $m$ real matrices. Unless specified otherwise, $\|\cdot\|$ represents the Euclidean norm of a vector and the Frobenius norm of a matrix, $|\cdot|$ denotes the absolute value of a real number and the cardinality of a set, and $[\cdot]_+$ denotes the positive elements of a vector. The ceiling function is denoted as $\lceil \cdot\rceil$. The element wise product of $a, b \in \Rmbb^n$ is denoted as $a \odot b$. The vector of all ones of dimension $n$ is denoted as $e_n$, and $I_n$ denotes the identity matrix of dimension $n$. Expectation and variance with respect to the distribution $\Pcal$ are denoted as $\Embb[ \cdot ]$ and $\Var( \cdot )$, respectively.

\subsection{Paper Organization} \label{sec:Paper_Organization}
The paper is organized as follows. In \cref{sec:Equality_constrained}, we formalize the framework for equality constrained problems. We describe the framework for general deterministic solvers in \cref{sec:General_Deterministic_Solver_alg} and provide theoretical convergence and complexity results in \cref{sec:General_Deterministic_Solver_theory}.
We then describe the \texttt{SQP} variant of the framework for equality constrained problems in \cref{sec:SQP_based_algorithm_eq} and the corresponding theoretical analysis in \cref{sec:SQP_based_algorithm_eq_theory}. In \cref{sec:Inequality_constrained}, we present an \texttt{RA} based method for problems with general nonlinear constraints using the robust-\texttt{SQP} method, describe the algorithm in \cref{sec:Inequality_constrained_algorithm}, and present the theoretical analysis in \cref{sec:Inquality_constrained_analysis}.
In \cref{sec:Numerical_experiments}, we illustrate the empirical performance of the proposed algorithms on multi-class classification logistic regression and the S2MPJ CUTEst problem set, and compare with methods from the literature.
Finally, we conclude with final remarks in \cref{sec:Conclusion}.
\section{Equality Constrained Problems} \label{sec:Equality_constrained}

In this section, we consider problem \eqref{eq:intro_problem} with only equality constraints, i.e.,
\begin{equation}
    \min_{x \in \Rmbb^n} f(x) \quad s.t. \quad c(x) = 0, \label{eq:equality_prob}
\end{equation}
where the constraint function $c_E$ is denoted by $c : \Rmbb^n \rightarrow \Rmbb^m$. 
We describe the proposed \texttt{RA} based framework in \cref{sec:General_Deterministic_Solver_alg} for general deterministic solvers and provide the corresponding convergence and complexity results in \cref{sec:General_Deterministic_Solver_theory}. We then focus on using \texttt{SQP} methods \cite{byrd2008inexact,berahas2021sequential,berahas2022adaptive} in the inner loop in \cref{sec:SQP_based_algorithm_eq} with the corresponding theoretical analysis in \cref{sec:SQP_based_algorithm_eq_theory}.

We introduce some additional notation for this section. The Lagrangian of problem \eqref{eq:equality_prob} is $\Lcal(x, \lambda) = f(x) + \lambda^T c(x)$ where $\lambda \in \Rmbb^m$ is the vector of dual variables. The
KKT conditions \eqref{eq:KKT_conditions} in the equality constraint setting reduce to
\begin{equation*}
    T(x, \lambda) = \begin{bmatrix}
        \nabla_x \Lcal(x, \lambda) \\ c(x)
    \end{bmatrix} = \begin{bmatrix}
        \nabla f(x) + \nabla c(x) \lambda \\ c(x)
    \end{bmatrix} = 0,
\end{equation*}
where $x^*$ is a first-order stationary point if there exists  $
\lambda^* \in \Rmbb^m$ such that $T(x^*, \lambda^*) = 0$. The corresponding quantities for subsampled problem \eqref{eq:subsampled_problem} with sample set $S$ are denoted as
\begin{equation*}
    \Lcal_{S}(x, \lambda) = F_{S}(x) + \lambda^T c(x)
    \quad \text{and} \quad
    T_{S}(x, \lambda) = \begin{bmatrix}
        \nabla_x \Lcal_{S}(x, \lambda) \\ c(x)
    \end{bmatrix}.
\end{equation*}
We denote the iterate (primal variable) and the dual variable at outer iteration $k$ and inner iteration $j$ as $x_{k, j}$ and $\lambda_{k,j}$, respectively. For convenience, we use the shorthand notation $c_{k, j} = c(x_{k, j})$ and $J_{k, j} = \nabla c(x_{k, j})^T$.

\subsection{Retrospective Approximation General Framework} \label{sec:General_Deterministic_Solver_alg}
In this section, we describe the proposed \texttt{RA} based framework with nested inner and outer loops (\cref{alg:RA_base}) to solve stochastic problems with deterministic equality constraints. In outer iteration $k$, we obtain a batch of samples $S_k$, independent of the current iterate, and define the subsampled problem \eqref{eq:subsampled_problem}. The subsampled problem is then solved up to a specified accuracy in the inner loop. For now, we assume that the batch size sequence is pre-specified and later introduce an adaptive strategy to select the batch size. 

The primal iterate in the inner loop is initialized as the output of the previous outer iteration, i.e., $x_{k+1, 0} = x_{k, N_k}$ where $N_k$ is the number of inner iterations at outer iteration $k$.
However, the dual variable iterate is initialized as
\begin{equation} \label{eq:dual_update_eq}
    \lambda_{k, 0} = 
    \begin{cases}
        \lambda_{k-1, N_{k-1}} & \text{if} \quad \texttt{Carry-Over,} \\
        -(J_{k, 0} J_{k, 0}^T)^{-1} J_{k, 0} g_{S_{k}}(x_{k, 0}) & \text{if} \quad \texttt{Reinitialize}. 
    \end{cases}
\end{equation}
The \texttt{Carry-Over} option reuses the dual variable from the previous outer iteration, requiring no additional computation. 
In contrast, the \texttt{Reinitialize} option solves the least squares problem $\lambda_{k, 0} = \argmin_{\lambda \in \Rmbb^m}\|T_{S_{k}}(x_{k, 0}, \lambda)\|^2$, which minimizes the KKT error for the subsampled problem at $x_{k, 0}$ to obtain a dual variable estimate for the inner loop.
The two options represent extreme choices for initialization. 
While the \texttt{Reinitialize} option provides a better dual variable for the subsampled problem, improving the initialization for the inner loop, it also incurs additional computational cost. Therefore, the choice between the two should be made based on the specific problem requirements. 
Alternatively, an inexact solution to the least squares problem is acceptable if it performs at least as well as the \texttt{Carry-Over} option, i.e., $\|T_{S_{k}}(x_{k, 0}, \lambda_{k, 0})\| \leq \|T_{S_{k}}(x_{k, 0}, \lambda_{k-1, N_{k-1}})\|$. 

In outer iteration $k$, the inner loop terminates when the subsampled problem is solved to a specified accuracy, as determined by termination criterion $\{\Tcal_k\}$ in \cref{alg:RA_base} (Line \ref{RA_base:termination}). For the case of equality constraints, it is defined as 
\begin{align} \label{eq:termination_criterion_eq}
    \Tcal_k : \|T_{S_k}(x_{k, j}, \lambda_{k, j}) \| \leq \gamma_k \|T_{S_k}(x_{k, 0}, \lambda_{k, 0})\| + \epsilon_k,
\end{align}
where $\gamma_k \in [0, 1)$ and $\epsilon_k \geq 0$ are user-defined parameters. This criterion measures solution accuracy in terms of the KKT error, and ensures a reduction in the KKT error of the subsampled problem relative to its initial value while allowing for an additional tolerance $\epsilon_k$. 
When the initial KKT error is high, the relative nature of \eqref{eq:termination_criterion_eq} prevents large deviations from the current solution.
A high initial KKT error may indicate that the current iterate and dual variable are either poor solutions to the true problem, or they are good solutions but the chosen sample set is dominated by outliers. 
In either case, one would like to proceed with caution as in the former, subsequent outer iterations should reaffirm the scenario, leading to a better solution, while in the latter, one would avoid biasing the solution to outlier samples.
Conversely, the term $\epsilon_k$ prevents unnecessary precision when the initial KKT error is low, which can occur when the solutions to successive subsampled problems are close to each other. 
In such cases, solving the subsampled problem to high accuracy may bias the iterates towards the solution of the subsampled problem rather than the solution of the true problem.

Any deterministic solver that updates the iterate and dual variable to satisfy $\Tcal_k$ within a finite number of iterations can be used in the inner loop.
The complete framework is presented in \cref{alg:Equality_Constrained_RA} and the remark that follows.

\begin{framework}[H]
    \textbf{Inputs:} Sequences $\{|S_k|\}$, $\{\gamma_k\} < 1$ and $\{\epsilon_k\} \geq 0$; initial iterate $x_{0,0}$ and dual variable $\lambda_{-1, 0}$.
    \caption{Equality Constrained Retrospective Approximation}
    \begin{algorithmic}[1]
        \For{$k = 0, 1, 2, \dots$}
            \State Construct subsampled problem \eqref{eq:subsampled_problem} from sample set $S_k$ and set $\lambda_{k, 0}$ via \eqref{eq:dual_update_eq} \label{EQ_general:subsample}
            \For{$j = 0, 1, 2, \dots$}
                \State \textbf{if} \eqref{eq:termination_criterion_eq} is satisfied; $N_k = j$, \textbf{break} \label{EQ_general:termination}
                \State Update $x_{k, j + 1}$ and $\lambda_{k, j + 1}$ \label{EQ_general:update}
            \EndFor
            \State Set $x_{k+1, 0} = x_{k, N_k}$
        \EndFor
    \end{algorithmic}
    \label{alg:Equality_Constrained_RA}
\end{framework}

\begin{remark}
    We make the following remarks about \cref{alg:Equality_Constrained_RA}, building upon \cref{alg:RA_base}.
    \begin{enumerate}
        \item \textbf{Outer Loop (Line \ref{EQ_general:subsample}):}  We obtain a sample set $
        S_k$ and use it to form the subsampled problem \eqref{eq:subsampled_problem}. The batch size is a user-specified sequence and an adaptive strategy is described later. Next, we initialize the dual variable according to \eqref{eq:dual_update_eq}. The update can follow either \texttt{Carry-Over} or \texttt{Reinitialize} options, or any dual variable estimate that obtains a KKT error less than $\|T_{S_k}(x_{k, 0}, \lambda_{k-1, N_{k-1}})\|$ for the subsampled problem at $x_{k, 0}$. To ensure the algorithm is well defined at $k=0$, we define $N_{-1} = 0$.    
        \item \textbf{Inner Loop (Lines \ref{EQ_general:termination}-\ref{EQ_general:update}):} We update the iterate and dual variable until the termination criterion \eqref{eq:termination_criterion_eq} is satisfied. Any deterministic solver that can achieve this in a finite number of steps can be used in the inner loop. 
    \end{enumerate}
\end{remark}

In \cref{alg:Equality_Constrained_RA}, one should choose a deterministic solver suitable for the subsampled problems. 
While \texttt{SQP} methods work well for general equality constraints, one might prefer projection-based methods when the constraints are simple and projections are inexpensive.
If the selected algorithm does not update the dual variable, a least squares problem, similar to \texttt{Reinitialize} in \eqref{eq:dual_update_eq}, can be solved within the inner loop to update it. 
Alternatively, the termination criterion $\Tcal_k$ may use solver-specific optimality measures instead of the KKT error. We discuss such criterion when discussing the \texttt{SQP} method in \cref{sec:SQP_based_algorithm_eq}.

\subsection{General RA Framework Convergence Analysis} \label{sec:General_Deterministic_Solver_theory}
In this section, we present theoretical results for \cref{alg:Equality_Constrained_RA}. We first establish the basic conditions required for convergence with any deterministic solver. 
Then, we propose an adaptive sampling strategy for selecting the batch size and provide complexity results in terms of the number of inner loop iterations and gradient evaluations.

To begin, we define $\Fcal_k$ as the $\sigma$-algebra corresponding to the initial conditions, i.e., $\{x_{0, 0}, \lambda_{-1, 0}\}$ and sample sets selected up to outer iteration $k$, i.e., $\{S_0, S_1, \dots, S_{k-1} \}$. This results in the nested structure $\Fcal_0 \subseteq \Fcal_1 \subseteq \dots \subseteq \Fcal_k$, with $\Fcal_0 = \{x_{0, 0}, \lambda_{-1, 0}\}$.
Given a sample set $S_k$, the next outer iterate $x_{k+1, 0} = x_{k, N_k}$ is obtained via a deterministic process starting from $x_{k, 0}$. 
Thus, the iterate $x_{k, 0}$ is specified under $\Fcal_k$ while the sample set $S_k$ is independent of $\Fcal_k$, and we only need to define the filtration over the outer iterations. 
The total expectation is then defined as the expectation conditioned on the initial conditions, i.e., $\Embb[\cdot] = \Embb[\cdot | \Fcal_0]$.
We now state some basic assumptions for optimization of smooth functions over equality constraints.

\begin{assumption} \label{ass:EQ_base assumption}
    Let $\chi \subset \Rmbb^n$ be an open convex set containing the sequence of iterates generated by any run of \cref{alg:Equality_Constrained_RA}. 
    The objective function $f : \Rmbb^n \rightarrow \Rmbb$ is continuously differentiable and bounded below over $\chi$, and its gradient function $\nabla f : \Rmbb^n \rightarrow \Rmbb^n$ is Lipschitz continuous over $\chi$. 
    The constraint function $c : \Rmbb^n \rightarrow \Rmbb^{m}$ is continuously differentiable 
    and the Jacobian function, $ J = \nabla c^T : \Rmbb^n \rightarrow \Rmbb^{m \times n}$ is Lipschitz continuous over $\chi$. The Jacobian $J(x)$ over $x \in \chi$ has full row rank, i.e., the linear independence constraint qualification (LICQ) holds. 
\end{assumption}

\begin{assumption} \label{ass:well_posed}
    For any run of \cref{alg:Equality_Constrained_RA},  the inner loop terminates finitely satisfying \eqref{eq:termination_criterion_eq}, i.e., $N_k < \infty$ for all $k \geq 0$.
\end{assumption}

\begin{assumption} \label{ass:Bounded_outer_gradient}
    Let $\{x_{k, N_k}\}$ be the sequence of outer iterates generated by any run of \cref{alg:Equality_Constrained_RA}. There exists $\kappa_g \geq 0$ such that:
    \begin{enumerate}
        \item For the finite-sum problem \eqref{eq:intro_deter_error_obj}: $\|\nabla f(x_{k, N_k})\|^2 \leq \kappa_g^2$ for all $k \geq 0$.
        \item For the expectation problem \eqref{eq:intro_stoch_error_obj}: $\Embb\left[\|\nabla f(x_{k, N_k})\|^2 | \Fcal_k\right] \leq \tilde{\kappa}_g^2$ for all $k \geq 0$.
    \end{enumerate}
\end{assumption}

\cref{ass:EQ_base assumption} is a standard assumption in the equality constrained optimization literature, see e.g., \cite{nocedal2006numerical,byrd2008inexact}, that ensures the existence of a first-order stationary point under the LICQ.
\cref{ass:well_posed} guarantees that the deterministic solver used in the inner loop is well-posed and will terminate finitely.
\cref{ass:Bounded_outer_gradient} imposes bounds on the gradient of the true objective function over the iterates generated by Algorithm \ref{alg:Equality_Constrained_RA}.
This assumption is common in constrained optimization for both deterministic \cite{powell2006fast,nocedal2006numerical,byrd2008inexact} and stochastic \cite{berahas2021sequential,curtis2024worst,o2024two} settings, where the gradient of the true problem is typically assumed to be deterministically bounded over the set of iterates. 
Thus, \cref{ass:Bounded_outer_gradient} do not pose strict restrictions on our analysis.

We now present standard assumptions on the error in the subsampled objective function gradients, similar to \cite{bollapragada2018adaptive, berahas2022adaptive, friedlander2012hybrid}.
\begin{assumption} \label{ass:gradient_errors}
For all $x \in \chi$, the individual component gradients are bounded relative to the gradient of the true objective functions: 
\begin{enumerate}
    \item For the finite-sum problem \eqref{eq:intro_deter_error_obj}: There exist constants $\omega_1, \omega_2 \geq 0$ such that,
    \begin{equation}
        \|\nabla F(x, \xi)\|^2 \leq \omega_1^2 \|\nabla f(x)\|^2 + \omega_2^2 \quad \text{for all} \,\, \xi \in \Scal.
    \end{equation}
    \item For the expectation problem \eqref{eq:intro_stoch_error_obj}: There exist constants $\Tilde{\omega}_1, \Tilde{\omega}_2 \geq 0$ such that,
    \begin{equation} 
        \E \left[\|\nabla f(x) - \nabla F(x, \xi)\|^2 | x \right] \leq \Tilde{\omega}_1^2 \|\nabla f(x)\|^2 + \Tilde{\omega}_2^2.
    \end{equation}    
\end{enumerate}
\end{assumption}
For the finite-sum problem \eqref{eq:intro_deter_error_obj}, under \cref{ass:gradient_errors} by 
\cite[Section 3.1]{friedlander2012hybrid}, the error in the subsampled gradient over a sample set $S \subseteq \Scal$ can be bounded as,
\begin{equation} \label{eq:deter_sampled_gradient_error}
    \|\nabla f(x) - g_{S}(x)\|^2 \leq 4 \left(1 - \tfrac{|S|}{|\Scal|}\right)^2 (\omega_1^2 \|\nabla f(x)\|^2 + \omega_2^2) \quad \text{for all} \,\, x \in \chi.
\end{equation}
For the expectation problem \eqref{eq:intro_stoch_error_obj}, the error in the subsampled gradient over a set of i.i.d. samples $S$ from the distribution $\Pcal$, independent of $x$, can be bounded as,
\begin{equation} \label{eq:stoch_sampled_gradient_error}
    \Embb[\|\nabla f(x) - g_{S}(x)\|^2 | x] = \tfrac{1}{|S|^2} \sum_{\xi \in S} \Embb[\|\nabla f(x) - \nabla F(x, \xi)\|^2 | x]  \leq \tfrac{\Tilde{\omega}_1^2 \|\nabla f(x)\|^2 + \Tilde{\omega}_2^2}{|S|}.
\end{equation} 
We define an additional gradient error metric for the expectation problem \eqref{eq:intro_stoch_error_obj} over a set of samples $S$, given by
\begin{equation} \label{eq:G_metric_def}
    G_{S} = \max_{x \in \chi} \tfrac{\|\nabla f(x) - g_{S}(x)\|}{\epsilon_G + \|\nabla f(x)\|},
\end{equation}
where $\epsilon_G > 0$, 
which represents the maximum error in the subsampled gradient relative to the true gradient, as defined in \cite{newton2023retrospective,newtonretrospective} for analyzing the \texttt{RA} framework for unconstrained optimization. 

To analyze \cref{alg:Equality_Constrained_RA}, we first establish a recursive bound on the KKT error of the true problem across outer iterations.

\begin{lemma} \label{lem:main_error_bound}
Suppose Assumptions \ref{ass:EQ_base assumption} and \ref{ass:well_posed} hold.
Then, for all $k \geq 0$, the outer iterates generated by \cref{alg:Equality_Constrained_RA} satisfy
\begin{align*}
    \|T(x_{k, N_k}, \lambda_{k, N_k})\|
    &\leq  \gamma_k \|T(x_{k-1, N_{k-1}}, \lambda_{k-1, N_{k-1}})\| + \epsilon_k \numberthis \label{eq:Eq_error_recursion_1} \\
    & \quad + \left\|\nabla f(x_{k, N_k}) - g_{S_k}(x_{k, N_k})\right\| + \gamma_k \|\nabla f(x_{k, 0}) - g_{S_k}(x_{k, 0})\|.
\end{align*}
For the expectation problem \eqref{eq:intro_stoch_error_obj}, for all $k \geq 0$
\begin{align*}
    \Embb\left[\|T(x_{k, N_k}, \lambda_{k, N_k})\| | \Fcal_k\right]
    &\leq \gamma_k \|T(x_{k-1, N_{k-1}}, \lambda_{k-1, N_{k-1}})\| + \Embb[\epsilon_k | \Fcal_k] \numberthis \label{eq:Eq_error_recursion_1_expectation} \\
    & \quad + \Embb[\left\|\nabla f(x_{k, N_k}) - g_{S_k}(x_{k, N_k})\right\| | \Fcal_k]\\
    & \quad + \gamma_k \Embb[\|\nabla f(x_{k, 0}) - g_{S_k}(x_{k, 0})\| | \Fcal_k].
\end{align*}
\end{lemma}
\begin{proof}
The KKT error for the true problem at the end of outer iteration $k \geq 0$ can be bounded as,
\begin{align*}
    \|T(x_{k, N_k}, \lambda_{k, N_k})\|
    &\leq \|T_{S_k}(x_{k, N_{k}}, \lambda_{k, N_{k}})\| + \|T(x_{k, N_{k}}, \lambda_{k, N_{k}}) - T_{S_k}(x_{k, N_{k}}, \lambda_{k, N_{k}})\|\\
    &= \|T_{S_k}(x_{k, N_k}, \lambda_{k, N_k})\| + \left\|\nabla f (x_{k, N_k}) - g_{S_k}(x_{k, N_k})\right\|\\
    &\leq  \gamma_k \|T_{S_k}(x_{k, 0}, \lambda_{k, 0})\| + \epsilon_k + \left\|\nabla f(x_{k, N_k}) - g_{S_k}(x_{k, N_k})\right\|\\ 
    &\leq \gamma_k \|T_{S_k}(x_{k-1, N_{k-1}}, \lambda_{k-1, N_{k-1}})\| + \epsilon_k + \left\|\nabla f(x_{k, N_k}) - g_{S_k}(x_{k, N_k})\right\|\\ 
    &\leq \gamma_k \|T(x_{k-1, N_{k-1}}, \lambda_{k-1, N_{k-1}})\| + \epsilon_k + \left\|\nabla f(x_{k, N_k}) - g_{S_k}(x_{k, N_k})\right\| \\
    &\quad + \gamma_k \|T(x_{k-1, N_{k-1}}, \lambda_{k-1, N_{k-1}}) - T_{S_k}(x_{k-1, N_{k-1}}, \lambda_{k-1, N_{k-1}})\| \\
    &= \gamma_k \|T(x_{k-1, N_{k-1}}, \lambda_{k-1, N_{k-1}})\| + \epsilon_k \\
    &\quad + \left\|\nabla f(x_{k, N_k}) - g_{S_k}(x_{k, N_k})\right\| + \gamma_k \|\nabla f(x_{k, 0}) - g_{S_k}(x_{k, 0})\| ,  
\end{align*}
where the second inequality follows from \eqref{eq:termination_criterion_eq}, 
the third inequality follows from $x_{k-1, N_{k-1}} = x_{k, 0}$ and the update rule \eqref{eq:dual_update_eq} for $\lambda_{k, 0}$, thus completing the proof for \eqref{eq:Eq_error_recursion_1}. Taking the conditional expectation of \eqref{eq:Eq_error_recursion_1} given $\Fcal_k$ yields \eqref{eq:Eq_error_recursion_1_expectation}, since $(x_{k-1, N_{k-1}}, \lambda_{k-1, N_{k-1}})$ are known under $\Fcal_k$.
\end{proof}

We now establish the convergence of \cref{alg:Equality_Constrained_RA}.
\begin{theorem} \label{th:Eq_convergence}
    Suppose Assumptions \ref{ass:EQ_base assumption}, \ref{ass:well_posed}, \ref{ass:Bounded_outer_gradient} and \ref{ass:gradient_errors} hold for any run of \cref{alg:Equality_Constrained_RA}.
    \begin{enumerate}
        \item For the finite-sum problem \eqref{eq:intro_deter_error_obj}: 
        If the sample set sequence is chosen such that $\{S_k\} \rightarrow \Scal$ and the termination criterion parameters as $0 \leq \{\gamma_k\} \leq \gamma < 1$ and $\epsilon_k \rightarrow 0$,
        then, $\|T(x_{k, N_k}, \lambda_{k, N_k})\| \rightarrow 0$ and $\{x_{k, N_k}\}$ converges to a first-order stationary point of \eqref{eq:intro_problem}.
        \item For the expectation problem \eqref{eq:intro_stoch_error_obj}: 
        If the batch size sequence is chosen such that $\{|S_k|\} \rightarrow \infty$, the termination criterion parameters as $0 \leq \{\gamma_k\} \leq \tilde{\gamma} < 1$ and $\Embb[\epsilon_k | \Fcal_k] \rightarrow 0$, and given that $\Embb[G_{S_k}^2 | \Fcal_k] \rightarrow 0$ as $\{|S_k|\} \rightarrow \infty$, then, $\Embb[\|T(x_{k, N_k}, \lambda_{k, N_k})\|] \rightarrow 0$ and $\{x_{k, N_k}\}$ converges to a first-order stationary point of \eqref{eq:intro_problem} in expectation.
    \end{enumerate}
\end{theorem}
\begin{proof}
    For the finite-sum problem \eqref{eq:intro_deter_error_obj}, unrolling \eqref{eq:Eq_error_recursion_1} from \cref{lem:main_error_bound} yields for all $k \geq 0$,
    \begin{align*}
        \|T(x_{k, N_k}, \lambda_{k, N_k})\|
        & \leq \left\{\prod_{h = 0}^{k} \gamma_h \right\}\|T(x_{0, 0}, \lambda_{-1, 0})\|  + \epsilon_k  \numberthis \label{eq:eq_recurion_finite}\\
        &\quad + \|\nabla f(x_{k, N_k}) - g_{S_k}(x_{k, N_k})\|+ \gamma_k \|\nabla f(x_{k, 0}) - g_{S_k}(x_{k, 0})\|\\
        &\quad + \sum_{i=0}^{k-1} \left\{\prod_{h = i+1}^{k} \gamma_h \right\} \left[\epsilon_i + \|\nabla f(x_{i, N_i}) - g_{S_i}(x_{i, N_i})\| \right] \\
        &\quad + \sum_{i=0}^{k-1} \left\{\prod_{h = i+1}^{k} \gamma_h \right\} \gamma_i \|\nabla f(x_{i, 0}) - g_{S_i}(x_{i, 0})\| \\
        &\leq \left\{\prod_{h = 0}^{k} \gamma_h \right\}\|T(x_{0, 0}, \lambda_{-1, 0})\|  + \epsilon_k \\
        &\quad + 2 \left(\tfrac{|\Scal| - |S_k|}{|\Scal|}\right) \left(\omega_1 (\|\nabla f(x_{k, N_k})\| +\gamma_k \|\nabla f(x_{k, 0})\|) + \omega_2 (1 + \gamma_k)
        \right)  \\
        &\quad + \sum_{i=0}^{k-1} \left\{\prod_{h = i+1}^{k} \gamma_h \right\} \left[\epsilon_i + \left(\tfrac{|\Scal| - |S_i|}{|\Scal|}\right)\left(\omega_1 \|\nabla f(x_{i, N_i})\| + \omega_2 \right) \right]  \\
        &\quad + 2 \sum_{i=0}^{k-1} \left\{\prod_{h = i+1}^{k} \gamma_h \right\} \gamma_i \left(\tfrac{|\Scal| - |S_i|}{|\Scal|}\right) \left[\omega_1 \|\nabla f(x_{i, 0})\| + \omega_2 \right]\\
        &\leq \left\{\prod_{h = 0}^{k} \gamma_h \right\}\|T(x_{0, 0}, \lambda_{-1, 0})\|  + \epsilon_k  + 2 \left(\tfrac{|\Scal| - |S_k|}{|\Scal|}\right) (1 + \gamma_k) \left(\omega_1 \kappa_g + \omega_2 \right) \\
        &\quad + \sum_{i=0}^{k-1} \left\{\prod_{h = i+1}^{k} \gamma_h \right\} \left[\epsilon_i 
        + 2 \left(\tfrac{|\Scal| - |S_i|}{|\Scal|}\right) (1 + \gamma_i)\left(\omega_1 \kappa_g  + \omega_2 
        \right)
        \right] ,
    \end{align*}
    where the second inequality follows from \eqref{eq:deter_sampled_gradient_error} and third inequality follows by \cref{ass:Bounded_outer_gradient}. 
    The first term in the above inequality converges to zero as $\left\{\prod_{h = 0}^{k} \gamma_h \right\} \leq \gamma^{k+1}$ where $\gamma < 1$. The second and third terms converge to zero under the specifications $\epsilon_k \rightarrow 0$ and $\{|S_k|\} \rightarrow |\Scal|$, respectively. The final term is a sum of the form $\sum_{i=0}^{k-1}a_i \left\{\prod_{h = i+1}^{k} \gamma_h \right\} \leq \sum_{i=0}^{k-1}a_i \gamma^{k - i}$ which converges to zero by \cref{lem:series_for_convergence}, with $b_i = a_i = \epsilon_i + 2 \left(\tfrac{|\Scal| - |S_i|}{|\Scal|}\right) (1 + \gamma_i)\left(\omega_1 \kappa_g  + \omega_2 \right) \rightarrow 0$.
    Thus, all terms in the KKT error bound converge to zero completing the proof.

    For the expectation problem \eqref{eq:intro_stoch_error_obj}, taking the total expectation of \eqref{eq:eq_recurion_finite} and applying the tower property, i.e., $\Embb[\epsilon_k] = \Embb[\Embb[\epsilon_k | \Fcal_k]]$, for all relevant quantities yields for all $k \geq 0$,
    \begin{align*}
        \Embb[\|T(x_{k, N_k}, \lambda_{k, N_k})\|]
        &\leq \left\{\prod_{h = 0}^{k} \gamma_h \right\}\|T(x_{0, 0}, \lambda_{-1, 0})\|  + \Embb[\Embb[\epsilon_k | \Fcal_k]] \numberthis \label{eq:eq_recursion_expec}\\
        &\quad + \Embb\left[\Embb[\|\nabla f(x_{k, N_k}) - g_{S_k}(x_{k, N_k})\|| \Fcal_k]\right]\\
        &\quad + \gamma_k \Embb\left[\Embb[\|\nabla f(x_{k, 0}) - g_{S_k}(x_{k, 0})\|| \Fcal_k]\right]\\
        &\quad + \sum_{i=0}^{k-1} \left\{\prod_{h = i+1}^{k} \gamma_h \right\} \Embb\left[\Embb[\epsilon_i| \Fcal_i] + \Embb[\|\nabla f(x_{i, N_i}) - g_{S_i}(x_{i, N_i})\|| \Fcal_i]\right]  \\
        &\quad + \sum_{i=0}^{k-1} \left\{\prod_{h = i+1}^{k} \gamma_h \right\} \gamma_i \Embb\left[\Embb[\|\nabla f(x_{i, 0}) - g_{S_i}(x_{i, 0})\|| \Fcal_i] \right] .
    \end{align*}
    The result contains two types of errors for the subsampled gradients.
    The first is $\Embb[\|\nabla f(x_{i, 0}) - g_{S_i}(x_{i, 0})\|| \Fcal_i]$, where $S_i$ is independent of $\Fcal_i$ (and thus $x_{i, 0}$), and \eqref{eq:stoch_sampled_gradient_error} holds. 
    The second is $\Embb[\|\nabla f(x_{i, N_i}) - g_{S_i}(x_{i, N_i})\|| \Fcal_i]$ where $x_{i, N_i}$ is dependent on $S_i$. 
    Here, we utilize the metric defined in \eqref{eq:G_metric_def} as,
    \begin{align*}
    \Embb[\|\nabla f(x_{i, N_i}) - g_{S_i}(x_{i, N_i})\|| \Fcal_i] &\leq \Embb[G_{S_i}(\epsilon_G + \|\nabla f(x_{i, N_i})\|)| \Fcal_i] \\
    &= \epsilon_G\Embb[G_{S_i}|\Fcal_i] + \Embb[G_{S_i} \|\nabla f(x_{i, N_i})\||\Fcal_i]\\
    &\leq \epsilon_G\Embb[G_{S_i}|\Fcal_i] + \sqrt{\Embb[G_{S_i}^2 |\Fcal_i]\Embb[\|\nabla f(x_{i, N_i})\|^2|\Fcal_i]}, \numberthis \label{eq:biased_gradient_error_bound}
    \end{align*} 
    where the last inequality follows from the expectation version of the Cauchy–Schwarz inequality. 
    Substituting \eqref{eq:biased_gradient_error_bound} and \eqref{eq:stoch_sampled_gradient_error} into \eqref{eq:eq_recursion_expec}, along with the bound on the gradient from \cref{ass:Bounded_outer_gradient} yields,
    \begin{align*}
        \Embb[\|T(x_{k, N_k}, \lambda_{k, N_k})\|]
        &\leq  \|T(x_{0, 0}, \lambda_{-1, 0})\| \left\{\prod_{h = 0}^{k} \gamma_h \right\} \\
        &\quad + \Embb\left[\Embb[\epsilon_k | \Fcal_k] + \epsilon_G\Embb[G_{S_k}|\Fcal_k] + \tilde{\kappa}_g\sqrt{\Embb[G_{S_k}^2 |\Fcal_k]}+ \gamma_k \tfrac{\Tilde{\omega}_1 \tilde{\kappa}_g + \Tilde{\omega}_2}{\sqrt{|S_k|}}\right] \\
        &\quad + \sum_{i=0}^{k-1} \left\{\prod_{h = i+1}^{k} \gamma_h \right\} \Embb\left[\Embb[\epsilon_i| \Fcal_i] + \epsilon_G\Embb[G_{S_i}|\Fcal_i] + \tilde{\kappa}_g\sqrt{\Embb[G_{S_i}^2 |\Fcal_i]}\right] \\
        &\quad + \sum_{i=0}^{k-1} \left\{\prod_{h = i+1}^{k} \gamma_h \right\} \gamma_i \Embb\left[ \tfrac{\Tilde{\omega}_1 \tilde{\kappa}_g + \Tilde{\omega}_2}{\sqrt{|S_i|}} \right].
    \end{align*}
    In the above bound, the first term is the same as the first term for the finite-sum problem (above). Moreover, all terms within the second term expectation converge to zero under the specifications $\Embb[\epsilon_k | \Fcal_k] \rightarrow 0$ and $\{|S_k|\} \rightarrow \infty$, and the assumption $\Embb[G_{S_k}^2  | \Fcal_k] \rightarrow 0$. The last two terms together are a series of the form $\sum_{i=0}^{k-1}\Embb[a_i] \left\{\prod_{h = i+1}^{k} \gamma_h \right\} \leq \sum_{i=0}^{k-1} \Embb[a_i] \tilde{\gamma}^{k - i}$ which converges to zero by \cref{lem:series_for_convergence}, with  $b_i = \Embb[a_i] \rightarrow 0$ as $a_i \rightarrow 0$, thus completing the proof.
\end{proof}

\cref{th:Eq_convergence} establishes the convergence of \cref{alg:Equality_Constrained_RA} for stochastic problems with deterministic equality constraints, using any deterministic solver in the inner loop. 
While most of the requirements are standard assumptions, for the expectation problem \eqref{eq:intro_stoch_error_obj}, an additional assumption that $\Embb[G_{S_k}^2 | \Fcal_k] \rightarrow 0$ is needed. 
This ensures that the maximum deviation of the subsampled gradient from the true problem gradient decreases with the sample size when normalized by the true problem gradient, thereby enforcing the uniform convergence of the error in subsampled gradients. This condition is not strict and is satisfied by many error classes due to the relative nature of the error, as shown in \cite{newton2023retrospective,newtonretrospective} for unconstrained stochastic optimization.

We now introduce a strategy to adaptively select the batch size $|S_k|$ in \cref{alg:Equality_Constrained_RA}. This strategy is a modification of the well-known norm test \cite{bollapragada2018adaptive,o2024fast,byrd2012sample,carter1991global}.
\begin{condition} \label{cond:eq_norm_test}
    At the beginning of outer iteration $k \geq 0$ in \cref{alg:Equality_Constrained_RA}, $|S_k|$ is selected such that:
    \begin{enumerate}
        \item For the finite-sum problem \eqref{eq:intro_deter_error_obj}: With constants $\theta, \hat{\theta}, a \geq 0$ and $\beta \in (0, 1)$, 
        \begin{align*}
            \|\nabla f(x_{k, 0}) - g_{S_k}(x_{k, 0})\|^2 &\leq \theta^2 \|T(x_{k, 0}, \lambda_{k-1, N_{k-1}})\|^2 + a^2 \beta^{2k},\\
            \|\nabla f(x_{k, N_k}) - g_{S_k}(x_{k, N_k})\|^2 &\leq \hat{\theta}^2 \left(\theta^2 \|T(x_{k, 0}, \lambda_{k-1, N_{k-1}})\|^2 + a^2 \beta^{2k}\right).
        \end{align*}
        \item For the expectation problem \eqref{eq:intro_stoch_error_obj}: With constants $\tilde{\theta}, \tilde{a} \geq 0$ and $\tilde{\beta} \in (0, 1)$, 
        \begin{align*}
            \Embb\left[\|\nabla f(x_{k, 0}) - g_{S_k}(x_{k, 0})\|^2 | \Fcal_k\right] \leq \tilde{\theta}^2 \|T(x_{k, 0}, \lambda_{k-1, N_{k-1}})\|^2 + \tilde{a}^2 \tilde{\beta}^{2k}.
        \end{align*}
    \end{enumerate}
\end{condition}
\begin{remark} \label{remark:norm_test_conditions}
    We make the following remarks regarding \cref{cond:eq_norm_test}.
    \begin{enumerate}
        \item \textbf{Batch size:} Under \cref{ass:gradient_errors}, \cref{cond:eq_norm_test} can always be satisfied with a sufficiently large batch size. If $\theta \,\, (\text{or } \tilde{\theta}) = 0$, the condition is equivalent to increasing the batch size at a geometric rate. When $\theta \,\, (\text{or } \tilde{\theta}) > 0$, the condition takes the KKT error in the true problem at the current iterate and dual variable into account when selecting the batch size. We later provide sequences of sufficiently large batch sizes that satisfy \cref{cond:eq_norm_test}.
        
        \item \textbf{True problem estimates:} 
        The right-hand side of the inequalities in \cref{cond:eq_norm_test}
        require access to true problem quantities which cannot be computed.
        However, we utilize these conditions to understand the permissible errors in the algorithm that guarantee convergence and develop a practical strategy using sampled estimates, detailed in \cref{sec:Numerical_experiments}.
        The condition can also be utilized with any other measure of optimality for the problem instead of the KKT error, as will be discussed in the context of solver-specific termination criteria in \cref{sec:SQP_based_algorithm_eq_theory}.
        
        \item \textbf{Finite-sum condition:} For the finite-sum problem \eqref{eq:intro_deter_error_obj}, \cref{cond:eq_norm_test} bounds the gradient error at the start and end of the inner loop. While this is a stronger condition than usually considered for the norm test, due to the finite-sum nature of the problem, if one controls the gradient error at a particular iterate, one can ensure the gradient error is bounded over the space of iterates. Thus, with a large enough constant $\hat{\theta}$, the condition can be satisfied by selecting a batch size with small enough gradient error at the start of the inner loop.
    \end{enumerate}
\end{remark}
We proceed to analyze \cref{alg:Equality_Constrained_RA} with batch sizes selected to satisfy \cref{cond:eq_norm_test}.
First, we introduce additional assumptions to characterize the variance of the expectation problem \eqref{eq:intro_stoch_error_obj}.  
\begin{assumption} \label{ass:Variance_regularity}
    We make the following assumptions regarding the expectation problem \eqref{eq:intro_stoch_error_obj}.
    \begin{enumerate}
        \item CLT Scaling: The gradient error metric \eqref{eq:G_metric_def} follows CLT-scaling, i.e., there exists $\kappa_{G} \geq 0$ such that for any set of i.i.d. samples $S$ from $\Pcal$, $\Embb\left[G_{S}^2\right] \leq \tfrac{\kappa_{G}^2}{|S|}$.
        \item Variance Lower Bound: The variance of the subsampled gradients is lower bounded, i.e., there exists $\kappa_{\sigma} > 0$ such that $\Var\left(\nabla F(x_{k, 0}, \xi) | \Fcal_k\right) \geq \kappa_G^2 \kappa_{\sigma}^2$ for all $k \geq 0$.
    \end{enumerate}
\end{assumption}
In \cref{ass:Variance_regularity}, the CLT scaling assumption states that the maximum error in the subsampled gradient, when normalized by the true gradient norm, scales inversely with the sample size. This assumption, also used in \cite{newton2023retrospective,newtonretrospective} for unconstrained stochastic optimization, parallels conditions in stochastic gradient methods but applies over the entire space of iterates rather than at each individual iterate, as required in the RA framework. 

The variance lower bound in \cref{ass:Variance_regularity} is counterintuitive to the conditions prescribed in most works for stochastic optimization where the variance is assumed to be upper bounded \cite{berahas2022adaptive,curtis2023sequential,o2024fast}. 
Such an assumption becomes essential in our work to ensure that the adaptive sampling strategy of satisfying \cref{cond:eq_norm_test} results in an increasing batch size sequence, as opposed to using a deterministically prescribed sequence in \cite{newtonretrospective,curtis2023sequential,royset2006optimal,newton2023retrospective}.
While the assumption does not pose any reasonable barriers in practice, it is also not required when one uses a prescribed sequence for the batch sizes, as will also be demonstrated.
When $\kappa_G = 0$, the lower bound does not exist but increasing the batch size is unnecessary since the deviation of the subsampled gradient from the true problem gradient is zero from CLT scaling. 
When $\kappa_G > 0$, the assumption can be relaxed to requiring non-zero variance infinitely often. 
This ensures that the batch size increases across multiple outer iterations when \cref{cond:eq_norm_test} is satisfied for the expectation problem. 
If the variance is zero at outer iterates infinitely often, increasing the batch size is unnecessary, as the error in the subsampled gradient is zero infinitely often. 
Thus, the variance lower bound assumption excludes such degenerate cases from the theoretical analysis.
In practice, since $|S_k| \geq |S_{k-1}|$ is enforced, the batch size remains unchanged when zero variance is encountered. As better solutions are found, the batch size eventually increases, preventing any violations of this assumption.

We now establish the complexity guarantees for \cref{alg:Equality_Constrained_RA} to obtain an $\epsilon > 0$ stationary point, defined as $\|T(x_{k, N_k}, \lambda_{k, N_k})\| \leq \epsilon$ for the finite-sum problem \eqref{eq:intro_deter_error_obj} and in expectation as $\Embb[\|T(x_{k, N_k}, \lambda_{k, N_k})\|] \leq \epsilon$ for the expectation problem \eqref{eq:intro_stoch_error_obj}, with batch sizes selected to satisfy \cref{cond:eq_norm_test}. 
We first look at this complexity across outer iterations.

\begin{theorem} \label{th:EQ_outer_iter_complexity}
    Suppose Assumptions \ref{ass:EQ_base assumption}, \ref{ass:well_posed}, \ref{ass:Bounded_outer_gradient} and \ref{ass:gradient_errors} hold and the batch size sequence $\{|S_k|\}$ is chosen to satisfy \cref{cond:eq_norm_test} in \cref{alg:Equality_Constrained_RA}. 
    \begin{enumerate}
        \item For the finite-sum problem \eqref{eq:intro_deter_error_obj}: For all $k \geq 0$, if the termination criterion parameters are chosen such that $0 \leq \gamma_k \leq \gamma < 1$ and $\epsilon_k = \omega \|\nabla f(x_{k, 0}) - g_{S_k}(x_{k, 0})\| + \hat{\omega}\beta^k$ with $\omega, \hat{\omega} \geq 0$,  
        and \cref{cond:eq_norm_test} parameters are chosen such that $a_1 = \left[  \gamma + \theta(\omega + \hat{\theta} + \gamma)\right] < 1$, then, the KKT error converges at a linear rate across outer iterations, i.e., 
        \begin{align*}
            \|T(x_{k, N_k}, \lambda_{k, N_k})\|
            \leq \max \{a_1 + \nu, \beta\}^{k+1} \max \left\{ \|T(x_{0, 0}, \lambda_{-1, 0})\|, \tfrac{a_2}{\nu}\right\},
        \end{align*}
        where $a_2 =a(\omega + \hat{\theta}  + \gamma) + \hat{\omega}$ and $\nu > 0$ such that $a_1 + \nu < 1$.
        \item For the expectation problem \eqref{eq:intro_stoch_error_obj}: For all $k \geq 0$, if \cref{ass:Variance_regularity} is satisfied, the termination criterion parameters are chosen such that $0 \leq \gamma_k \leq \tilde{\gamma} < 1$ and $\epsilon_k = \tilde{\omega}\sqrt{\tfrac{\Var(\nabla F(x_{k, 0}, \xi) | \Fcal_k)}{|S_k|}}$ where $\tilde{\omega} \geq 0$, and \cref{cond:eq_norm_test} parameters are chosen such that $\tilde{a}_1 = \left[ \tilde{\gamma}  + \tilde{\theta} \left(\tilde{\omega} + \tfrac{(\epsilon_G + \kappa_g)}{\kappa_{\sigma}} + \tilde{\gamma}\right)\right]  < 1$, then, the expected KKT error converges at a linear rate across outer iterations, i.e., 
        \begin{align*}
            \Embb\left[\|T(x_{k, N_k}, \lambda_{k, N_k})\| \right]
            \leq  \max \{\tilde{a}_1 + \tilde{\nu}, \tilde{\beta}\}^{k+1} \max\left\{\|T(x_{0, 0}, \lambda_{-1, 0})\|, \tfrac{\tilde{a}_2}{\tilde{\nu}}\right\},
        \end{align*}
        where $\tilde{a}_2 = \tilde{a} \left(\tilde{\omega} + \tfrac{(\epsilon_G + \kappa_g)}{\kappa_{\sigma}} + \tilde{\gamma} \right)$ and $\tilde{\nu} > 0$ such that $\tilde{a}_1 + \tilde{\nu} < 1$.
    \end{enumerate}
\end{theorem}
\begin{proof}
    For the finite-sum problem \eqref{eq:intro_deter_error_obj}, under the stated conditions, the error bound \eqref{eq:Eq_error_recursion_1} from \cref{lem:main_error_bound} can be further simplified as,
    \begin{align*}
        \|T(x_{k, N_k}, \lambda_{k, N_k})\|
        &\leq  \gamma \|T(x_{k-1, N_{k-1}}, \lambda_{k-1, N_{k-1}})\| + \omega \left\|\nabla f(x_{k, 0}) - g_{S_k}(x_{k, 0})\right\| + \hat{\omega}\beta^k\\
        &\quad+ \left\|\nabla f(x_{k, N_k}) - g_{S_k}(x_{k, N_k})\right\| + \gamma \|\nabla f(x_{k, 0}) - g_{S_k}(x_{k, 0})\|\\
        &\leq \gamma \|T(x_{k-1, N_{k-1}}, \lambda_{k-1, N_{k-1}})\| + \omega \left(\theta \|T(x_{k, 0}, \lambda_{k-1, N_{k-1}})\| + a \beta^{k}\right) + \hat{\omega}\beta^k\\
        &\quad + \hat{\theta}(\theta \|T(x_{k, 0}, \lambda_{k-1, N_{k-1}})\| + a \beta^{k}) + \gamma \left(\theta \|T(x_{k, 0}, \lambda_{k-1, N_{k-1}})\| + a \beta^{k}\right)\\
        &=  a_1 \|T(x_{k-1, N_{k-1}}, \lambda_{k-1, N_{k-1}})\| + a_2 \beta^{k},
    \end{align*}
    where the second inequality follows from \cref{cond:eq_norm_test} and the rest follows from the defined constants with $x_{k, 0} = x_{k-1, N_{k-1}}$.
    Using the above bound, applying \cref{lem:linear_convergence_induction} with $Z_k = \|T(x_{k, N_k}, \lambda_{k, N_k})\|$, $\rho_1 = a_1$, $\rho_2 = \beta$ and $b = a_2$ completes the proof.
    
    For the expectation problem \eqref{eq:intro_stoch_error_obj}, under the stated conditions, the error bound \eqref{eq:Eq_error_recursion_1_expectation} from \cref{lem:main_error_bound} can be further simplified as,
    \begin{align*}
        \Embb\left[\|T(x_{k, N_k}, \lambda_{k, N_k})\| | \Fcal_k \right]
        &\leq \tilde{\gamma} \|T(x_{k-1, N_{k-1}}, \lambda_{k-1, N_{k-1}})\|  + \tilde{\omega}\sqrt{\tfrac{\Var(\nabla F(x_{k, 0}, \xi) | \Fcal_k)}{|S_k|}} \\
        &\quad + \Embb\left[\left\|\nabla f(x_{k, N_k}) - g_{S_k}(x_{k, N_k})\right\| | \Fcal_k \right] \\
        &\quad + \tilde{\gamma} \Embb\left[\|\nabla f(x_{k, 0}) - g_{S_k}(x_{k, 0})\| | \Fcal_k \right]\\
        &\leq \tilde{\gamma} \|T(x_{k-1, N_{k-1}}, \lambda_{k-1, N_{k-1}})\| + \tilde{\omega}\sqrt{\tfrac{\Var( \nabla F(x_{k, 0}, \xi) | \Fcal_k)}{|S_k|}}  \\
        &\quad + \kappa_G\tfrac{(\epsilon_G + \kappa_g)}{\sqrt{|S_k|}} + \tilde{\gamma} \Embb\left[\|\nabla f(x_{k, 0}) - g_{S_k}(x_{k, 0})\| | \Fcal_k \right] \\ 
        &\leq \tilde{\gamma} \|T(x_{k-1, N_{k-1}}, \lambda_{k-1, N_{k-1}})\| \\
        &\quad + \tilde{\omega}\left[\tilde{\theta} \|T(x_{k-1, N_{k-1}}, \lambda_{k-1, N_{k-1}})\| + \tilde{a}\tilde{\beta}^{k}\right] \\
        &\quad + \tfrac{(\epsilon_G + \kappa_g)}{\kappa_{\sigma}} \left[\tilde{\theta} \|T(x_{k-1, N_{k-1}}, \lambda_{k-1, N_{k-1}})\| + \tilde{a}\tilde{\beta}^{k}\right]\\
        &\quad + \tilde{\gamma} \left[\tilde{\theta} \|T(x_{k, 0}, \lambda_{k-1, N_{k-1}})\| + \tilde{a}\tilde{\beta}^{k}\right]\\
        &= \tilde{a}_1 \|T(x_{k-1, N_{k-1}}, \lambda_{k-1, N_{k-1}})\| + \tilde{a}_2\tilde{\beta}^k,
    \end{align*}
    where the second inequality follows from \eqref{eq:biased_gradient_error_bound} and \cref{ass:Variance_regularity}, the third inequality follows by substituting $\kappa_G \leq \tfrac{\sqrt{\Var(\nabla F(x_{k, 0}, \xi) | \Fcal_k)}}{\kappa_{\sigma}}$ and \cref{cond:eq_norm_test} and the last equality substitutes the defined constants with $x_{k, 0} = x_{k-1, N_{k-1}}$. 
    Taking the total expectation of the above bound yields,
    \begin{align*}
        \Embb\left[\|T(x_{k, N_k}, \lambda_{k, N_k})\| \right] &\leq \tilde{a}_1 \Embb[\|T(x_{k-1, N_{k-1}}, \lambda_{k-1, N_{k-1}})\|] + \tilde{a}_2\tilde{\beta}^k,
    \end{align*}
    where applying \cref{lem:linear_convergence_induction} with $Z_k = \Embb\left[\|T(x_{k, N_k}, \lambda_{k, N_k})\| | \Fcal_k \right]$, $\rho_1 = \tilde{a}_1 $, $\rho_2 = \tilde{\beta}$ and $b = \tilde{a}_2$ completes the proof.
\end{proof}

\cref{th:EQ_outer_iter_complexity} establishes linear convergence across outer iterations for \cref{alg:Equality_Constrained_RA} when the batch sizes are chosen to satisfy \cref{cond:eq_norm_test}. 
While the additional tolerance term $\epsilon_k$ in termination criterion \eqref{eq:termination_criterion_eq} is controlled by the variance for the expectation problem \eqref{eq:intro_stoch_error_obj}, an extra relaxation term $\hat{\omega} \beta^k$ is introduced in $\epsilon_k$ for the finite-sum problem \eqref{eq:intro_deter_error_obj}. In the finite-sum setting, \cref{cond:eq_norm_test} evaluates the gradient error only over the selected sample set $S_k$, which may yield a low error at the current iterate but fail to represent the full problem, potentially leading to over-solving the subsampled problem. To avoid this, a geometrically decreasing error term is added to $\epsilon_k$. In contrast, for the expectation problem, \cref{cond:eq_norm_test} considers the expected error over the entire probability space $\Pcal$, yielding a more representative measure of gradient error and eliminating the need for such additional error terms.

The conditions proposed in \cref{th:EQ_outer_iter_complexity} show an interesting trade-off between the parameters of \cref{cond:eq_norm_test} and termination criterion \eqref{eq:termination_criterion_eq}. 
If the subsampled problems are solved to high accuracy by setting small values of $\gamma$, $\omega$ and $\hat{\omega}$ (or $\tilde{\gamma}$ and $\tilde{\omega}$), a larger $\theta$ (or $\tilde{\theta}$) is permitted.
This suggests that achieving high accuracy in solving subproblems reduces the necessity for an aggressive increase in batch size, as the high accuracy inherently introduces significant bias in the solutions.
Conversely, if $\gamma$, $\omega$ and $\hat{\omega}$  (or $\tilde{\gamma}$ and $\tilde{\omega}$) are set to larger values, $\theta$ (or $\tilde{\theta}$) needs to be decreased for an aggressive increase in batch sizes to achieve progress every outer iteration.
Setting $\theta$ (or $\tilde{\theta}$) to zero effectively enforces a geometric increase in batch size under \cref{cond:eq_norm_test}. This eliminates the need for the parameter conditions in \cref{th:EQ_outer_iter_complexity} and the variance lower bound in \cref{ass:Variance_regularity}. The results for this simplified case are presented in the following corollary.   

\begin{corollary} \label{cor:Eq_geometric_batch_increase_outer_iter_complexity}
    Suppose the conditions of \cref{th:EQ_outer_iter_complexity} hold.
    \begin{enumerate}
        \item For the finite-sum problem \eqref{eq:intro_deter_error_obj}: For all $k \geq 0$, if the batch size is chosen as $|S_{k}| =\left\lceil (1 - \beta^k)|\Scal| \right\rceil$ with $\beta \in (0, 1)$ and the termination criterion parameters are chosen as $0 \leq \gamma_k \leq \gamma < 1$ and $\epsilon_k = \omega \left(1 - \tfrac{|S_k|}{|\Scal|}\right)$ with $\omega \geq 0$, then, the KKT error converges at a linear rate across outer iterations as expressed in \cref{th:EQ_outer_iter_complexity} with $a_1 = \gamma$ and $a_2 = \omega + 2(1 + \gamma)(\omega_1 \kappa_g + \omega_2)$.
        \item For the expectation problem \eqref{eq:intro_stoch_error_obj}: For all $k \geq 0$, if \cref{ass:Variance_regularity} 
        is satisfied, the batch size is chosen as $|S_{k+1}| = \left\lceil\tfrac{|S_k|}{\tilde{\beta}^2}\right\rceil$ with $\tilde{\beta} \in (0, 1)$ and the termination criterion parameters are chosen as $0 \leq \{\gamma_k\} \leq \tilde{\gamma} < 1$ and $\epsilon_k = \tfrac{\tilde{\omega}}{\sqrt{|S_k|}}$ where $\tilde{\omega} \geq 0$, then, the expected KKT error converges at a linear rate across outer iterations as expressed in \cref{th:EQ_outer_iter_complexity} with $\tilde{a}_1 = \tilde{\gamma}$ and $\tilde{a}_2 = \tfrac{\tilde{\omega} +\kappa_G(\epsilon_G + \tilde{\kappa}_g) + \tilde{\gamma}(\tilde{\omega}_1 \tilde{\kappa}_g + \tilde{\omega}_2)} {\sqrt{|S_0|}}$, where $|S_0|$ is the initial user-specified batch size.
    \end{enumerate}
\end{corollary}
\begin{proof}
    The proof follows the same procedure as \cref{th:EQ_outer_iter_complexity}, but with more pessimistic bounds on the gradient errors, using the explicitly specified form for the batch size.
    For the finite-sum problem \eqref{eq:intro_deter_error_obj}, under the chosen parameters, the error bound \eqref{eq:Eq_error_recursion_1} from \cref{lem:main_error_bound} can be simplified as,
    \begin{align*}
        \|T(x_{k, N_k}, \lambda_{k, N_k})\|
        &\leq \gamma \|T(x_{k-1, N_{k-1}}, \lambda_{k-1, N_{k-1}})\| + \omega \left(1  - \tfrac{|S_k|}{|\Scal|}\right) \\
        &\quad + \left\|\nabla f(x_{k, N_k}) - g_{S_k}(x_{k, N_k})\right\| + \gamma \|\nabla f(x_{k, 0}) - g_{S_k}(x_{k, 0})\| \\
        &\leq \gamma \|T(x_{k-1, N_{k-1}}, \lambda_{k-1, N_{k-1}})\| + \omega \left(\tfrac{|\Scal| - (1 - \beta^k)|\Scal|}{|\Scal|}\right) \\
        &\quad + 2\left(\tfrac{|\Scal| - (1 - \beta^k)|\Scal|}{|\Scal|}\right)\left[\omega_1 (\|\nabla f(x_{k, N_k})\| + \gamma  \|\nabla f(x_{k, 0})\|) + \omega_2 (1 + \gamma)\right] \\ 
        &\leq \gamma \|T(x_{k-1, N_{k-1}}, \lambda_{k-1, N_{k-1}})\| + (\omega + 2(1 + \gamma)(\omega_1 \kappa_g + \omega_2)) \beta^k \\
        &=  \gamma \|T(x_{k-1, N_{k-1}}, \lambda_{k-1, N_{k-1}})\| + a_2\beta^k  ,
    \end{align*}
    where the second inequality follows from \eqref{eq:deter_sampled_gradient_error} and the chosen batch size sequence, the third inequality follows from \cref{ass:Bounded_outer_gradient} and the final equality substitutes the defined constants. Using the above bound, applying \cref{lem:linear_convergence_induction} with $Z_k = \|T(x_{k, N_k}, \lambda_{k, N_k})\|$, $\rho_1 = \gamma$, $\rho_2 = \beta$ and $b = a_2$ completes the proof.
    
    For the expectation problem \eqref{eq:intro_stoch_error_obj}, under the stated conditions, the error bound \eqref{eq:Eq_error_recursion_1_expectation} from \cref{lem:main_error_bound} can be further simplified as,
    \begin{align*}
        \Embb\left[\|T(x_{k, N_k}, \lambda_{k, N_k})\| | \Fcal_k \right]
        &\leq \tilde{\gamma} \|T(x_{k-1, N_{k-1}}, \lambda_{k-1, N_{k-1}})\| + \tfrac{\tilde{\omega}}{\sqrt{|S_k|}} \\
        &\quad + \Embb\left[\left\|\nabla f(x_{k, N_k}) - g_{S_k}(x_{k, N_k})\right\| | \Fcal_k \right] \\
        &\quad + \tilde{\gamma} \Embb\left[\|\nabla f(x_{k, 0}) - g_{S_k}(x_{k, 0})\| | \Fcal_k \right]\\ 
        &\leq \tilde{\gamma} \|T(x_{k-1, N_{k-1}}, \lambda_{k-1, N_{k-1}})\| + \tfrac{\tilde{\omega}}{\sqrt{|S_k|}}  \\
        &\quad + \kappa_G\tfrac{(\epsilon_G + \tilde{\kappa}_g)}{\sqrt{|S_k|}} + \tilde{\gamma} \tfrac{\tilde{\omega}_1 \|\nabla f(x_{k, 0})\| + \tilde{\omega}_2}{\sqrt{|S_k|}} \\
        &\leq \tilde{\gamma}\|T(x_{k-1, N_{k-1}}, \lambda_{k-1, N_{k-1}})\| \\
        &\quad + \tfrac{\tilde{\beta}^k}{\sqrt{|S_0|}}\left[\tilde{\omega} + \kappa_G(\epsilon_G + \tilde{\kappa}_g) + \tilde{\gamma} (\tilde{\omega}_1 \|\nabla f(x_{k, 0})\| + \tilde{\omega}_2)\right]
    \end{align*}
    where the second inequality follows from \eqref{eq:stoch_sampled_gradient_error} and \eqref{eq:biased_gradient_error_bound} for the gradient error with bias and \cref{ass:gradient_errors} for the unbiased gradient error, and the third inequality follows from the chosen batch size sequence. 
    Taking total expectation of the above bound, with $\Embb[\|\nabla f(x_{k, 0})\|] = \Embb\left[\Embb\left[\|\nabla f(x_{k-1, N_{k-1}})\| | \Fcal_{k-1}\right]\right] \leq \tilde{\kappa}_g$ by \cref{ass:Bounded_outer_gradient} and substituting the defined constants,
    \begin{align*}
    \Embb\left[\|T(x_{k, N_k}, \lambda_{k, N_k})\| \right]
        &\leq \tilde{\gamma} \Embb[\|T(x_{k-1, N_{k-1}}, \lambda_{k-1, N_{k-1}})\|] + \tilde{a}_2 \tilde{\beta}^k,
    \end{align*}
    where applying \cref{lem:linear_convergence_induction} with $Z_k = \Embb\left[\|T(x_{k, N_k}, \lambda_{k, N_k})\| | \Fcal_k \right]$, $\rho_1 = \tilde{\gamma} $, $\rho_2 = \tilde{\beta}$ and $b = \tilde{a}_2$ completes the proof.
\end{proof}
\cref{cor:Eq_geometric_batch_increase_outer_iter_complexity} establishes linear convergence for \cref{alg:Equality_Constrained_RA} under a pre-specified geometrically increasing batch size sequence, a simplified alternative to \cref{cond:eq_norm_test}. 
Since the batch sizes are deterministically chosen, the tolerance sequence ${\epsilon_k}$ is controlled via the batch size, and the conditions and results in \cref{th:EQ_outer_iter_complexity} are simplified, particularly for the finite-sum problem \eqref{eq:intro_deter_error_obj}.
Alternatively, one can select sufficiently large batch sizes when $\theta$ (or $\tilde{\theta}$) is non-zero to ensure that \cref{cond:eq_norm_test} holds. 
These batch sizes are conservative and derived under a stronger bound on the true problem objective function gradient than \cref{ass:Bounded_outer_gradient} for the expectation problem \eqref{eq:intro_stoch_error_obj}. 
Since a deterministic bound on the true problem gradient is a standard assumption in both deterministic \cite{powell2006fast,nocedal2006numerical,byrd2008inexact} and stochastic \cite{berahas2021sequential,curtis2024worst,o2024two} constrained optimization as previously discussed, the assumption is reasonable.

\begin{assumption} \label{ass:Bounded_outer_gradient_strong}
    Let $\{x_{k, N_k}\}$ be the sequence of outer iterates generated by any run of \cref{alg:Equality_Constrained_RA}. There exists $\kappa_g \geq 0$ such that $\|\nabla f(x_{k, N_k})\| \leq \kappa_g$ for all $k \geq 0$.
\end{assumption}

The next lemma provides sufficiently large batch sizes that ensure \cref{cond:eq_norm_test} holds.

\begin{lemma} \label{lem:eq_large_enough_batch_sizes}
    Suppose Assumptions \ref{ass:EQ_base assumption}, \ref{ass:well_posed}, \ref{ass:gradient_errors} and \ref{ass:Bounded_outer_gradient_strong} hold in \cref{alg:Equality_Constrained_RA}.
    \begin{enumerate}
        \item For the finite-sum problem \eqref{eq:intro_deter_error_obj}: For $k \geq 0$, \cref{cond:eq_norm_test} is satisfied if 
        \begin{equation*}
            |S_k| \geq |\Scal|\left(1 - \sqrt{\tfrac{\theta^2 \|T(x_{k, 0}, \lambda_{k-1, N_{k-1}})\|^2 + a^2 \beta^{2k}}{4(\omega_1^2 \kappa_g^2 + \omega_2^2)}}\right) \quad \text{with} \quad \hat{\theta} = 1.
        \end{equation*}
        \item For the expectation problem \eqref{eq:intro_stoch_error_obj}: For $k \geq 0$, \cref{cond:eq_norm_test} is satisfied if
        \begin{equation*}
            |S_k| \geq  \tfrac{\tilde{\omega}_1^2 \kappa_g^2 + \tilde{\omega}_2^2}{\tilde{\theta}^2 \|T(x_{k, 0}, \lambda_{k-1, N_{k-1}})\|^2 + \tilde{a}^2 \tilde{\beta}^{2k}}.
        \end{equation*}
    \end{enumerate}
\end{lemma}
\begin{proof}
    For the finite-sum problem \eqref{eq:intro_deter_error_obj},
    \begin{align*}
        \|\nabla f(x_{k, 0}) - g_{S_k} (x_{k, 0})\|^2 &\leq 4\left(1 - \tfrac{|S_k|}{|\Scal|}\right)^2 (\omega_1^2 \|\nabla f (x_{k, 0})\|^2 + \omega_2^2) \\
        &\leq 4 \left(\sqrt{\tfrac{\theta^2 \|T(x_{k, 0}, \lambda_{k-1, N_{k-1}})\|^2 + \tilde{a}^2 \tilde{\beta}^{2k}}{4(\omega_1^2 \kappa_g^2 + \omega_2^2)}}\right)^2 (\omega_1^2 \kappa_g^2 + \omega_2^2),
    \end{align*}
    where the first inequality follows from \eqref{eq:deter_sampled_gradient_error} and the second inequality follows from \cref{ass:Bounded_outer_gradient_strong} and the lower bound on $|S_k|$, thereby satisfying \cref{cond:eq_norm_test}. 
    The gradient error condition on $x_{k, N_k}$ with $\hat{\theta} = 1$ can be verified in the same way. For the expectation problem \eqref{eq:intro_stoch_error_obj},
    \begin{align*}
        \Embb[\|\nabla f(x_{k, 0}) - g_{S_k} (x_{k, 0})\|^2 | \Fcal_k] &\leq \tfrac{\tilde{\omega}_1^2 \|\nabla f(x_{k, 0})\|^2 + \tilde{\omega}_2^2}{|S_k|}\\ 
        &\leq (\tilde{\omega}_1^2 \kappa_g^2 + \tilde{\omega}_2^2)\tfrac{\tilde{\theta}^2 \|T(x_{k, 0}, \lambda_{k-1, N_{k-1}})\|^2 + a^2 \beta^{2k}}{\tilde{\omega}_1^2 \kappa_g^2 + \tilde{\omega}_2^2},
    \end{align*}
    where the first inequality follows from \eqref{eq:stoch_sampled_gradient_error} and the second inequality follows from \cref{ass:Bounded_outer_gradient_strong} and the suggested lower bound on $|S_k|$, thereby satisfying \cref{cond:eq_norm_test}.
\end{proof}

We have analyzed the iterates generated by \cref{alg:Equality_Constrained_RA} under \cref{cond:eq_norm_test} for sampling across outer iterations. \cref{th:EQ_outer_iter_complexity} established a linear convergence rate across outer iterations, implying that an $\epsilon$-accurate solution can be achieved in $K_{\epsilon}$ outer iterations, where $K_\epsilon = \Ocal\left(\log \tfrac{1}{\epsilon}\right)$. While this result provides an upper bound on the number of outer iterations, the dominant computational cost in \cref{alg:Equality_Constrained_RA} arises from the inner loop—specifically, from solving deterministic subproblems (e.g., the \texttt{SQP} subproblem \eqref{eq:SQP_quad_model}) and computing the subsampled gradients.

We now analyze the total work complexity of \cref{alg:Equality_Constrained_RA} to achieve an $\epsilon$-accurate solution in terms of two metrics: $(1)$ inner iteration complexity, which is the total number of inner iterations performed $\left(\sum_{k=0}^{K_\epsilon} N_k\right)$, and $(2)$ gradient complexity, which is the total number of subsampled gradients computed $\left(\sum_{k=0}^{K_\epsilon} |S_k| N_k\right)$, assuming one gradient computation per inner iteration. 
To derive these results, we need to characterize the performance of the deterministic solver used in the inner loop beyond \cref{ass:well_posed}. We focus on a specific sublinear rate of convergence for the deterministic solver, most commonly achieved for deterministic problems with a nonconvex objective function and nonconvex equality constraints \cite{berahas2022adaptive,curtis2024worst}.

\begin{assumption} \label{ass:sublinear_solver}
The inner loop in \cref{alg:Equality_Constrained_RA} exhibits sublinear convergence, i.e., for all $k \geq 0$ an $\epsilon$-accurate solution to the subsampled problem, $\|T_{S_k}(x_{k, N_k}, \lambda_{k, N_k})\| \leq \epsilon$, is achieved in $N_k = \Ocal(\epsilon^{-2})$ iterations.
\end{assumption}

Given \cref{ass:sublinear_solver}, we analyze the total work complexity of \cref{alg:Equality_Constrained_RA}.

\begin{theorem} \label{th:EQ_work_complexity} 
    Suppose Assumptions \ref{ass:EQ_base assumption}, \ref{ass:gradient_errors}, \ref{ass:Bounded_outer_gradient_strong} and \ref{ass:sublinear_solver} hold and the batch size sequence $\{|S_k|\}$ is chosen to satisfy \cref{cond:eq_norm_test} in \cref{alg:Equality_Constrained_RA}.
    \begin{enumerate}
        \item For the finite-sum problem \eqref{eq:intro_deter_error_obj}: For all $k \geq 0$, if termination criterion parameters are chosen as $0 \leq \gamma_k \leq \gamma < 1$ and $\epsilon_k = \omega \|\nabla f(x_{k, 0}) - g_{S_k}(x_{k, 0})\| + \hat{\omega}\beta^k$ with $\omega, \hat{\omega} > 0$, and \cref{cond:eq_norm_test} parameters are chosen such that $\left[  \gamma + \theta(\omega + \hat{\theta} + \gamma)\right]  < \beta$ 
        , one achieves an $\epsilon > 0$ accurate solution in $\Ocal(\epsilon^{-2})$ inner iterations and $\Ocal(|\Scal|\epsilon^{-2})$ gradient evaluations.

        \item For the expectation problem \eqref{eq:intro_stoch_error_obj}: For all $k \geq 0$, if \cref{ass:Variance_regularity} holds, the termination criterion parameters are chosen as $0 \leq \gamma_k \leq \tilde{\gamma} < 1$, $\epsilon_k = \tilde{\omega}\sqrt{\tfrac{\Var(\nabla F(x_{k, 0}, \xi) | \Fcal_k)}{|S_k|}}$ where $\tilde{\omega} > 0$, and \cref{cond:eq_norm_test} parameters are chosen such that $\left[ \tilde{\gamma}  + \tilde{\theta} \left(\tilde{\omega} + \tfrac{(\epsilon_G + \kappa_g)}{\kappa_{\sigma}} + \tilde{\gamma}\right)\right] < \tilde{\beta}$ and $\tilde{a} = \sqrt{\tfrac{\tilde{\omega}_1^2 \kappa_g^2 + \tilde{\omega}_2^2}{|S_{0}|}}$, where $|S_0|$ is the initial batch size, one achieves an $\epsilon > 0$ accurate solution in expectation in $\Ocal(\epsilon^{-2})$ inner iterations and $\Ocal(\epsilon^{-4})$ gradient evaluations.
    \end{enumerate}
\end{theorem}
\begin{proof}
    For the finite-sum problem \eqref{eq:intro_deter_error_obj}, the number of outer iteration to reach an $\epsilon$ accurate solution is $K_\epsilon = \Ocal\left(\tfrac{1}{\log (1 / \beta)} \log \left(\tfrac{1}{\epsilon}\right)\right)$ under the stated conditions from \cref{th:EQ_outer_iter_complexity}. Thus, under \cref{ass:sublinear_solver}, the total number of inner iterations can be bounded  with a constant $C > 0$ as,
    \begin{align*}
        \sum_{k = 0}^{K_\epsilon} N_k 
        \leq  \sum_{k = 0}^{K_\epsilon} \tfrac{C}{(\gamma_k \|T_{S_k}(x_{k, 0}, \lambda_{k, 0})\| + \omega \|\nabla f(x_{k, 0}) - g_{S_k}(x_{k, 0})\| + \hat{\omega}\beta^k)^2} 
        \leq \sum_{k = 0}^{K_\epsilon} \tfrac{C}{\hat{\omega}^2 \beta^{2k}}
        = \tfrac{C}{\hat{\omega}^2} \tfrac{(1 / \beta)^{2 (K_\epsilon + 1)} - 1}{(1 / \beta)^2 - 1},
    \end{align*}
    which is of the order 
    $\left(\tfrac{1}{\beta}\right)^{2 K_\epsilon} = \Ocal (\epsilon^{-2})$. 
    As the batch size is bounded by $|S_k| \leq |\Scal|$, the total number of gradient evaluations can be bounded as $\sum_{k = 0}^{K_\epsilon} |S_k|N_k  \leq |\Scal| \sum_{k = 0}^{K_\epsilon} N_k = \Ocal(|\Scal| \epsilon^{-2})$.

    For the expectation problem \eqref{eq:intro_stoch_error_obj}, the number of outer iterations to achieve an $\epsilon$ accurate solution in expectation is $K_\epsilon = \Ocal\left(\tfrac{1}{\log (1 / \tilde{\beta})} \log \left(\tfrac{1}{\epsilon}\right)\right)$ under the stated conditions from \cref{th:EQ_outer_iter_complexity}. The expected number of inner iterations in outer iteration $k \geq 0$ can be bounded with a constant $C > 0$ as,
    \begin{align*}
        \Embb[N_k | \Fcal_k ] 
        \leq \Embb \left[\left. \tfrac{C}{\left(\gamma_k \|T_{S_k}(x_{k, 0}, \lambda_{k, 0})\| + \tilde{\omega}\sqrt{\tfrac{\Var(\nabla F(x_{k, 0}, \xi) | \Fcal_k)}{|S_k|}}\right)^2} \right| \Fcal_k  \right]
        \leq \tfrac{C|S_k|}{\tilde{\omega}^2\Var(\nabla F(x_{k, 0}, \xi) | \Fcal_k)},
    \end{align*} 
    where the final bound is a deterministic quantity given $\Fcal_k$.
    Let $|S_k|$ be selected such that \cref{cond:eq_norm_test} holds with equality, i.e.,
    \begin{align*}
        |S_k| = \left\lceil\tfrac{\Var(\nabla F(x_{k, 0}, \xi) | \Fcal_k)}{\tilde{\theta}^2 \|T(x_{k, 0}, \lambda_{k-1, N_{k-1}})\|^2 + \tilde{a}^2 \tilde{\beta}^{2k}}\right\rceil \leq \tfrac{\Var(\nabla F(x_{k, 0}, \xi) | \Fcal_k)}{\tilde{\theta}^2 \|T(x_{k, 0}, \lambda_{k-1, N_{k-1}})\|^2 + \tilde{a}^2 \tilde{\beta}^{2k}} + 1.
    \end{align*}
    With this selection, the bound on the expected number of inner iterations can be further refined,
    \begin{align*}
        \Embb[N_k | \Fcal_k ] 
        \leq \tfrac{C}{\tilde{\omega}^2 (\tilde{\theta}^2 \|T(x_{k, 0}, \lambda_{k-1, N_{k-1}})\|^2 + \tilde{a}^2 \tilde{\beta}^{2k})} + \tfrac{C}{\tilde{\omega}^2\Var(\nabla F(x_{k, 0}, \xi) | \Fcal_k)}
        \leq\tfrac{C}{\tilde{\omega}^2 \tilde{a}^2 \tilde{\beta}^{2k}} + \tfrac{C}{\tilde{\omega}^2\kappa_G^2\kappa_{\sigma}^2}, 
    \end{align*} 
    where the final bound follows from \cref{ass:Variance_regularity}. Thus, the expected total number of inner iterations can be bounded as,
    \begin{align*}
        \sum_{k = 0}^{K_\epsilon} \Embb[\Embb[N_k | \Fcal_k ]] 
        \leq \sum_{k = 0}^{K_\epsilon} \tfrac{C}{ \tilde{\omega}^2\tilde{a}^2 \tilde{\beta}^{2k}} + \tfrac{C}{\tilde{\omega}^2\kappa_G^2\kappa_{\sigma}^2} 
        = \tfrac{C}{ \tilde{\omega}^2 \tilde{a}^2} \tfrac{(1 / \tilde{\beta})^{2 (K_\epsilon + 1)} - 1}{(1 / \tilde{\beta})^2 - 1} + \tfrac{C(K_{\epsilon} + 1)}{\tilde{\omega}^2\kappa_G^2\kappa_{\sigma}^2}, 
    \end{align*} 
    which is of the order $\left(\tfrac{1}{\tilde{\beta}}\right)^{2 K_\epsilon} = \Ocal (\epsilon^{-2})$. 
    The expected number of gradient evaluations in outer iteration $k \geq 0$ can be bounded as
    \begin{align*}
        \Embb[|S_k| N_k | \Fcal_k ] 
        \leq \tfrac{C |S_k| }{\tilde{\omega}^2 \tilde{a}^2 \tilde{\beta}^{2k}} + \tfrac{C |S_k| }{\tilde{\omega}^2\kappa_G^2\kappa_{\sigma}^2},
    \end{align*} 
    which is a deterministic quantity given $\Fcal_k$.
    Under the stated parameter selection, if one constructs a batch size sequence $|\tilde{S}_k| = \left\lceil\tfrac{|S_0|}{\tilde{\beta}^{2k}}\right\rceil$, where $|S_0|$ is the initial batch size, $\Embb[\|g_{\tilde{S}_k}(x_{k, 0}) - \nabla f(x_{k, 0})\|^2 | \Fcal_k] \leq \tilde{\beta}^{2k}\tfrac{\tilde{\omega}_1^2 \kappa_g^2 + \tilde{\omega}_2^2}{|S_{0}|} = \tilde{a}^2 \tilde{\beta}^{2k}$ based on \cref{lem:eq_large_enough_batch_sizes}. Thus, one can satisfy a stronger gradient error condition with $|\tilde{S}_k|$ than \cref{cond:eq_norm_test}, implying $|S_k| \leq |\tilde{S}_k|$. The expected total number of gradient evaluations can then be bounded as,
    \begin{align*}
        \sum_{k = 0}^{K_\epsilon} \Embb[\Embb[|S_k| N_k | \Fcal_k ]] 
        &\leq \sum_{k = 0}^{K_\epsilon}\tfrac{C |S_k| }{\tilde{\omega}^2 \tilde{a}^2 \tilde{\beta}^{2k}} + \tfrac{C |S_k| }{\tilde{\omega}^2\kappa_G^2\kappa_{\sigma}^2}
        \leq \sum_{k = 0}^{K_\epsilon}\tfrac{C |\tilde{S}_k| }{\tilde{\omega}^2 \tilde{a}^2 \tilde{\beta}^{2k}} + \tfrac{C |\tilde{S}_k| }{\tilde{\omega}^2\kappa_G^2\kappa_{\sigma}^2}\\
        &= \sum_{k = 0}^{K_\epsilon}\tfrac{C |S_0| }{\tilde{\omega}^2 \tilde{a}^2 \tilde{\beta}^{4k}} + \tfrac{C |S_0| }{\tilde{\omega}^2\kappa_G^2\kappa_{\sigma}^2\tilde{\beta}^{2k}} \\
        &= \tfrac{C |S_0| }{\tilde{\omega}^2 \tilde{a}^2} \tfrac{(1 / \tilde{\beta})^{4 (K_\epsilon + 1)} - 1}{(1 / \tilde{\beta})^2 - 1} + \tfrac{C |S_0| }{\tilde{\omega}^2 \kappa_G^2\kappa_{\sigma}^2} \tfrac{(1 / \tilde{\beta})^{2 (K_\epsilon + 1)} - 1}{(1 / \tilde{\beta})^2 - 1},
    \end{align*} 
    which is of the order 
    $\left(\tfrac{1}{\tilde{\beta}}\right)^{4 K_\epsilon} = \Ocal (\epsilon^{-4})$.
\end{proof}

The results of \cref{th:EQ_work_complexity} hold for any deterministic solver satisfying \cref{ass:sublinear_solver}, which assumes sublinear convergence in the inner loop. Improved complexity results are possible if faster convergence rates are achieved within the inner loop. For the finite-sum problem \eqref{eq:intro_deter_error_obj}, the derived complexity matches that of a deterministic solver applied directly to the full problem using the true gradient. 
The current analysis is somewhat pessimistic in terms of gradient evaluations in bounding $|S_k| \leq |\Scal|$. A more refined argument, similar to that for the expectation problem with $a = 2\sqrt{\omega_1^2 \kappa_g^2 + \omega_2^2}$ in \cref{cond:eq_norm_test}, yields
\begin{align*}
    \sum_{k = 0}^{K_\epsilon} N_k |S_k|
    \leq \sum_{k = 0}^{K_\epsilon} \tfrac{C}{\hat{\omega}^2 \beta^{2k}} (1 - \beta^k) |\Scal|
    &= |\Scal| \tfrac{C}{\hat{\omega}^2} \sum_{k = 0}^{K_\epsilon} \left(\tfrac{1}{\beta^{2k}} - \tfrac{1}{\beta^{k}}\right)
\end{align*}
which is of the same order, since $|\Scal|\left(\epsilon^{-2} - \epsilon^{-1}\right) = \Ocal\left(|\Scal|\epsilon^{-2}\right)$. Thus, we present the simpler analysis.
For the expectation problem, the gradient evaluation complexity of $\Ocal(\epsilon^{-4})$ aligns with the optimal complexity bounds established for stochastic SQP methods \cite{berahas2025sequential, curtis2024worst, o2024two}.
The gradient and inner iteration complexities are independent of the choice of inner loop solver under \cref{ass:sublinear_solver}, assuming one subsampled gradient computation per inner iteration. However, the actual computational cost also depends on the subproblem being solved in each inner iteration by the deterministic solver. For instance, \texttt{SQP}-based methods solve a linear system of equations for general nonlinear equality constraints, while projection-based methods compute a projection onto the feasible region. The relative cost of these subproblems depends on the subsampled problem and is an important consideration when selecting the deterministic solver.

\subsection{Retrospective Approximation SQP Framework} \label{sec:SQP_based_algorithm_eq}
In this section, we present an instance of the \texttt{RA}-based framework (\cref{alg:Equality_Constrained_RA}) that employs a deterministic \texttt{SQP} solver in the inner loop to solve stochastic problems with deterministic equality constraints.
We describe the general framework, illustrate how different options can be implemented within the \texttt{SQP} method, and propose alternative termination criteria to \eqref{eq:termination_criterion_eq} that leverage quantities from the \texttt{SQP} method to improve empirical performance.

In outer iteration $k$, a subsampled problem \eqref{eq:subsampled_problem} is constructed using a sample set $S_k$, and the dual variable is initialized according to \eqref{eq:dual_update_eq} for the inner loop. 
Within each inner iteration $j$ of outer iteration $k$, the termination criterion \eqref{eq:termination_criterion_eq} is tested.
If the criterion is not satisfied, the iterate and dual variable are updated using an \texttt{SQP} method as follows, based on the approaches in \cite{byrd2010inexact,byrd2008inexact,nocedal2006numerical,berahas2021sequential}.
Given an approximation of the Hessian of the Lagrangian $H_{k,j}$ for the subsampled problem, the method computes a search direction via the subproblem given in \eqref{eq:SQP_quad_model}.
Under the LICQ and assuming the matrix $H_{k, j}$ is positive definite in the null space of $J_{k, j}$, the solution to the \texttt{SQP} subproblem \eqref{eq:SQP_quad_model} can be obtained by solving the linear system
\begin{equation} \label{eq:eq_sqp_system}
    \begin{bmatrix}
        H_{k,j} & J_{k, j}^T \\ J_{k, j} & 0 
    \end{bmatrix}
    \begin{bmatrix}
        d_{k, j} \\ \delta_{k, j} 
    \end{bmatrix}
    = 
    - T_{S_k}(x_{k, j}, \lambda_{k, j})
    +
    \begin{bmatrix}
        \rho_{k, j} \\ r_{k, j} 
    \end{bmatrix},
\end{equation}
where ($d_{k, j}$, $\delta_{k, j}$) is the search direction, and ($\rho_{k, j}$, $r_{k, j}$) are vectors that represent the residual in the linear system.
The linear system can be solved either exactly, i.e., $\rho_{k, j} = 0$, $r_{k, j} = 0$, or inexactly, where $\rho_{k, j}$ and $r_{k, j}$ satisfy certain conditions to ensure ($d_{k, j}$, $\delta_{k, j}$) is a productive search direction \cite{byrd2008inexact,berahas2022adaptive,curtis2021inexact}. 

To balance the two possibly competing goals of minimizing the subsampled objective function and minimizing the constraint violation when updating the iterate, the $l_1$ merit function $\phi_{S_k} : \Rmbb^n \times \Rmbb \to \Rmbb$ defined over the set of samples $S_k$ as,
\begin{align*}
    \phi_{S_k} (x, \tau) = \tau F_{S_k}(x) + \|c(x)\|_1
\end{align*}
is employed, where $\tau > 0$ is the merit parameter updated as necessary within the inner loop.
A local linear model of the merit function $l_{S_k} : \Rmbb^n \times \Rmbb \times \Rmbb^n \to \Rmbb$ defined as, 
\begin{align*}
    l_{S_k} (x_{k, j}, \tau_{k, j}, d_{k,j}) &= \tau_{k, j} \left(F_{S_k}(x_{k, j}) + g_{S_k}(x_{k, j})^Td_{k,j}\right) + \|c_{k, j} + J_{k, j}d_{k, j}\|_1.
\end{align*}
and the reduction in this local model for the search direction $d_{k, j}$ is used to guide the merit parameter selection, defined as $\Delta l_{S_k} : \Rmbb^n \times \Rmbb \times \Rmbb^n \to \Rmbb$, 
\begin{align*}
    \Delta l_{S_k} (x_{k, j}, \tau_{k, j}, d_{k,j}) 
    &= l_{S_k} (x_{k, j}, \tau_{k, j}, 0) - l_{S_k} (x_{k, j}, \tau_{k, j}, d_{k,j}) \\
    &= -\tau_{k, j}g_{S_k}(x_{k, j})^Td_{k,j} + \|c_{k, j}\|_1 - \|c_{k, j} + J_{k, j}d_{k, j}\|_1 \\
    &= -\tau_{k, j}g_{S_k}(x_{k, j})^Td_{k,j} + \|c_{k, j}\|_1 - \|r_{k,j}\|_1. \numberthis \label{eq:eq_delta_l}
\end{align*}
The merit parameter $\tau_{k, j}$ is updated to ensure the search direction is a descent direction for the merit function. With user-defined parameters $\epsilon_\sigma, \epsilon_{\tau}, \epsilon_d \in (0, 1)$, a trial value $\tau_{k,j}^{trial}$ is obtained as,
\begin{align*}
\tau^{trial}_{k,j} =
    \begin{cases}
        \infty & \text{if } g_{S_k}(x_{k, j})^Td_{k,j} + \max\{d_{k,j}^TH_{k,j}d_{k,j}, \epsilon_d \|d_{k, j}\|^2\} \leq 0, \\
        \tfrac{(1 - \epsilon_\sigma) (\|c_{k, j}\|_1 - \|r_{k, j}\|_1)}{g_{S_k}(x_{k, j})^Td_{k,j} + \max\{d_{k,j}^TH_{k,j}d_{k,j}, \epsilon_d \|d_{k, j}\|^2\}} & \text{otherwise,}
    \end{cases}
\end{align*}
and then the merit parameter is updated via,
\begin{align}  \label{eq:eq_merit_update}
\tau_{k,j} =
\begin{cases}
    \tau_{k,j-1} & \text{if } \tau_{k,j-1} \leq \tau^{trial}_{k,j}, \\
    (1 - \epsilon_{\tau}) \tau^{trial}_{k,j} & \text{otherwise,}
\end{cases}
\end{align}
ensuring a non-decreasing positive sequence $\{\tau_{k,j}\}$ within the inner loop of each outer iteration $k$. 
At the beginning of each outer iteration, the merit parameter is reinitialized ($\{\tau_{k, -1}\} = \bar{\tau} > 0$) rather than carried over from previous iterations. While the subsampled objective function can drastically change across outer iterations, carrying over the merit parameter would not affect convergence. However, carrying over the merit parameter might initialize it at a smaller value than required, in turn significantly slowing down the empirical performance of the \texttt{SQP} method in the inner loop.
Finally, the step size $\alpha_{k,j}$ is obtained to satisfy the Armijo sufficient decrease condition on the merit function, i.e.,
\begin{equation} \label{eq:line_search}
    \phi_{S_k}(x_{k, j} + \alpha_{k,j} d_{k,j}, \tau_{k,j}) \leq \phi_{S_k}(x_{k, j}, \tau_{k,j}) - \eta\alpha_{k,j} \Delta l_{S_k} (x_{k, j}, \tau_{k,j}, d_{k,j}),
\end{equation}
where $\eta \in (0, 1)$ is a user-defined parameter, and the iterate and dual variable are updated via $(x_{k, j + 1}, \lambda_{k, j + 1}) = (x_{k, j}, \lambda_{k, j}) + \alpha_{k, j} (d_{k,j}, \delta_{k,j})$.

The \texttt{SQP} linear system \eqref{eq:eq_sqp_system} can be solved inexactly using a numerical solver while still ensuring that $(d_{k, j}, \delta_{k, j})$ is a productive direction.
We adopt the inexactness conditions as proposed in \cite{byrd2008inexact,berahas2022adaptive}. Under these conditions, a solution to the \texttt{SQP} subproblem $(d_{k, j}, \delta_{k, j})$ is accepted if it satisfies either inexactness conditions I, which ensures sufficient decrease in the merit function model,
\begin{align*} 
    \Delta l_{S_k} (x_{k, j}, \tau_{k,j-1}, d_{k,j}) &\geq  \epsilon_{\sigma} ( 1 - \epsilon_{feas}) \max\{\|c_{k, j}\|_1, \|r_{k, j}\| - \|c_{k, j}\|_1\} \\
    &\quad + \epsilon_{\sigma} (1 - \epsilon_{feas}) \tau_{k, j-1} \max\{d_{k,j}^TH_{k,j}d_{k,j}, \epsilon_d \|d_{k, j}\|^2\},\\
    \numberthis \label{eq:eq_inexactness_condition_1}
    \left\|\begin{bmatrix}
    \rho_{k,j} \\ r_{k,j}
    \end{bmatrix}\right\| &\leq  \kappa_{T} 
    \min\left\{\left\|T_{S_k}(x_{k, j}, \lambda_{k, j})\right\|, \|d_{k, j}\|\right\}, \\ \text{and} 
    \qquad 
    \|\rho_{k,j}\| &\leq \kappa' \max\{\|J_{k, j}\|, \|g_{S_k}(x_{k, j})\| \},
\end{align*}
where $\epsilon_{\sigma},\epsilon_{d} \in \left(0, 1\right)$, $\kappa_{T} \in \left(0, \tfrac{1}{2}\right)$ and  $\kappa'>0$ are user-defined parameters, or inexactness conditions II, which guarantees sufficient decrease in the linear model of the constraints,
\begin{align} \label{eq:eq_inexactness_condition_2}
    \|r_{k,j}\| \leq \epsilon_{feas}\|c_{k ,j}\| \quad \text{and} \quad \|\rho_{k,j}\| \leq \epsilon_{opt} \|c_{k, j}\|,
\end{align}
where $\epsilon_{feas}, \epsilon_{opt} \in \left(0, \tfrac{1}{2}\right)$ are user-defined parameters. If inexactness condition I is satisfied, the merit parameter remains unchanged, i.e., $\tau_{k, j} = \tau_{k, j-1}$.
Compared to the conditions proposed in \cite{byrd2010inexact, byrd2008inexact}, inexactness condition I is a stronger condition due to the introduction of the norm of the search direction $\|d_{k, j}\|$ in the second inequality in \eqref{eq:eq_inexactness_condition_1}. This extra condition is not required when using termination criterion \eqref{eq:termination_criterion_eq} in the inner loop, but is introduced to ensure convergence when inexact solutions to the \texttt{SQP} linear system \eqref{eq:eq_sqp_system} are used along with the alternate termination criteria for the inner loop described next.

We summarize the \texttt{SQP}-based variant of \cref{alg:Equality_Constrained_RA} in \cref{alg:Equality_Constrained_RA_SQP} and the following remark. 
\begin{algorithm}
    \caption{Equality Constrained \texttt{RA}-\texttt{SQP}}
    \textbf{Inputs:} Sequences $\{|S_k|\}$, $\{\gamma_k\} < 1$ and $\{\epsilon_k\} \geq 0$; initial iterate $x_{0,0}$ and dual variable $\lambda_{-1, 0}$; initial merit parameter sequence $\{\tau_{k, -1}\} = \bar{\tau} > 0$; other parameters $ \epsilon_{\sigma}, \epsilon_{\tau}, \epsilon_d, \epsilon_{\alpha} \in (0, 1)$.
    \begin{algorithmic}[1]
        \For{$k = 0, 1, 2, \dots$}
            \State Construct subsampled problem \eqref{eq:subsampled_problem} and initialize $\lambda_{k, 0}$ according to \eqref{eq:dual_update_eq} \label{EQ_SQP:subsample}
            \For{$j = 0, 1, 2, \dots$}
                \State \textbf{if} \eqref{eq:termination_criterion_eq} is satisfied; $N_k = j$, \textbf{break} \label{EQ_SQP:termination}
                \State Obtain approximation of Hessian of Lagrangian $H_{k, j}$ \label{EQ_SQP:H}
                \State Solve \eqref{eq:eq_sqp_system} exactly or inexactly and obtain $(d_{k,j},\delta_{k,j})$ \label{EQ_SQP:d_cal}
                \State Update merit parameter via \eqref{eq:eq_merit_update} \label{EQ_SQP:tau_update}
                \State Set $\alpha_{k, j} = 1$ \label{EQ_SQP:line_Search_start}
                \While{\eqref{eq:line_search} is \textbf{not satisfied}}
                    \State Update $\alpha_{k, j} = \epsilon_{\alpha} \alpha_{k, j}$
                \EndWhile \label{EQ_SQP:line_Search_end}
                \State Set $(x_{k, j + 1}, \lambda_{k, j + 1}) = (x_{k, j}, \lambda_{k, j}) + \alpha_{k, j} (d_{k,j}, \delta_{k,j})$ \label{EQ_SQP:update}
            \EndFor
            \State Set $x_{k+1, 0} = x_{k, N_k}$
        \EndFor
    \end{algorithmic}
    \label{alg:Equality_Constrained_RA_SQP}
\end{algorithm}
\begin{remark}
\cref{alg:Equality_Constrained_RA_SQP} follows the framework of \cref{alg:Equality_Constrained_RA}, incorporating an \texttt{SQP} method within the inner loop. We make the following remarks regarding the inner loop. \begin{enumerate}
        \item \textbf{Lagrangian Hessian approximation (Line \ref{EQ_SQP:H}):} While the identity matrix is a common choice for $H_{k, j}$ in constrained stochastic optimization, the deterministic nature of the subsampled problem allows for the use of more sophisticated approximations. In particular, one may employ the subsampled Hessian of the Lagrangian or a quasi-Newton approximation. When quasi-Newton methods, such as BFGS or L-BFGS, are used, the Hessian approximation is updated at each inner iteration and initialized using the final approximation from the previous outer iteration, i.e.,  in outer iteration $k$, we set $H_{k,0} = H_{k-1,N_{k-1}}$.
        \item \textbf{SQP linear system (Line \ref{EQ_SQP:d_cal}):} The \texttt{SQP} subproblem \eqref{eq:eq_sqp_system} can be solved either exactly or inexactly. If a numerical solver is used, one can apply the inexactness conditions \eqref{eq:eq_inexactness_condition_1} and \eqref{eq:eq_inexactness_condition_2} based on \cite{byrd2008inexact,berahas2022adaptive} to ensure productive search directions.
        \item \textbf{Merit parameter and step size (Lines \ref{EQ_SQP:tau_update}-\ref{EQ_SQP:update}):} The merit parameter update follows from \cite{berahas2021sequential,berahas2022adaptive}. Since the objective function can change drastically across outer iterations, the merit parameter is reinitialized at the start of each outer iteration to a user-specified constant $\bar{\tau}$. The step size is computed via a backtracking Armijo line search procedure \cite{powell1986recursive,nocedal2006numerical} on the merit function \eqref{eq:line_search}.
    \end{enumerate}
\end{remark}

For \texttt{SQP} methods, we propose two alternative termination criteria that leverage quantities from the \texttt{SQP} method to assess the progress of the algorithm on the subsampled problem. The first termination criterion is based on the norm of the search direction,
\begin{equation} \label{eq:termination_criterion_d} 
    \Tcal_k : \|d_{k, N_k}\| \leq \gamma_{k} \|d_{k, 0}\| + \epsilon_{k},  
\end{equation}
where the search direction $d_{k, j}$ serves as a measure of optimality. The second termination criterion is based on the decrease in the merit function model,
\begin{align}
    \Tcal_k : \Delta l_{S_k} (x_{k, N_k}, \tau_{k, N_k}, d_{k,N_k})  &\leq \gamma_{k} \min\{\Delta l_{S_k} (x_{k, 0}, \tau_{k, 0}, d_{k,0}), \kappa_d \|d_{k, 0}\|^2\} + \epsilon_{k},\label{eq:termination_criterion_delta_l}
\end{align}
where $\kappa_d > 0$. While the decrease in the merit function model is a bounded quantity, the bound can be quite large, necessitating the inclusion of the search direction norm for proper scaling. In practice, $\kappa_d$ is chosen to be quite large, and the search direction norm term on the right-hand side of \eqref{eq:termination_criterion_delta_l} is typically inactive.

The termination criterion \eqref{eq:termination_criterion_eq}, based on the KKT error of the subsampled problem, while theoretically sound presents challenges in practice.
Specifically, the two components of the KKT error, $\nabla_x \Lcal_{S_k} (x_{k, j}, \lambda_{k, j})$ and $c_{k, j}$ can differ significantly in magnitude, potentially causing the algorithm to focus disproportionately on either reducing the objective function or satisfying the constraints, rather than a balanced approach. Termination criterion \eqref{eq:termination_criterion_d} avoids these scaling issues as it uses the norm of the search direction instead of combining multiple quantities. Termination criterion \eqref{eq:termination_criterion_delta_l} on the other hand explicitly balances the objective function and constraint violation using the merit parameter within the merit function model.

Unlike \eqref{eq:termination_criterion_eq}, criteria \eqref{eq:termination_criterion_d} and \eqref{eq:termination_criterion_delta_l} require the \texttt{SQP} linear system \eqref{eq:eq_sqp_system} solution for evaluation. Thus, they perform one extra solve of the \texttt{SQP} subproblem beyond the number of inner iterations executed during each outer iteration. 
When employing criteria \eqref{eq:termination_criterion_d} or \eqref{eq:termination_criterion_delta_l}, we use the \texttt{Carry-Over} option in \eqref{eq:dual_update_eq} for dual variable initialization and require the Hessian approximation of the Lagrangian to be carried over between outer iterations, i.e., $H_{k, 0} = H_{k-1, N_{k-1}}$. This Hessian requirement is naturally satisfied when using either a constant Hessian approximation or a quasi-Newton method.
Since $x_{k, 0} = x_{k-1, N_{k-1}}$ and $J_{k, 0} = J_{k-1, N_{k-1}}$, the Hessian approximation remains positive definite over the null space of $J_{k, 0}$. 
Thus, if this restriction poses challenges, e.g., when using subsampled Hessians, one can either use a different Hessian approximation for the first inner iteration or instead apply the termination criterion \eqref{eq:termination_criterion_eq}. If additional assumptions regarding the error in the subsampled Hessians are made, the requirement to carry over the Hessian can be dropped. However, for brevity, we maintain this requirement in this work. Termination criteria \eqref{eq:termination_criterion_d} and \eqref{eq:termination_criterion_delta_l} can be used when the \texttt{SQP} linear system is solved exactly or inexactly. While we limit our analysis and empirical exploration to these termination criterion and find that \eqref{eq:termination_criterion_delta_l} performs best empirically, one can consider other criteria based on the deterministic method employed in the inner loop.

\subsection{Retrospective Approximation SQP Convergence Analysis} \label{sec:SQP_based_algorithm_eq_theory}
In this section, we present the convergence properties of \cref{alg:Equality_Constrained_RA_SQP} for problems with only equality constraints. 
We first establish the required conditions under which \cref{alg:Equality_Constrained_RA_SQP} adheres to the theoretical results developed for \cref{alg:Equality_Constrained_RA} in \cref{sec:General_Deterministic_Solver_theory}. We then present results for the alternate termination tests proposed in \cref{sec:SQP_based_algorithm_eq} for the \texttt{SQP} method. 
We begin by stating the assumptions required to ensure the well-posedness of the algorithm.
\begin{assumption}  \label{ass:SQP_eq_assumptions}
    Let $\chi \subset \Rmbb^n$ be a closed bounded convex set containing the sequence of iterates generated by any run of \cref{alg:Equality_Constrained_RA_SQP}. 
    The objective function $f : \Rmbb^n \rightarrow \Rmbb$ is continuously differentiable and its gradient function $\nabla f : \Rmbb^n \rightarrow \Rmbb^n$ is Lipschitz continuous over $\chi$.
    The subsampled objective function $F_{S_k} : \Rmbb^n \rightarrow \Rmbb$ for all $k \geq 0$ is continuously differentiable with Lipschitz continuous gradients over $\chi$.
    The constraint function $c : \Rmbb^n \rightarrow \Rmbb^{m}$ is continuously differentiable and the Jacobian function, $ J = \nabla c^T : \Rmbb^n \rightarrow \Rmbb^{m \times n}$ is Lipschitz continuous over $\chi$. 
    The Jacobian $J(x)$ over $x \in \chi$ has full row rank (LICQ). 
    The sequence of symmetric Hessian approximation matrices $\{H_{k, j}\}$ is bounded in norm by $\kappa_H > 0$ such that $\|H_{k, j}\| \leq \kappa_H$. 
    In addition, there exists a constant $\mu_H > 0$ such that the matrix $u^T H_{k, j} u \geq \mu_H \|u\|^2$ for all $u \in \Rmbb^n$, where $J_{k, j}u = 0$.
\end{assumption}
Under \cref{ass:SQP_eq_assumptions}, there exist $\kappa_J, \kappa_g, \kappa_{HJ} > 0$ such that $\{\|J_{k, j}\|\} \leq \kappa_J$, $\{\|\nabla f(x_{k,j})\|\} \leq \kappa_g$ and $\left\{\left\|\begin{bmatrix}H_{k, j} & J_{k, j}^T\\J_{k, j} & 0\end{bmatrix}^{-1} \right\| \right\} \leq \tfrac{1}{\kappa_{HJ}}$. Thus, \cref{ass:EQ_base assumption} and \ref{ass:Bounded_outer_gradient_strong} are satisfied under \cref{ass:SQP_eq_assumptions}. Compared to \cite{berahas2021sequential,berahas2022adaptive,berahas2025sequential,byrd2008inexact,o2024two}, this is a stronger assumption for \texttt{SQP} based algorithms for stochastic problems with equality constraints, as it requires a closed bounded set of iterates and all subsampled problems to be smooth functions.
These requirements stem from the fundamental structure of \texttt{RA}, which treats each subsampled problem as a deterministic problem, requiring each to satisfy the standard assumptions for deterministic \texttt{SQP} methods, where the objective functions must be smooth with bounded gradients and be bounded below over the set of iterates $\chi$.
In extreme cases, this assumption may not hold for every subsampled problem. As a result, in practice, an upper limit is imposed on the number of inner iterations per outer iteration.

We now formally state the well-posedness of \cref{alg:Equality_Constrained_RA_SQP}.
\begin{theorem} \label{th:EQ_SQP_well_posed}
    Suppose Assumptions \ref{ass:gradient_errors}, \ref{ass:Variance_regularity} and \ref{ass:SQP_eq_assumptions} hold. Then \cref{alg:Inequality_Constrained_RA_SQP} is well-posed and the inner loop converges at a sublinear rate satisfying \cref{ass:sublinear_solver}, achieving the gradient and inner iteration complexity described in \cref{th:EQ_work_complexity}.
\end{theorem}
\begin{proof}
    The convergence of the \texttt{SQP} algorithm in the inner loop is established in \cite{curtis2024worst} for exact subproblem solutions and in \cite{byrd2008inexact} for inexact subproblem solutions, under \cref{ass:SQP_eq_assumptions}. 
    These results apply because, over the bounded set of iterates $\chi$, the subsampled problem $F_{S_k}(x)$ for all $k \geq 0$ is bounded below and has bounded gradients, thereby satisfying \cref{ass:well_posed}. 
    Under \cref{ass:SQP_eq_assumptions} and by \cite[Theorem 1]{curtis2024worst}, when exact search direction solutions are used, the number of inner loop iterations to achieve an $\epsilon$ accurate solution for the subsampled problem in outer iteration $k$ is
    \begin{align*}
        N_k \leq \left[\tfrac{\bar{\tau} \left(F_{S_k}(x_{k, 0}) - \min_{x \in \chi} F_{S_k}(x)\right) + \|c_{k, 0}\|_1}{\kappa_k}\right] \tfrac{1}{\epsilon^2},
    \end{align*}
    where $\kappa_k$ is a finite constant based on $F_{S_k}(x)$ and the term within the right-hand side brackets is bounded as the set of iterates $\chi$ is a bounded set, thereby satisfying \cref{ass:sublinear_solver}. This ensures that \cref{alg:Equality_Constrained_RA_SQP} achieves the complexity results stated in \cref{th:EQ_work_complexity}.
    Similar results can be obtained when inexact search direction solutions are used, as shown in \cite{berahas2022adaptive,curtis2021inexact,byrd2008inexact,byrd2010inexact}.
\end{proof}

\cref{th:EQ_SQP_well_posed} establishes that \cref{alg:Equality_Constrained_RA_SQP} achieves the optimal gradient ($\Ocal\left(\epsilon^{-4}\right)$) and \texttt{SQP} linear system subproblem complexity ($\Ocal\left(\epsilon^{-2}\right)$) for stochastic problems with equality constraints.
We now analyze the convergence of \cref{alg:Equality_Constrained_RA_SQP} under the alternate termination criteria proposed in \cref{sec:SQP_based_algorithm_eq}. 
To establish the results in \cref{sec:General_Deterministic_Solver_theory}, we need to show an error bound analogous to \cref{lem:main_error_bound} and adapt \cref{cond:eq_norm_test} to align with the optimality measure used in the 
termination criteria.

We first consider the termination criterion \eqref{eq:termination_criterion_d}, which is based on the norm of the search direction.
We define $(\hat{d}_{k,j},\hat{\delta}_{k,j})$ as the exact solutions to the \texttt{SQP} linear system \eqref{eq:eq_sqp_system}, and $(d_{k,j}^{true},\delta_{k,j}^{true})$ as the exact solutions to the \texttt{SQP} linear system \eqref{eq:eq_sqp_system} when the KKT error of the true problem is used instead of the KKT error of the subsampled problem.
If $d_{k, j}^{true} = 0$, then $x_{k, j}$ is a first-order stationary point for the true problem \eqref{eq:intro_problem}.
\begin{lemma} \label{lem:error_inexact_d}
    Suppose \cref{ass:SQP_eq_assumptions} holds. Then for all $k \geq 0, \,\, j \geq 0$, if $\kappa_T$ in \eqref{eq:eq_inexactness_condition_1}, and $\epsilon_{feas}$ and $\epsilon_{opt}$ in \eqref{eq:eq_inexactness_condition_2} are chosen to be such that $\tfrac{\kappa_T}{\kappa_{HJ}} < \frac{1}{2}$ and $\tfrac{\kappa_J (\epsilon_{feas} + \epsilon_{opt})}{\kappa_{HJ}} < 1$, there exists $\eta_d \in (0,1)$
    such that the search direction generated by \cref{alg:Equality_Constrained_RA_SQP} satisfies,
    \begin{align*}
        \|d_{k, j} - \hat{d}_{k, j}\| \leq \eta_d \|\hat{d}_{k, j}\|.
    \end{align*}
\end{lemma}
\begin{proof}
    If the \texttt{SQP} linear system \eqref{eq:eq_sqp_system} is solved exactly, the relation is trivially satisfied with $\eta_d = 0$. When $d_{k, j}$ is an inexact solution, it satisfies either \eqref{eq:eq_inexactness_condition_1} or \eqref{eq:eq_inexactness_condition_2}. By the definitions of $d_{k, j}$ and $\hat{d}_{k, j}$ from \eqref{eq:eq_sqp_system}, we have
    $\begin{bmatrix}
        H_{k, j} & J_{k, j}^T \\ J_{k, j} & 0
    \end{bmatrix}
    \begin{bmatrix}
        d_{k, j} - \hat{d}_{k, j} \\ \delta_{k, j} - \hat{\delta}_{k, j}
    \end{bmatrix}
    =
    \begin{bmatrix}
        \rho_{k, j} \\ r_{k, j}
    \end{bmatrix}$. Thus, by \cref{ass:SQP_eq_assumptions},
    \begin{align*}
        \|d_{k, j} - \hat{d}_{k, j}\|
        \leq
        \left\|\begin{bmatrix}
            d_{k, j} - \hat{d}_{k, j} \\ \delta_{k, j} - \hat{\delta}_{k, j}
        \end{bmatrix}\right\|
        =
        \left\|\begin{bmatrix}
            H_{k, j} & J_{k, j}^T \\ J_{k, j} & 0
        \end{bmatrix}^{-1}
        \begin{bmatrix}
            \rho_{k, j} \\ r_{k, j}
        \end{bmatrix}\right\|
        \leq
        \kappa_{HJ}^{-1}
        \left\|
        \begin{bmatrix}
            \rho_{k, j} \\ r_{k, j}
        \end{bmatrix}\right\|.
    \end{align*}
    If the inexactness conditions \eqref{eq:eq_inexactness_condition_1} are satisfied,
    \begin{align*}
        \|d_{k, j} - \hat{d}_{k, j}\|
        &
        \leq
        \tfrac{\kappa_T}{\kappa_{HJ}}
        \min\left\{\left\|T_{S_k}(x_{k, j}, \lambda_{k, j})\right\|, \|d_{k, j}\|\right\}
        \leq
        \tfrac{\kappa_T}{\kappa_{HJ}} \|d_{k, j}\| \leq \eta_d \|\hat{d}_{k, j}\|,
    \end{align*}
    where $\eta_d = \tfrac{\tfrac{\kappa_T}{\kappa_{HJ}} }{1 - \tfrac{\kappa_T}{\kappa_{HJ}}} < 1$ as $\tfrac{\kappa_T}{\kappa_{HJ}} < \tfrac{1}{2}$. 
    If the inexactness conditions \eqref{eq:eq_inexactness_condition_2} are satisfied,
    \begin{align*}
        \|d_{k, j} - \hat{d}_{k, j}\|
        &
        \leq
        \tfrac{\|\rho_{k, j}\| + \|r_{k, j}\|}{\kappa_{HJ}}
        \leq
        \tfrac{\epsilon_{feas} + \epsilon_{opt}}{\kappa_{HJ}} \|c_{k, j}\|
        =
        \tfrac{\epsilon_{feas} + \epsilon_{opt}}{\kappa_{HJ}} \|J_{k, j } \hat{d}_{k, j}\|
        \leq 
        \eta_d \|\hat{d}_{k, j}\|,
    \end{align*}
    where $\eta_d = \tfrac{\kappa_J(\epsilon_{feas} + \epsilon_{opt})}{\kappa_{HJ}} < 1$ as specified.
\end{proof}

We now establish an error bound similar to \cref{lem:main_error_bound}, under the termination criterion \eqref{eq:termination_criterion_d}. 

\begin{lemma} \label{lem:step_norm_error_bound}
    Suppose \cref{ass:SQP_eq_assumptions} holds. Then for all $k \geq 0$, the true search directions for the outer iterates generated by \cref{alg:Equality_Constrained_RA_SQP} with termination criterion \eqref{eq:termination_criterion_d} satisfy
    \begin{align*}
        \|d_{k, N_k}^{true}\|
        &\leq \gamma_k \tfrac{1 + \eta_d}{1 - \eta_d} \|d_{k-1, N_{k-1}}^{true}\| + \tfrac{\epsilon_k}{1 - \eta_d}  +  \tfrac{1 + \eta_d}{1 - \eta_d}\kappa_{HJ}^{-1}\|\nabla f(x_{k, N_k}) - g_{S_k}(x_{k, N_k})\| \\
        &\quad + \gamma_k \tfrac{1 + \eta_d}{1 - \eta_d}\kappa_{HJ}^{-1}\|\nabla f(x_{k, 0}) - g_{S_k}(x_{k, 0})\|.
    \end{align*}
    For the expectation problem \eqref{eq:intro_stoch_error_obj}, for all $k \geq 0$
    \begin{align*}
        \Embb[\|d_{k, N_k}^{true}\| | \Fcal_k]
        &\leq  \gamma_k \tfrac{1 + \eta_d}{1 - \eta_d} \|d_{k-1, N_{k-1}}^{true}\| + \tfrac{\Embb\left[\epsilon_k | \Fcal_k\right]}{1 - \eta_d} + \kappa_{HJ}^{-1} \tfrac{1 + \eta_d}{1 - \eta_d} \Embb[\|\nabla f(x_{k, N_k}) - g_{S_k}(x_{k, N_k})\| | \Fcal_k] \\
        &\quad +  \gamma_k \tfrac{1 + \eta_d}{1 - \eta_d} \kappa_{HJ}^{-1} \Embb[\|\nabla f(x_{k, 0}) - g_{S_k}(x_{k, 0})\| | \Fcal_k]. 
    \end{align*}
\end{lemma}
\begin{proof}
    We start with a bound on the error in the search direction due to subsampled gradients. By the definition of $d_{k, j}^{true}$ and $\hat{d}_{k, j}$ from \eqref{eq:eq_sqp_system},
    \begin{align*}
        \begin{bmatrix}
            H_{k, j} & J_{k,j}^T \\
            J_{k,j} & 0
        \end{bmatrix}
        \begin{bmatrix}
            \hat{d}_{k,j} - d_{k,j}^{true} \\
            \hat{\delta}_{k,j} - \delta_{k,j}^{true}
        \end{bmatrix}
        = -
        \begin{bmatrix}
            g_{S_k}(x_{k,j}) - \nabla f(x_{k,j}) \\
            0
        \end{bmatrix}.
    \end{align*}
    By \cref{ass:SQP_eq_assumptions},
    \begin{align*}
        \|\hat{d}_{k,j} - d_{k,j}^{true}\| &\leq
        \left\|\begin{bmatrix}
            \hat{d}_{k,j} - d_{k,j}^{true} \\
            \hat{\delta}_{k,j} - \delta_{k,j}^{true}
        \end{bmatrix}\right\|
        \leq
        \kappa_{HJ}^{-1} \|\nabla f(x_{k,j}) - g_{S_k}(x_{k,j})\|.
    \end{align*}
    At the end of outer iteration $k$, the norm of the true search direction can be bounded as,
    \begin{align*}
        \|d_{k, N_k}^{true}\|
        &\leq \|d_{k, N_k}\| + \|\hat{d}_{k, N_k} - d_{k, N_k}\| + \|d_{k, N_k}^{true} - \hat{d}_{k, N_k}\| \\
        &\leq  \gamma_k \|d_{k, 0}\| + \epsilon_k + \eta_d \|\hat{d}_{k, N_k}\| + \|d_{k, N_k}^{true} - \hat{d}_{k, N_k}\|\\ 
        &\leq \gamma_k \|\hat{d}_{k, 0}\| + \gamma_k \|d_{k, 0} - \hat{d}_{k, 0}\| + \epsilon_k + \eta_d \|d_{k, N_k}^{true}\| + \eta_d \|\hat{d}_{k, N_k} - d_{k, N_k}^{true}\| + \|d_{k, N_k}^{true} - \hat{d}_{k, N_k}\| \\ 
        &\leq \gamma_k (1 + \eta_d) \left(\|\hat{d}_{k, 0} - d_{k, 0}^{true}\| + \|d_{k, 0}^{true}\|\right) + \epsilon_k +  (1 + \eta_d)\|d_{k, N_k}^{true} - \hat{d}_{k, N_k}\| + \eta_d \|d_{k, N_k}^{true}\|\\ 
        &= \gamma_k \tfrac{1 + \eta_d}{1 - \eta_d} \|d_{k - 1, N_{k-1}}^{true}\| + \tfrac{\epsilon_k}{1 - \eta_d} +  \tfrac{1 + \eta_d}{1 - \eta_d}\|d_{k, N_k}^{true} - \hat{d}_{k, N_k}\|  + \gamma_k \tfrac{1 + \eta_d}{1 - \eta_d}\|\hat{d}_{k, 0} - d_{k, 0}^{true}\|,
    \end{align*}
    where the second inequality follows from the termination criterion \eqref{eq:termination_criterion_d} and \cref{lem:error_inexact_d}, the fourth inequality follows form \cref{lem:error_inexact_d} and the final equality follows from subtracting $\eta_d \|d^{true}_{k, N_k}\|$ from both sides and dividing by $1 -\eta_d$, and from $d_{k, 0}^{true} = d_{k-1, N_{k-1}}^{true}$ as $x_{k, 0} = x_{k-1, N_{k-1}}$, $H_{k, 0} = H_{k-1, N_{k-1}}$ and $\lambda_{k, 0} = \lambda_{k-1, N_{k-1}}$.
    Employing the search direction error bound in the above inequality yields the desired deterministic bound. 
    Taking the conditional expectation given $\Fcal_k$ yields the expectation bound as $d_{k-1, N_{k-1}}^{true}$ is defined under $\Fcal_k$.
\end{proof}
Using \cref{lem:step_norm_error_bound}, we can establish results analogous to \cref{th:Eq_convergence} by incorporating the new scaling factor $\kappa_{HJ}^{-1} \left(\tfrac{1 + \eta_d}{1 - \eta_d}\right)$ and requiring the termination criterion parameter to satisfy $\gamma_k < \tfrac{1 - \eta_d}{1 + \eta_d}$ for all $k \geq 0$. When exact solutions to the \texttt{SQP} linear system are used, $\eta_d = 0$, and it suffices to have $\{\gamma_k\} < 1$. We now modify \cref{cond:eq_norm_test} to align the adaptive sampling strategy with termination criterion \eqref{eq:termination_criterion_d}.
\begin{condition} \label{cond:eq_norm_test_d}
    At the beginning of outer iteration $k \geq 0$ in \cref{alg:Equality_Constrained_RA_SQP} with termination criterion \eqref{eq:termination_criterion_d}, the batch size $|S_k|$ is selected such that:
    \begin{enumerate}
        \item For the finite-sum problem \eqref{eq:intro_deter_error_obj}: With constants $\theta, \hat{\theta}, a \geq 0$ and $\beta \in (0, 1)$, 
        \begin{align*}
            \|\nabla f(x_{k, 0}) - g_{S_k}(x_{k, 0})\|^2 &\leq \theta^2 \|d_{k, 0}^{true}\|^2 + a^2 \beta^{2k},\\
            \|\nabla f(x_{k, N_k}) - g_{S_k}(x_{k, N_k})\|^2 &\leq \hat{\theta}^2 \left(\theta^2 \|d_{k, 0}^{true}\|^2 + a^2 \beta^{2k}\right).
        \end{align*}
        \item For the expectation problem \eqref{eq:intro_stoch_error_obj}: With constants $\tilde{\theta}, \tilde{a} \geq 0$ and $\tilde{\beta} \in (0, 1)$, 
        \begin{align*}
            \Embb\left[\|\nabla f(x_{k, 0}) - g_{S_k}(x_{k, 0})\|^2 | \Fcal_k\right] \leq \tilde{\theta}^2 \|d_{k, 0}^{true}\|^2 + \tilde{a}^2 \tilde{\beta}^{2k}.
        \end{align*}
    \end{enumerate}
\end{condition}
Following the approach of \cref{th:EQ_outer_iter_complexity}, we can establish linear convergence for the true search direction norm across outer iterations under \cref{ass:SQP_eq_assumptions}, termination criterion \eqref{eq:termination_criterion_d} and 
\cref{cond:eq_norm_test_d}. The result is formalized in the following corollary.
\begin{theorem} \label{th:EQ_outer_iter_complexity_step_norm}
    Suppose Assumptions \ref{ass:gradient_errors} and \ref{ass:SQP_eq_assumptions} hold and that the batch size sequence $\{|S_k|\}$ is chosen to satisfy \cref{cond:eq_norm_test_d} in \cref{alg:Equality_Constrained_RA_SQP} with termination criterion \eqref{eq:termination_criterion_d}.
    \begin{enumerate}
        \item For the finite-sum problem \eqref{eq:intro_deter_error_obj}: For all $k \geq 0$, if termination criterion parameters are chosen as $0 \leq \gamma_k \leq \gamma < \tfrac{1 - \eta_d}{1 + \eta_d}$, $\epsilon_k = \omega \|\nabla f(x_{k, 0}) - g_{S_k}(x_{k, 0})\| + \hat{\omega}\beta^k$ with $\omega, \hat{\omega} \geq 0$ and \cref{cond:eq_norm_test_d} parameters are chosen such that $a_1 = \tfrac{1}{1 - \eta_d}\left[ \gamma(1 + \eta_d) + \theta\left(\omega + (\gamma + \hat{\theta})\tfrac{1 + \eta_d}{\kappa_{HJ}}\right) \right] < 1$, then, the true search direction norm converges at a linear rate across outer iterations, i.e., 
        \begin{align*}
            \|d_{k, N_k}^{true}\|
            \leq \max \{a_1 + \nu, \beta\}^{k+1} \max\left\{\|d_{0, 0}^{true}\|, \tfrac{a_2}{\nu}\right\},
        \end{align*}
        where $a_2 = \tfrac{a}{1 - \eta_d}\left(\omega + (\gamma + \hat{\theta})\tfrac{1 + \eta_d}{\kappa_{HJ}}\right) + \tfrac{\hat{\omega}}{1 - \eta_d}$ and $\nu > 0$ such that $a_1 + \nu < 1$.
        \item For the expectation problem \eqref{eq:intro_stoch_error_obj}: For all $k \geq 0$, if \cref{ass:Variance_regularity} holds, the termination criterion parameters are chosen as $0 \leq \gamma_k \leq \tilde{\gamma} < \tfrac{1 - \eta_d}{1 + \eta_d}$, $\epsilon_k = \tilde{\omega}\sqrt{\tfrac{\Var(\nabla F(x_{k, 0}) | \Fcal_k)}{|S_k|}}$ where $\tilde{\omega} \geq 0$ and \cref{cond:eq_norm_test_d} parameters are chosen such that $\tilde{a}_1 = \tfrac{1}{1- \eta_d}\left[\tilde{\gamma} (1 + \eta_d) + \tilde{\theta} \left(\tilde{\omega} + \tfrac{1 + \eta_d}{\kappa_{HJ}}\left(\tfrac{\epsilon_G + \kappa_g}{\kappa_{\sigma}} + \tilde{\gamma}\right)\right) \right] < 1$, then, the expected true search direction norm converges at a linear rate across outer iterations, i.e., 
        \begin{align*}
            \Embb\left[\|d_{k, N_k}^{true}\| \right]
            \leq \max \{\tilde{a}_1 + \tilde{\nu}, \tilde{\beta}\}^{k+1} \max \left\{\|d_{0, 0}^{true}\|, \tfrac{\tilde{a}_2}{\tilde{\nu}}\right\},
        \end{align*}
        where $\tilde{a}_2 = \tfrac{\tilde{a}}{1 - \eta_d} \left[\tilde{\omega} + \tfrac{1 + \eta_d}{\kappa_{HJ}}\left(\tfrac{\epsilon_G + \kappa_g}{\kappa_{\sigma}} + \tilde{\gamma}\right) \right]$ and $\tilde{\nu} > 0$ such that $\tilde{a}_1 + \tilde{\nu} < 1$.
    \end{enumerate}
\end{theorem}
\begin{proof}
    The proof follows the same procedure as the proof of \cref{th:EQ_outer_iter_complexity}, with adjustments to account for the changes in \cref{lem:step_norm_error_bound} relative to \cref{lem:main_error_bound}.
    Specifically, it incorporates the scaling factor $\kappa_{HJ}^{-1} \left(\tfrac{1 + \eta_d}{1 - \eta_d}\right)$ and the condition ${\gamma_k} < \tfrac{1 - \eta_d}{1 + \eta_d}$. 
    The complete proof is presented in \cref{appendix:proofs}.
\end{proof}
Using the same approach, we can also establish a result analogous to \cref{cor:Eq_geometric_batch_increase_outer_iter_complexity} for the case of a geometrically increasing batch size sequence. This variant removes several of the requirements imposed in \cref{th:EQ_outer_iter_complexity_step_norm}, while still guaranteeing convergence under milder conditions.

We now analyze the termination criterion \eqref{eq:termination_criterion_delta_l}, employing the merit function model decrease. We establish convergence under this criterion by showing that it is equivalent, up to constant factors, to the search direction norm-based termination criterion \eqref{eq:termination_criterion_d}.

\begin{lemma}
    Suppose \cref{ass:SQP_eq_assumptions} holds. Then, there exists $\hat{\gamma} < 1$ such that termination criterion \eqref{eq:termination_criterion_delta_l} with $\{\gamma_k\} < \hat{\gamma}$ satisfies termination criterion \eqref{eq:termination_criterion_d} for all $k\geq 0$ in \cref{alg:Equality_Constrained_RA_SQP}.
\end{lemma}
\begin{proof}
    Under termination criterion \eqref{eq:termination_criterion_delta_l},
    \begin{align*}
        \|d_{k, N_k}\|^2 &\leq \kappa_{d, \Delta l} \Delta l_{S_k} (x_{k, N_k}, \tau_{k, N_k}, d_{k,N_k}) \\
        &\leq \kappa_{d, \Delta l} \gamma_k \min\{\Delta l_{S_k} (x_{k, 0}, \tau_{k, 0}, d_{k,0}), \kappa_d \|d_{k, 0}\|^2\} + \kappa_{d, \Delta l}  \epsilon_k \\
        &\leq \kappa_{d, \Delta l} \gamma_k \kappa_d \|d_{k, 0}\|^2 + \kappa_{d, \Delta l} \epsilon_k,
    \end{align*}
    where the first inequality follows from \cite[Lemma 4.11]{berahas2022adaptive} with $\kappa_{d, \Delta l} > 0$, since, by \cref{lem:error_inexact_d} and \cref{ass:SQP_eq_assumptions}, the error $\|d_{k, N_k} - \hat{d}_{k, N_k}\|$ is bounded and from \cite[Lemma 4.7]{byrd2008inexact}, the merit parameter is bounded below. Therefore, termination criterion \eqref{eq:termination_criterion_d} is satisfied with parameters $(\sqrt{\kappa_{d, \Delta l} \gamma_k \kappa_d}, \sqrt{\kappa_{d, \Delta l} \epsilon_k})$, which satisfy the conditions for convergence if $\gamma_k$ is sufficiently small.
\end{proof}

Thus, we have established the convergence of \cref{alg:Equality_Constrained_RA_SQP} under termination criterion \eqref{eq:termination_criterion_delta_l}. Similarly, \cref{cond:eq_norm_test_d} can be adapted for this termination criterion \eqref{eq:termination_criterion_delta_l}. We first define $\tau_{k, 0}^{true}$ as the merit parameter obtained in the first inner iteration of outer iteration $k$ if the true problem gradient and the true search direction $d_{k, 0}^{true}$ are used for the update. We also define the true problem merit function model decrease as
\begin{align*}
    \Delta l (x_{k, 0}, \tau_{k, 0}^{true}, d_{k,0}^{true}) &= -\tau_{k, 0}^{true} \nabla f(x_{k, 0})^Td_{k,0}^{true} + \|c_{k, 0}\|_1 - \|c_{k, 0} + J_{k, 0}d_{k, 0}^{true}\|_1.
\end{align*}

Next, we present the sampling condition to accompany termination criterion \eqref{eq:termination_criterion_delta_l}.
\begin{condition} \label{cond:eq_norm_test_delta_l}
    At the beginning of outer iteration $k \geq 0$ in \cref{alg:Equality_Constrained_RA_SQP} with termination criterion \eqref{eq:termination_criterion_delta_l}, the batch size $|S_k|$ is selected such that:
    \begin{enumerate}
        \item For the finite-sum problem \eqref{eq:intro_deter_error_obj}: With constants $\theta, \hat{\theta}, a \geq 0$ and $\beta \in (0, 1)$, 
        \begin{align*}
            \|\nabla f(x_{k, 0}) - g_{S_k}(x_{k, 0})\|^2 &\leq \theta^2 \min\{\Delta l (x_{k, 0}, \tau_{k, 0}^{true}, d_{k, 0}^{true}), \kappa_d \|d_{k, 0}^{true}\|^2\} + a^2 \beta^{2k},\\
            \|\nabla f(x_{k, N_k}) - g_{S_k}(x_{k, N_k})\|^2 &\leq \hat{\theta}^2 \left(\theta^2 \min\{\Delta l (x_{k, 0}, \tau_{k, 0}^{true}, d_{k, 0}^{true}), \kappa_d \|d_{k, 0}^{true}\|^2\} + a^2 \beta^{2k}\right).
        \end{align*}
        \item For the expectation problem \eqref{eq:intro_stoch_error_obj}: With constants $\tilde{\theta}, \tilde{a} \geq 0$ and $\tilde{\beta} \in (0, 1)$, 
        \begin{align*}
            \Embb\left[\|\nabla f(x_{k, 0}) - g_{S_k}(x_{k, 0})\|^2 | \Fcal_k\right] \leq \tilde{\theta}^2 \min\{\Delta l (x_{k, 0}, \tau_{k, 0}^{true}, d_{k, 0}^{true}), \kappa_d \|d_{k, 0}^{true}\|^2\} + \tilde{a}^2 \tilde{\beta}^{2k}.
        \end{align*}
    \end{enumerate}
\end{condition}
It is straightforward to observe that satisfying \cref{cond:eq_norm_test_delta_l} also implies \cref{cond:eq_norm_test_d}. If the parameter $\theta$ (and $\tilde{\theta}$) in \cref{cond:eq_norm_test_delta_l} is chosen to be sufficiently small so as to meet the requirements in \cref{th:EQ_outer_iter_complexity_step_norm}, then linear convergence across outer iterations can be achieved while using the termination criterion \eqref{eq:termination_criterion_delta_l}, under \cref{cond:eq_norm_test_delta_l}.

We established the convergence properties of \cref{alg:Equality_Constrained_RA_SQP} under the alternate termination criterion \eqref{eq:termination_criterion_d} and \eqref{eq:termination_criterion_delta_l} with inexact solutions to the \texttt{SQP} linear system \eqref{eq:eq_sqp_system} that satisfy either \eqref{eq:eq_inexactness_condition_1} or \eqref{eq:eq_inexactness_condition_2}. 
When the \texttt{SQP} linear system is solved exactly, $\eta_d = 0$ in \cref{lem:error_inexact_d} reducing the conditions on parameters under the alternate termination criterion. 
The results presented in this section can be derived for termination criteria based on other optimality metics, provided an error bound analogous to \cref{lem:main_error_bound} can be established.
\section{Inequality Constrained Problems} \label{sec:Inequality_constrained}

In this section, we present our proposed 
method for solving problem \eqref{eq:intro_problem} with general nonlinear equality and inequality constraints. 
We begin by describing the challenges introduced by the presence of general nonlinear inequality constraints and highlighting the differences from the framework proposed in \cref{sec:Equality_constrained}.
We describe the proposed algorithm in \cref{sec:Inequality_constrained_algorithm} and provide theoretical guarantees in \cref{sec:Inquality_constrained_analysis}.

As mentioned in Section~\ref{sec:Introduction}, $x^*$ is a first-order stationary point for problem \eqref{eq:intro_problem} with general nonlinear equality and inequality constraints if there exist $\lambda^*_E \in \mathbb{R}^{m_E}$ and $\lambda^*_I \in \mathbb{R}^{m_I}$ such that,
\begin{equation*}
    \nabla_x \Lcal (x^*, \lambda^*_E, \lambda^*_I) = 0, \quad
    c_E(x^*) = 0, \quad c_I(x^*) \leq 0, \quad c_I(x^*) \odot \lambda^*_I = 0, \quad \lambda^*_I \geq 0.
\end{equation*}
While one could extend the general framework from \cref{sec:General_Deterministic_Solver_alg} by incorporating the complementary slackness error from \eqref{eq:KKT_conditions} into the termination criterion \eqref{eq:termination_criterion_eq}, instead we directly present a solver-specific termination criterion due to the various challenges associated with the general constrained setting.
First, incorporating the complementary slackness error would add a third term in termination criterion \eqref{eq:termination_criterion_eq}, exacerbating the scaling issue discussed in \cref{sec:SQP_based_algorithm_eq}.
Second, evaluating the termination criterion would require good dual variable estimates for the subsampled problem, something not readily provided or updated by deterministic methods when dealing with general nonlinear constraints.
Unlike \cref{sec:Equality_constrained}, if the deterministic method in the inner loop does not provide good dual variable estimates, computing them requires solving a linear or quadratic program due to the presence of the complimentary slackness conditions in \eqref{eq:KKT_conditions}, which is significantly more expensive than solving the least squares problem in \cref{sec:Equality_constrained}.
Thus, to address these challenges, in order to establish theoretical results and achieve good empirical performance, we propose a framework tailored to a deterministic solver with a termination criterion that leverages solver-specific quantities. Given our goal of handling general nonlinear constraints, we adopt an \texttt{SQP}-based deterministic algorithm for the inner loop.

An important consideration when \texttt{SQP}-based methods are used for general nonlinear constraints are infeasible stationary points. The convergence of \texttt{SQP}-based methods towards a feasible solution can be characterized as reaching a stationary point of an optimization problem that minimizes the constraint violation, i.e., $\min_{x \in \Rmbb^n} \frac{1}{2} \|c_E(x)\|^2 + \frac{1}{2} \|[c_I(x)]_+\|^2$. In the case of only equality constraints, LICQ conditions ensure that every stationary point to this problem is feasible. However, no such guarantees exist for general nonlinear constraints. If such a stationary point is reached and found to be infeasible, termination is necessary, returning an infeasible stationary point. The precise definition of this concept varies by algorithm and is provided later.

When solving a problem with general nonlinear constraints using \texttt{SQP} methods, one must solve subproblem \eqref{eq:SQP_quad_model} at each iteration. This subproblem is more computationally expensive than the \texttt{SQP} linear system \eqref{eq:eq_sqp_system} in the equality constraint only setting, and its feasibility is not guaranteed even when the original problem is feasible, as demonstrated in \cite[Example 19.2]{nocedal2006numerical}. Few works exist for the general constrained stochastic setting that handle this issue.
Among these, \cite{na2023inequality} bypassed the issue of infeasibility by assuming that every \texttt{SQP} subproblem \eqref{eq:SQP_quad_model} is feasible, and,  subsequent works have tackled this challenge using techniques such as step decomposition \cite{curtis2023sequential} and robust subproblems \cite{qiu2023sequential}, which ensure subproblem feasibility while identifying infeasible stationary points. The latter two are techniques adapted from determninistic constrained optimization.
Naturally, we wish to incorporate such \texttt{SQP} methods within the \texttt{RA} framework. We use \texttt{SQP} methods with robust subproblems \cite{burke1989robust,omojokun1989trust,qiu2023sequential}, instead of a slack variable reformulation of the constraints, i.e., $s \in \Rmbb^{m_I}$, $c_I(x) + s = 0$ and $s \geq 0$, as is common, e.g., \cite{nocedal2006numerical,curtis2023sequential,powell2006fast}. Within the \texttt{RA} framework, we observe that the slack variable reformulation incurs significant computational costs due to the need to update slack variables and track the active set, even after the inner loop achieves sufficient optimality for termination. In contrast, robust-\texttt{SQP} methods eliminate these inefficiencies while preserving strong convergence properties.

\subsection{Algorithm Description} \label{sec:Inequality_constrained_algorithm}
In this section, we describe the proposed 
method for solving problems with stochastic objective functions and general nonlinear constraints based on \cref{alg:RA_base} and the robust-\texttt{SQP} method. Specifically, we adopt the robust-\texttt{SQP} algorithm \cite{burke1989robust} that uses a two-step approach. The two-step approach first minimizes the linearized constraint violation and then proceeds to compute a search direction that minimizes a quadratic approximation of the objective function while maintaining the constraint violation achieved in the first step. This ensures the subproblems for the search direction calculation are feasible, contrary to subproblem \eqref{eq:SQP_quad_model}, and are able to detect infeasible stationary points via the constraint minimization problem in the first step. 
We now detail the inner loop update mechanism and the termination criterion for the inner loop. 
We adopt notation similar to that in \cref{sec:Equality_constrained}, where $x_{k, j}$ denotes the iterate at the outer iteration $k$ and inner iteration $j$. 

In outer iteration $k$, a set of samples $S_k$ is obtained to form the subsampled problem \eqref{eq:subsampled_problem}. To update the iterate in inner iteration $j$, we first compute a search direction $p_{k, j}$ that minimizes the linearized constraint violation, based on a chosen measure of constraint violation. This violation can be measured using either the $l_{\infty}$ or the $l_1$ norm, resulting in the following linear programs ($l_\infty$ on the left and $l_1$ on the right),
\begin{equation}
\begin{array}{ccc}
\begin{aligned}
    \min_{\substack{y \in \Rmbb, \\ p_{k, j} \in \Rmbb^n}} &y \\
    s.t. \,\, &c_E(x_{k,j}) + \nabla c_E^T(x_{k, j}) p_{k, j} \geq -ye_{m_E} \\
    &c_E(x_{k,j}) + \nabla c_E^T(x_{k, j}) p_{k, j} \leq y e_{m_E}  \\
    &c_I(x_{k, j}) + \nabla c_I^T(x_{k, j}) p_{k, j} \leq y e_{m_I} \\
    &y \geq 0, \|p_{k, j}\|_{\infty} \leq \sigma_p
\end{aligned}
& \text{or} &
\begin{aligned}
    \min_{\substack{y_{E} \in \Rmbb^{m_E}, \\ y_{I} \in \Rmbb^{m_I}, \\ p_{k, j} \in \Rmbb^n}} &e_{m_E}^T y_E + e_{m_I}^T y_I\\
    s.t. \,\, &c_E(x_{k,j}) + \nabla c_E^T(x_{k, j}) p_{k, j}  \geq - y_E  \\
    &c_E(x_{k,j}) + \nabla c_E^T(x_{k, j}) p_{k, j}  \leq y_E  \\
    &c_I(x_{k, j}) + \nabla c_I^T(x_{k, j}) p_{k, j} \leq y_I  \\
    &y_E \geq 0, y_I \geq 0, \|p_{k, j}\|_{1} \leq \sigma_p,
\end{aligned}
\end{array}
\label{eq:constraint_prob}
\end{equation}
where $y$ is the minimized linearized constraint violation when measured with the $l_\infty$ norm, ($y_E$, $y_I$) represent the minimized linearized constraint violation components when measured with the $l_1$ norm, and $\sigma_p$ is the bound on the norm of the search direction. The linear program corresponding to the $l_\infty$ norm involves fewer variables and constraints compared to the one based on the $l_1$ norm. However, the $l_1$ norm provides a more precise measure of constraint violation which can potentially lead to improved performance.
Although other norms can be used to measure constraint violation, they often lead to more challenging subproblems, for example, the Euclidean norm results in quadratically constrained quadratic programs (QCQPs). 
Thus, we restrict our focus to the $l_{\infty}$ and $l_1$ norms in this work.

Next, a quadratic program is solved to compute the search direction $d_{k,j}$, which minimizes a quadratic approximation of the objective function at $x_{k, j}$, while preserving the constraint violation achieved by $p_{k, j}$. The resulting quadratic programs ($l_\infty$ on the left and $l_1$ on the right) are
\begin{equation}
\begin{array}{ccc}
\begin{aligned}
    \min_{d_{k,j} \in \Rmbb^n} &g_{S_k}(x_{k, j})^T d_{k,j} + \tfrac{1}{2} d_{k,j}^T H_{k,j} d_{k,j} \\
    s.t. \,\, &c_E(x_{k,j}) + \nabla c_E^T(x_{k, j}) d_{k,j} \geq - y e_{m_E}  \\
    &c_E(x_{k,j}) + \nabla c_E^T(x_{k, j}) d_{k,j} \leq y e_{m_E}  \\
    &c_I(x_{k, j}) + \nabla c_I^T(x_{k, j}) d_{k,j} \leq y e_{m_I} \\
    &\|d_{k,j}\|_{\infty} \leq \sigma_d 
\end{aligned}
& \text{or} &
\begin{aligned}
    \min_{d_{k,j} \in \Rmbb^n} &g_{S_k}(x_{k, j})^T d_{k,j} + \tfrac{1}{2} d_{k,j}^T H_{k,j} d_{k,j} \\
    s.t. \,\,\, &c_E(x_{k, j}) + \nabla c_E^T(x_{k, j}) d_{k,j} \geq - y_E  \\
    &c_E(x_{k, j}) + \nabla c_E^T(x_{k, j}) d_{k,j} \leq y_E  \\
    &c_I(x_{k, j}) + \nabla c_I^T(x_{k, j}) d_{k,j} \leq y_I \\
    &\|d_{k,j}\|_{1} \leq \sigma_d,
\end{aligned}
\end{array}
\label{eq:opt_prob}
\end{equation}
where $y$, ($y_E,y_I$) are the outputs from solving \eqref{eq:constraint_prob}, $H_{k, j}$ is an approximation of the Hessian of the Lagrangian, and $\sigma_d$ is the bound on the norm of the search direction.
The bound constraint parameters $\sigma_p$ and $\sigma_d$ are user defined constants such that $0 < \sigma_p < \sigma_d < \infty$, ensuring that \eqref{eq:opt_prob} remains feasible. These finite bound constraints on the search direction are essential for the identification of infeasible stationary points and to establish convergence theory with $\{\|d_{k, j}\|\} \rightarrow 0$, as demonstrated in \cite[Example 2.1]{burke1989robust}. 
While $\sigma_p$ and $\sigma_d$ can be constants, in practice they are often set proportional to the constraint violation of the current iterate, following the approach in \cite{qiu2023sequential, curtis2010matrix}; further details are provided in \cref{sec:Numerical_experiments}. 

A merit function is employed to balance the two possibly competing goals of decreasing the objective function and reducing constraint violation. Depending on the norm used to measure the constraint violation, the merit function $\phi_{S_k} : \Rmbb^n \times \Rmbb \to \Rmbb$ is defined as ($l_\infty$ on left and $l_1$ on right),
\begin{equation} \label{eq:merit_function_ineq}
    \phi_{S_k} (x, \tau) = \tau F_{S_k}(x)  + \left\| \begin{matrix}
    c_E(x) \\ [c_I(x)]_+
    \end{matrix} \right\|_{\infty} 
    \quad \text{or,} \quad
    \phi_{S_k} (x, \tau) = \tau F_{S_k}(x)  + \left\| \begin{matrix}
    c_E(x) \\ [c_I(x)]_+
    \end{matrix} \right\|_{1} ,
\end{equation}
where $\tau > 0$ is the merit parameter. Note that in \eqref{eq:merit_function_ineq} the penalty parameter $\tau$ is on the objective function whereas in many robust-\texttt{SQP} methods \cite{burke1989robust,omojokun1989trust,qiu2023sequential} the penalty parameter $\rho$ is on the constraint violation term. 
The two forms are equivalent for appropriate choices of these parameters, and we prefer the former to maintain consistency with the merit function defined in \cref{sec:SQP_based_algorithm_eq} for the equality constrained setting.
To update the merit parameter, we define the predicted reduction in a linear model of the merit function $\Delta l_{S_k} : \Rmbb^n \times \Rmbb^n \times \Rmbb \to \Rmbb$ as,
\begin{equation} \label{eq:merit_function_model_ineq}
    \Delta l_{S_k} (x_{k, j}, d_{k, j}, \tau_{k,j}) = - \tau_{k,j} g_{S_k}(x_{k, j})^Td_{k, j} +  \Delta_c, 
\end{equation}
where $\Delta_c$ denotes the improvement in the linearized constraint violation, defined as ($l_\infty$ on the left and $l_1$ on the right),
\begin{align}
    \Delta_c = \left\| \begin{matrix}
    c_E(x_{k, j}) \\ [c_I(x_{k, j})]_+
    \end{matrix} \right\|_{\infty} 
    - y \quad \text{or,} \quad
    \Delta_c = \left\| \begin{matrix}
    c_E(x_{k, j}) \\ [c_I(x_{k, j})]_+
    \end{matrix} \right\|_{1} 
    - (e_{m_E}^T y_E + e_{m_I}^T y_I).  \label{eq:constraint_change}
\end{align}
The merit parameter is updated to ensure that the search direction $d_{k, j}$ is a descent direction for the merit function, similar to the approaches in \cite{curtis2025interior,qiu2023sequential}.
Given user-defined parameters $\epsilon_\sigma, \epsilon_{\tau} \in (0, 1)$, a trial value $\tau_{k,j}^{trial}$ is computed as,
\begin{align*}
\tau^{trial}_{k,j} =
\begin{cases}
    \infty & \text{if } g_{S_k}(x_{k, j})^Td_{k,j} + d_{k,j}^TH_{k,j}d_{k,j} \leq 0, \\
    \tfrac{(1 - \epsilon_\sigma) \Delta_c}{g_{S_k}(x_{k, j})^Td_{k,j} + d_{k,j}^TH_{k,j}d_{k,j}} & \text{otherwise,}
\end{cases}
\end{align*}
and then the merit parameter is updated via
\begin{align}  \label{eq:ineq_merit_update}
\tau_{k,j} =
\begin{cases}
    \tau_{k,j-1} & \text{if } \tau_{k,j-1} \leq \tau^{trial}_{k,j}, \\
    \min\left\{(1 - \epsilon_{\tau})\tau_{k, j}, \tau^{trial}_{k,j}\right\} & \text{otherwise}.
\end{cases}
\end{align}
This rule ensures that $\{\tau_{k,j}\}$ is a monotonically non-increasing sequence. The merit parameter is reinitialized to a user-defined value at the beginning of each inner loop. Finally, the step size $\alpha_{k,j}$ is selected to satisfy the 
sufficient decrease
condition on the merit function, i.e.,
\begin{equation} \label{eq:line_search_ineq}
    \phi_{S_k} (x_{k, j} + \alpha_{k, j} d_{k, j}, \tau_j) \leq \phi_{S_k} (x_{k, j}, \tau_{k, j}) + \eta\alpha_{k, j} \Delta l_{S_k} (x_{k, j}, d_{k, j}, \tau_{k, j})
\end{equation}
where $\eta \in (0, 1)$ is a user defined line search parameter. 

Within an outer iteration $k$, under appropriate constraint qualifications, the search direction $d_{k, j}$ goes to zero, and the algorithm either approaches a KKT point for the subsampled problem or terminates at an infeasible stationary point \cite{burke1989robust,qiu2023sequential}.
An infeasible stationary point is infeasible to the original problem, and a stationary point to the constraint minimization problem ($l_\infty$ on the left and $l_1$ on the right)
\begin{equation} \label{eq:inf_stat_point}
    \min_{x \in \Rmbb^n} \left\|\begin{matrix}
        c_E(x) \\ [c_{I}(x)]_{+}
    \end{matrix}\right\|_{\infty} 
    \text{or,} \quad
    \min_{x \in \Rmbb^n} \left\|\begin{matrix}
        c_E(x) \\ [c_{I}(x)]_{+}
    \end{matrix}\right\|_{1}.
\end{equation}
If the search direction $p_{k, j}$ is zero while the constraints are violated, an infeasible stationary point has been reached, and both the inner and outer loops can be terminated, since the constraints in the subsampled problem are deterministic. Otherwise, in outer iteration $k$, the inner loop terminates when the termination criterion $\Tcal_k$ from \cref{alg:RA_base} (Line \ref{RA_base:termination}) is satisfied, defined as, 
\begin{equation} \label{eq:termination_criterion_ineq}
    \Tcal_k : \|d_{k, j}\| \leq \gamma_k \|d_{k, 0}\| + \epsilon_k.
\end{equation}
The algorithm is summarized in \cref{alg:Inequality_Constrained_RA_SQP} and the following remark.

Another possible measure of solution quality for the termination criterion is the merit function model reduction $\Delta l_{S_k}$. 
However, since it includes the positive part of the inequality constraints and the active set may change significantly within the inner loop due to possible drastic variations in the objective function across outer iterations, employing $\Delta l_{S_k}$ leads to instability and thus it is not used.
Moreover, one can use any other deterministic method in the inner loop but then the measure of optimality for termination must be appropriately chosen.

\begin{algorithm}[H]
    \caption{Inequality Constrained \texttt{RA}-\texttt{SQP}}
    \textbf{Inputs:} Sequences $\{|S_k|\}$, $\{\gamma_k\} < 1$ and $\{\epsilon_k\} \geq 0$; initial iterate $x_{0,0}$; initial merit parameter sequence $\{\tau_{k,-1}\} = \bar{\tau} > 0$; other parameters $ \epsilon_{\sigma}, \epsilon_{\tau}, \epsilon_{\alpha} \in (0, 1)$, $\sigma_p, \sigma_d > 0$.
\    \begin{algorithmic}[1]
        \For{$k = 0, 1, 2, \dots$}
            \State Construct subsampled problem \eqref{eq:subsampled_problem} from sample set $S_k$ \label{Ineq:subsample}
            \For{$j = 0, 1, 2, \dots$}
                \State Solve \eqref{eq:constraint_prob} and obtain $p_{k, j}$ and $y$ (or $y_E$ and $y_I$) \label{Ineq:p_cal}
                \State \textbf{if} $\|p_{k, j}\| = 0$ and $y > 0$ (or either $\|y_E\| > 0$ or $\|y_I\| > 0$) \label{Ineq:inf_Stat_check}
                \State \hskip1.0em  \textbf{return} $x_{k, j}$ as an infeasible stationary point
                \State Obtain approximation of Hessian of Lagrangian $H_{k, j}$ \label{Ineq:H}
                \State Solve \eqref{eq:opt_prob} and obtain $d_{k, j}$ \label{Ineq:d_cal}
                \State \textbf{if} \eqref{eq:termination_criterion_ineq} is satisfied; $N_k = j$, \textbf{break} \label{Ineq:term_test}
                \State Update merit parameter via \eqref{eq:ineq_merit_update} \label{Ineq:tau_update}
                \State Set $\alpha_{k,j} = 1$ \label{Ineq:line_search_start}
                \While{\eqref{eq:line_search_ineq} is \textbf{not satisfied}}
                    \State Update $\alpha_{k,j} = \epsilon_{\alpha} \alpha_{k,j}$
                \EndWhile 
                \State Set $x_{k, j + 1} = x_{k, j} + \alpha_{k, j} d_{k,j}$ \label{Ineq:update}
            \EndFor
            \State Set $x_{k+1, 0} = x_{k, N_k}$
        \EndFor
    \end{algorithmic}
    \label{alg:Inequality_Constrained_RA_SQP}
\end{algorithm}
\begin{remark}
    We make the following remarks about \cref{alg:Inequality_Constrained_RA_SQP}.
    \begin{enumerate}
        \item \textbf{Outer Iteration (Line \ref{Ineq:subsample}):} A set of samples $S_k$ is obtained to construct the subsampled problem \eqref{eq:subsampled_problem}. The batch size is a prespecified user-defined sequence and an adaptive strategy is described later.
        \item \textbf{\texttt{SQP} subproblem (Lines \ref{Ineq:p_cal}-\ref{Ineq:d_cal}):} 
        The search direction $d_{k, j}$ is computed via a two-step process. First, a search direction $p_{k, j}$ is computed via \eqref{eq:constraint_prob} that minimizes the linearized constraint violation. If $\|p_{k, j}\| = 0$ and $x_{k, j}$ is infeasible to the original problem, the algorithm terminates and returns $x_{k, j}$ as an infeasible stationary point. 
        Then, the search direction $d_{k, j}$ is computed via \eqref{eq:constraint_prob} that minimizes a quadratic approximation of the objective function while achieving linearized constraint violation at least as good as the linearized constraint violation achieved by $p_{k, j}$.
        The formulations used to compute $p_{k, j}$ and $d_{k, j}$ depend on the choice of constraint violation measure ($l_1$ or $l_\infty$ norm).
                
        \item \textbf{Termination Test (Line \ref{Ineq:term_test}):} The inner loop is terminated when the search direction norm $\|d_{k, j}\|$ falls below a certain tolerance \eqref{eq:termination_criterion_ineq}.
        Thus, one extra \texttt{SQP} subproblem is solved in the inner loop than the number of iterate updates, as $\|d_{k, j}\|$ is required to evaluate the termination criterion \eqref{eq:termination_criterion_ineq}.

        \item \textbf{Merit Parameter and Line Search (Lines \ref{Ineq:tau_update}-\ref{Ineq:update}):} The mechanism for updating the merit parameter and performing the line search to obtain the step size follows standard procedures from \cite{burke1989robust, nocedal2006numerical, qiu2023sequential, curtis2025interior}.
        Since the objective function can change drastically across outer iterations, similar to \cref{alg:Equality_Constrained_RA_SQP}, the merit parameter is reinitialized at the start of each outer iteration to the user-specified value $\bar{\tau}$ similar to \cref{alg:Equality_Constrained_RA_SQP}.
    \end{enumerate}
\end{remark}

\subsection{Convergence Analysis} \label{sec:Inquality_constrained_analysis}
In this section, we establish convergence guarantees for \cref{alg:Inequality_Constrained_RA_SQP} for problems with general nonlinear constraints \eqref{eq:intro_problem}. 
We first demonstrate the well-posedness of \cref{alg:Inequality_Constrained_RA_SQP} and establish conditions for convergence, and then present the adaptive strategy for batch size selection and complexity results.

To begin, we borrow a few definitions from \cref{sec:General_Deterministic_Solver_theory}. We use \cref{ass:gradient_errors} and error metric $G_S$ from \eqref{eq:G_metric_def} to characterize the error in the subsampled gradients, and define $\{\Fcal_k\}$ as the same sequence of $\sigma$-algebras with $\Fcal_0 = \{x_{0, 0}\}$, i.e., without the dual variable.
Next, we introduce the assumptions on \eqref{eq:intro_problem}. These assumptions are necessary to establish the well-posedness of \cref{alg:Inequality_Constrained_RA_SQP} that uses robust-\texttt{SQP} subproblems in the inner loop.

\begin{assumption}  \label{ass:SQP_ineq_assumptions}
    Let $\chi \subset \Rmbb^n$ be a closed bounded convex set containing the sequence of iterates generated by any run of \cref{alg:Inequality_Constrained_RA_SQP}. 
    The objective function $f : \Rmbb^n \rightarrow \Rmbb$ is continuously differentiable and its gradient function $\nabla f : \Rmbb^n \rightarrow \Rmbb^n$ is Lipschitz continuous over $\chi$.
    The subsampled objective function $F_{S_k} : \Rmbb^n \rightarrow \Rmbb$ for all $k \geq 0$ is continuously differentiable with Lipschitz continuous gradients over $\chi$.
    The constraint functions $c_E : \Rmbb^n \rightarrow \Rmbb^{m_E}$ and  $c_I : \Rmbb^n \rightarrow \Rmbb^{m_I}$ are continuously differentiable. The Jacobian functions, $\nabla c_E^T : \Rmbb^n \rightarrow \Rmbb^{m_E \times n}$ and $\nabla c_I^T : \Rmbb^n \rightarrow \Rmbb^{m_I \times n}$ are Lipschitz continuous over $\chi$. 
    The sequence of Hessian approximation matrices $\{H_{k, j}\} \in \Rmbb^{n \times n}$ are symmetric positive definite matrices with $\kappa_H \geq \mu_H > 0$, such that, $\mu_H I_n \preceq H_{k, j} \preceq \kappa_H I_n$.
\end{assumption}

While assuming a closed bounded set of iterates in \cref{ass:SQP_eq_assumptions} was a strong condition for problems with equality constraints, similar conditions (\cref{ass:SQP_ineq_assumptions}) are commonly required in both deterministic \cite{burke1989robust} and stochastic optimization \cite{na2023inequality,qiu2023sequential} with inequality constraints, thereby justifying their use in this scenario. 
However, we also need the subsampled problems to be smooth as it is required by the robust \texttt{SQP} method \cite[Assumption 2.2]{qiu2023sequential} and the \texttt{RA} framework treats each subsampled problem as a deterministic problem, similar to \cref{ass:SQP_eq_assumptions}.
These conditions can be relaxed if permitted by the inner loop deterministic solver employed to establish convergence under conditions similar to \cref{sec:General_Deterministic_Solver_theory}.
The conditions on the approximation of the Hessian of the Lagrangian can also be relaxed to positive definiteness over the null space of the active set. Under \cref{ass:SQP_ineq_assumptions}, there exists $\kappa_g \geq 0$ such that ${\|\nabla f(x)\|} \leq \kappa_g$ for all $x \in \chi$, and \cref{ass:Bounded_outer_gradient_strong} is satisfied.

\begin{assumption} \label{ass:MFCQ_ineq}
    The extended MFCQ hold at all points $x \in \chi$, i.e.,
    \begin{enumerate}
        \item the Jacobian for the equality constraints is full row rank, i.e., $rank(\nabla c_E(x)) = m_E$,
        \item there exists $u \in \Rmbb^n$ such that $\nabla c_E(x)^T u = 0$ and $[\nabla c_I(x)^T]_i u > 0$ for all $i \in \{i : [c_I(x)]_i \geq 0\}$.
    \end{enumerate}
\end{assumption}

The constraint qualification in \cref{ass:MFCQ_ineq} is an extension of the MFCQ (Mangasarian-Fromovitz constraint qualification), also found in \cite{burke1989robust, qiu2023sequential}, that incorporates infeasible points into the constraint qualification. These conditions are necessary to ensure the boundedness of the merit parameter ($\{\tau_{k, j}\} > 0$), thereby guaranteeing convergence to either a first-order stationary point or an infeasible stationary point. Under the stated assumptions, we now establish the well-posedness of \cref{alg:Inequality_Constrained_RA_SQP}.

\begin{theorem} \label{th:Well_posed_ineq}
    Suppose Assumptions \ref{ass:SQP_ineq_assumptions} and \ref{ass:MFCQ_ineq} hold. Then, for each outer iteration $k\geq0$ in \cref{alg:Inequality_Constrained_RA_SQP}, $\{d_{k, j}\} \rightarrow 0$ and the inner loop terminates with either an infeasible stationary point or upon satisfying termination criterion \eqref{eq:termination_criterion_ineq}. 
\end{theorem}
\begin{proof}
    Based on \cite[Theorem 6.1]{burke1989robust} (as restated in \cite[Theorem 2.4]{qiu2023sequential}), under the stated assumptions, the robust \texttt{SQP} method as described in \cref{sec:Inequality_constrained_algorithm} in outer iteration $k$ either terminates finitely with $d_{k, N_k} = 0$ or 
    $\|d_{k, j}\| \rightarrow 0$ as $j \rightarrow \infty$, thus satisfying termination criterion \eqref{eq:termination_criterion_ineq} finitely. If $d_{k, N_k} = 0$ but the constraints are not satisfied, the algorithm terminates returning an infeasible stationary point based on \cite[Lemma 2.2]{burke1989robust} and the deterministic nature of constraints.
\end{proof}

At $x_{k, j}$, the solution to \eqref{eq:constraint_prob} is independent of the gradient approximation employed. Therefore, the feasible region of \eqref{eq:opt_prob} is independent of the sample set given $x_{k, j}$.
Let $d_{k, j}^{true}$ be the solution to the subproblem \eqref{eq:opt_prob} where the true gradient of the objective function is used. 
If $d_{k, j}^{true} = 0$, $x_{k, j}$ is either a first-order stationary point or an infeasible stationary point by \cite[Lemma 2.2]{burke1989robust}. We analyze the convergence of \cref{alg:Inequality_Constrained_RA_SQP} by analyzing $\|d_{k, j}^{true}\|$, similar to the analysis presented in \cref{sec:General_Deterministic_Solver_theory}. We first establish a bound on $\|d_{k, N_k}^{true}\|$ similar to \cref{lem:main_error_bound}.

\begin{lemma} \label{lem:ineq_main_error_bound}
    Suppose Assumptions \ref{ass:SQP_ineq_assumptions} and \ref{ass:MFCQ_ineq} hold. Then for all $k\geq 0$, the norm of the true search direction at outer iterates generated by \cref{alg:Inequality_Constrained_RA_SQP} satisfy
    \begin{align*}
        \|d_{k, N_k}^{true}\| 
        &\leq  \gamma_k \|d_{k-1, N_{k-1}}^{true}\| + \epsilon_k \numberthis \label{eq:ineq_main_error_bound}  \\
        &\quad + \mu_H^{-1} \|\nabla f(x_{k, N_k}) - g_{S_k}(x_{k, N_k})\| +  \gamma_k\mu_H^{-1} \|\nabla f(x_{k, 0}) - g_{S_k}(x_{k, 0})\|. 
    \end{align*}
    For the expectation problem \eqref{eq:intro_stoch_error_obj}, for all $k \geq 0$
    \begin{align*}
        \Embb[\|d_{k, N_k}^{true}\| | \Fcal_k]
        &\leq  \gamma_k \|d_{k-1, N_{k-1}}^{true}\| + \Embb[\epsilon_k | \Fcal_k] \numberthis \label{eq:ineq_main_error_bound_expec}  \\
        &\quad + \mu_H^{-1} \Embb\left[\| \nabla f(x_{k, N_k}) - g_{S_k}(x_{k, N_k})\| | \Fcal_k\right] \\
        &\quad +  \gamma_k\mu_H^{-1} \Embb\left[\|\nabla f(x_{k, 0}) - g_{S_k}(x_{k, 0})\| | \Fcal_k\right]. 
    \end{align*}
\end{lemma}
\begin{proof}
    We first provide a bound on the error in the search direction based on the error in the subsampled gradient, following the procedure in \cite[Lemma 4.16]{curtis2023sequential}. We then follow the procedure from \cref{lem:step_norm_error_bound} to establish the desired bounds. 
    
    At $x_{k, j}$, the feasible region to the subproblems for $d_{k, j}$ and $d_{k, j}^{true}$ are the same. The objective function of \eqref{eq:opt_prob} can be rewritten as $(H_{k, j}^{-1} g_{S_k}(x_{k, j}) + d_{k, j})^T H_{k, j} (H_{k, j}^{-1} g_{S_k}(x_{k, j}) + d_{k, j})$ and as $(H_{k, j}^{-1} \nabla f(x_{k, j}) + d_{k, j}^{true})^T H_{k, j} (H_{k, j}^{-1} \nabla f(x_{k, j}) + d_{k, j}^{true})$ when the true problem gradient is employed. As the feasible region is convex for both $l_1$ and $l_{\infty}$ norm based formulations in \eqref{eq:opt_prob}, from \cite[Proposition 1.1.9]{bertsekas2009convex},
    \begin{equation*}
        \begin{aligned}
            (d_{k, j} - d_{k, j}^{true})^T H_{k, j} (H_{k, j}^{-1} \nabla f(x_{k, j}) - d_{k, j}^{true}) &\leq 0
            \quad \text{and,}\\
            (d_{k, j}^{true} - d_{k, j})^T H_{k, j} (H_{k, j}^{-1} g_{S_k}(x_{k, j}) - d_{k, j}) &\leq 0.
        \end{aligned}
    \end{equation*}
    Summing the two inequalities,
    \begin{align*}
        0 &\leq (d_{k, j}^{true} - d_{k, j})^T H_{k, j} (H_{k, j}^{-1} \nabla f(x_{k, j}) - d_{k, j}^{true} - H_{k, j}^{-1} g_{S_k}(x_{k, j}) + d_{k, j})\\
        &= - (d_{k, j}^{true} - d_{k, j})^T H_{k, j} (d_{k, j}^{true} - d_{k, j}) \\
        &\quad + (d_{k, j}^{true} - d_{k, j})^T H_{k, j} H_{k, j}^{-1} (\nabla f(x_{k, j}) - g_{S_k}(x_{k, j}))\\
        &\leq -\mu_H \|d_{k, j}^{true} - d_{k, j}\|^2 + (d_{k, j}^{true} - d_{k, j})^T (\nabla f(x_{k, j}) - g_{S_k}(x_{k, j})) \\
        &\leq -\mu_H \|d_{k, j}^{true} - d_{k, j}\|^2 + \|d_{k, j}^{true} - d_{k, j}\| \|\nabla f(x_{k, j}) - g_{S_k}(x_{k, j})\|, 
    \end{align*}
    where the second inequality follows from $\mu_H I_n \preceq H_{k, j}$ and the third inequality follows from Cauchy-Schwarz inequality. Using the above inequality, it follows that
    \begin{align*}
        \|d_{k, j}^{true} - d_{k, j}\| &\leq  \mu_H ^{-1}  \|\nabla f(x_{k, j}) - g_{S_k}(x_{k, j})\|.
    \end{align*}
    The true search direction norm can be bounded as,
    \begin{align*}
        \|d_{k, N_k}^{true}\|
        &\leq  \|d_{k, N_k}\| + \|d_{k, N_k}^{true} - d_{k, N_k}\| \\
        &\leq  \gamma_k \|d_{k, 0}\| + \epsilon_k + \|d_{k, N_k}^{true} - d_{k, N_k}\| \\ 
        &\leq  \gamma_k \|d_{k, 0}^{true}\| + \epsilon_k + \|d_{k, N_k}^{true} - d_{k, N_k}\| +  \gamma_k \|d_{k, 0}^{true} - d_{k, 0}\|\\ 
        &=  \gamma_k \|d_{k-1, N_{k-1}}^{true}\| + \epsilon_k + \|d_{k, N_k}^{true} - d_{k, N_k}\| +  \gamma_k \|d_{k, 0}^{true} - d_{k, 0}\|,
    \end{align*}
    where the second inequality follows from termination criterion \eqref{eq:termination_criterion_ineq} and the equality follows because $x_{k, 0} = x_{k-1, N_{k-1}}$, $H_{k, 0} = H_{k-1, N_{k-1}}$ resulting in the same subproblems for $d_{k, 0}^{true}$ and $d_{k-1, N_{k-1}}^{true}$. Substituting the derived bound on the search direction error yields \eqref{eq:ineq_main_error_bound}. Taking conditional expectation of \eqref{eq:ineq_main_error_bound} given $\Fcal_k$ yields \eqref{eq:ineq_main_error_bound_expec} as $d_{k-1, N_{k-1}}^{true}$ is known under $\Fcal_k$.
\end{proof}

Using \cref{lem:ineq_main_error_bound}, we now establish the convergence of \cref{alg:Inequality_Constrained_RA_SQP}, following the procedure in \cref{th:Eq_convergence}.

\begin{theorem} \label{th:convergence_ineq}
    Suppose Assumptions \ref{ass:gradient_errors}, \ref{ass:SQP_ineq_assumptions} and \ref{ass:MFCQ_ineq} hold for any run of \cref{alg:Inequality_Constrained_RA_SQP}.
    \begin{enumerate}
        \item For the finite-sum problem \eqref{eq:intro_deter_error_obj}: If the batch size sequence is chosen such that $\{S_k\} \rightarrow \Scal$ and the termination criterion parameters are chosen such that $\epsilon_k \rightarrow 0$ and $0 \leq \{\gamma_k\} \leq \gamma < 1$, then $\|d_{k, N_k}^{true}\| \rightarrow 0$.
        \item For the expectation problem \eqref{eq:intro_stoch_error_obj}: If the batch size sequence is chosen such that $\{|S_k|\} \rightarrow \infty$ and the termination criterion parameters are chosen such that $\Embb[\epsilon_k | \Fcal_k] \rightarrow 0$ and $0 \leq \{\gamma_k\} \leq \tilde{\gamma} < 1$, and given $\Embb\left[G_{S_k}^2 | \Fcal_k\right] \rightarrow 0$ as $\{|S_k|\} \rightarrow \infty$, then $\Embb\left[\|d_{k, N_k}^{true}\|\right] \rightarrow 0$.
    \end{enumerate}
\end{theorem}
\begin{proof}
    The proof follows the same procedure as \cref{th:Eq_convergence} and has been included here for completeness. 
    For the finite-sum problem \eqref{eq:intro_deter_error_obj}, unrolling \eqref{eq:ineq_main_error_bound} from \cref{lem:ineq_main_error_bound} and substituting \eqref{eq:deter_sampled_gradient_error} yields for all $k \geq 0$,
    \begin{align*}
        \|d_{k, N_k}^{true}\|
        &\leq \left\{\prod_{h = 0}^{k} \gamma_h \right\}\|d_{0, 0}^{true}\|  + \epsilon_k \\
        &\quad + 2\mu_H^{-1} \left(\tfrac{|\Scal| - |S_k|}{|\Scal|}\right) \left(\omega_1 (\|\nabla f(x_{k, N_k})\| +\gamma_k \|\nabla f(x_{k, 0})\|) + \omega_2 (1 + \gamma_k)
        \right)\\
        &\quad + \sum_{i=0}^{k-1} \left\{\prod_{h = i+1}^{k} \gamma_h \right\} \left[\epsilon_i + 2\mu_H^{-1} \left(\tfrac{|\Scal| - |S_i|}{|\Scal|}\right) \left(\omega_1 \|\nabla f(x_{i, N_i})\| + \omega_2 \right) \right] \\
        &\quad + \sum_{i=0}^{k-1} \left\{\prod_{h = i+1}^{k} \gamma_h \right\} \left[2\mu_H^{-1} \gamma_i \left(\tfrac{|\Scal| - |S_i|}{|\Scal|}\right) \left(\omega_1 \|\nabla f(x_{i, 0})\| + \omega_2 \right) \right] \\
        &\leq \left\{\prod_{h = 0}^{k} \gamma_h \right\}\|d_{0, 0}^{true}\|  + \epsilon_k + 2\mu_H^{-1} \left(\tfrac{|\Scal| - |S_k|}{|\Scal|}\right) \left(\omega_1 \kappa_g + \omega_2 
        \right)(1 +\gamma_k) \\
        &\quad + \sum_{i=0}^{k-1} \left\{\prod_{h = i+1}^{k} \gamma_h \right\} \left[\epsilon_i + 2\mu_H^{-1} \left(\tfrac{|\Scal| - |S_i|}{|\Scal|}\right) \left(\omega_1 \kappa_g + \omega_2\right)(1 +\gamma_i)  \right] ,
    \end{align*}
    where the second inequality follows from \cref{ass:SQP_ineq_assumptions}. 
    The first term in the above inequality converges to zero as $\left\{\prod_{h = 0}^{k} \gamma_h \right\} \leq \gamma^{k+1}$ where $\gamma < 1$. The second and third terms converge to zero based on the specifications $\epsilon_k \rightarrow 0$ and $\{|S_k|\} \rightarrow |\Scal|$, respectively.
    The final term is a series of the form $\sum_{i=0}^{k-1}a_i \left\{\prod_{h = i+1}^{k} \gamma_h \right\} \leq \sum_{i=0}^{k-1}a_i \gamma^{k - i}$ which converges to zero by \cref{lem:series_for_convergence}, with $b_i = a_i = \epsilon_i + 2\mu_H^{-1} \left(\tfrac{|\Scal| - |S_i|}{|\Scal|}\right) (1 + \gamma_i)\left(\omega_1 \kappa_g  + \omega_2 \right) \rightarrow 0$, thus completing the proof. 

    For the expectation problem \eqref{eq:intro_stoch_error_obj}, the expectation result \eqref{eq:ineq_main_error_bound_expec} from \cref{lem:ineq_main_error_bound} can be further refined following the procedure from \cref{th:Eq_convergence}. For all $k\geq0$,
    \begin{align*}
        \Embb[\|d_{k, N_k}^{true}\|]
        &\leq  \left\{\prod_{h = 0}^{k} \gamma_h \right\}\|d_{k, N_k}^{true}\|  + \Embb\left[\Embb[\epsilon_k | \Fcal_k] + \epsilon_G\mu_H^{-1}\Embb[G_{S_k}|\Fcal_k] + \kappa_g\mu_H^{-1}\sqrt{\Embb[G_{S_k}^2 |\Fcal_k]}\right] \\
        &\quad + \Embb\left[\gamma_k \mu_H^{-1} \tfrac{\Tilde{\omega}_1 \kappa_g + \Tilde{\omega}_2}{\sqrt{|S_k|}}\right] + \sum_{i=0}^{k-1} \left\{\prod_{h = i+1}^{k} \gamma_h \right\}\Embb\left[\Embb[\epsilon_i| \Fcal_i]\right] \\
        &\quad + \sum_{i=0}^{k-1} \left\{\prod_{h = i+1}^{k} \gamma_h \right\}\Embb\left[\epsilon_G\mu_H^{-1}\Embb[G_{S_i}|\Fcal_i] + \kappa_g\mu_H^{-1}\sqrt{\Embb[G_{S_i}^2 |\Fcal_i]} + \gamma_i\mu_H^{-1}\tfrac{\Tilde{\omega}_1 \kappa_g + \Tilde{\omega}_2}{\sqrt{|S_i|}} \right] .
    \end{align*}
    In the above bound, the first term is the same as it appeared for deterministic errors.
    All the terms within the second and third term expectation converge to zero based on the specifications $\Embb[\epsilon_k | \Fcal_k] \rightarrow 0$, $\{|S_k|\} \rightarrow \infty$, and the assumption $\Embb[G_{S_k}^2 | \Fcal_k] \rightarrow 0$.  
    The final terms together are a series of the form $\sum_{i=0}^{k-1}\Embb[a_i] \left\{\prod_{h = i+1}^{k} \gamma_h \right\} \leq \sum_{i=0}^{k-1} \Embb[a_i] \tilde{\gamma}^{k - i}$ which converges to zero by \cref{lem:series_for_convergence}, with $b_i = \Embb[a_i] \rightarrow 0$ as $a_i \rightarrow 0$, thus completing the proof.
\end{proof}

\cref{th:convergence_ineq} establishes the convergence of the norm of the true search direction to zero. As established in \cite{burke1989robust, qiu2023sequential}, if the $\|d_{k, N_k}^{true}\| \rightarrow 0$, the limit point is a stationary point for the constraint violation minimization problem \eqref{eq:inf_stat_point}. If the limit point is feasible, it is a first-order stationary point for \eqref{eq:intro_problem}. While the convergence of the true search direction norm holds deterministically for the finite-sum problem \eqref{eq:intro_deter_error_obj}, it holds in expectation for the expectation problem \eqref{eq:intro_stoch_error_obj}. Thus, for the expectation problem \eqref{eq:intro_stoch_error_obj}, if there exists a subsequence of $\{d_{k, N_k}\}$ that converges to zero, the limit point is either an infeasible stationary point or a first-order stationary point. However, this scenario is not guaranteed from the convergence in expectation of the true search direction norm, similar to the result in \cite{curtis2023sequential}.
Convergence can be established under stricter assumptions on the gradient errors, similar to the conditions in \cite{qiu2023sequential}.

We now introduce the adaptive strategy for sample set size selection in \cref{alg:Inequality_Constrained_RA_SQP} as a modification of the well-known norm test, similar to \cref{cond:eq_norm_test}.
\begin{condition} \label{cond:ineq_norm_test}
    At the beginning of outer iteration $k \geq 0$ in \cref{alg:Inequality_Constrained_RA_SQP}, the batch size $|S_k|$ is selected such that:
    \begin{enumerate}
        \item For the finite-sum problem \eqref{eq:intro_deter_error_obj}: With constants $\theta, \hat{\theta}, a \geq 0$ and $\beta \in (0, 1)$, 
        \begin{align*}
            \|\nabla f(x_{k, 0}) - g_{S_k}(x_{k, 0})\|^2 &\leq \theta^2 \|d_{k, 0}^{true}\|^2 + a^2 \beta^{2k},\\
            \|\nabla f(x_{k, N_k}) - g_{S_k}(x_{k, N_k})\|^2 &\leq \hat{\theta}^2 \left(\theta^2 \|d_{k, 0}^{true}\|^2 + a^2 \beta^{2k}\right).
        \end{align*}
        \item For the expectation problem \eqref{eq:intro_stoch_error_obj}: With constants $\tilde{\theta}, \tilde{a} \geq 0$ and $\tilde{\beta} \in (0, 1)$, 
        \begin{align*}
            \Embb\left[\|\nabla f(x_{k, 0}) - g_{S_k}(x_{k, 0})\|^2 | \Fcal_k\right] \leq \tilde{\theta}^2 \|d_{k, 0}^{true}\|^2 + \tilde{a}^2 \tilde{\beta}^{2k}.
        \end{align*}
    \end{enumerate}
\end{condition}
\cref{cond:ineq_norm_test} is analogous to \cref{cond:eq_norm_test}, and a discussion of its implications can be found in \cref{remark:norm_test_conditions}. The practical aspects of its implementation are discussed later in \cref{sec:Numerical_experiments}.
We now establish the complexity of \cref{alg:Inequality_Constrained_RA_SQP} under \cref{cond:ineq_norm_test} across outer iterations, assuming that \cref{ass:Variance_regularity} holds for the expectation problem \eqref{eq:intro_stoch_error_obj}, in a manner similar to \cref{th:EQ_outer_iter_complexity}.

\begin{theorem} \label{th:outer_complexity_ineq}
    Suppose Assumptions \ref{ass:gradient_errors}, \ref{ass:SQP_ineq_assumptions} and \ref{ass:MFCQ_ineq} hold, and the batch size sequence $\{|S_k|\}$ is chosen to satisfy \cref{cond:ineq_norm_test} in \cref{alg:Inequality_Constrained_RA_SQP}.
    \begin{enumerate}
        \item For the finite-sum problem \eqref{eq:intro_deter_error_obj}: If termination criterion parameters are chosen as $0 \leq \{\gamma_k\} \leq \gamma < 1$, $\epsilon_k = \omega \|\nabla f(x_{k, 0}) - g_{S_k}(x_{k, 0})\| + \hat{\omega}\beta^k$ with $\omega, \hat{\omega} \geq 0$  and  \cref{cond:ineq_norm_test} parameters are chosen such that $a_1 = \left[\gamma + \theta(\omega + \mu_H^{-1}(\hat{\theta} + \gamma))\right] < 1$, the true search direction converges to zero at a linear rate across outer iterations as,
        \begin{align*}
            \|d_{k, N_k}^{true}\|            
            \leq \max \{a_1 + \nu, \beta\}^{k+1} \max\left\{ \|d_{0, 0}^{true}\|, \tfrac{a_2}{\nu}\right\},
        \end{align*}
        where $a_2 = a(\omega + \mu_H^{-1}(\hat{\theta} + \gamma)) + \hat{\omega}$ and $\nu > 0$ such that $a_1 + \nu < 1$.
        
        \item For the expectation problem \eqref{eq:intro_stoch_error_obj}: If \cref{ass:Variance_regularity} is satisfied and the termination criterion parameters are chosen such that $0 \leq \{\gamma_k\} \leq \tilde{\gamma} < 1$, $\epsilon_k = \tilde{\omega}\sqrt{\tfrac{\Var( \nabla F(x_{k, 0}) | \Fcal_k)}{|S_k|}}$ where $\tilde{\omega} \geq 0$ and \cref{cond:ineq_norm_test} parameters are chosen such that $\tilde{a}_1 = \left[\tilde{\gamma} + \tilde{\theta} \left(\tilde{\omega} + \mu_H^{-1}\left(\tfrac{(\epsilon_G + \kappa_g)}{\kappa_{\sigma}} + \tilde{\gamma}\right)\right) \right] < 1$,  the true search direction converges to zero in expectation at a linear rate across outer iterations as,
        \begin{align*}
            \Embb\left[\|d_{k, N_k}^{true}\|\right]                
            \leq \max \{\tilde{a}_1 + \tilde{\nu}, \tilde{\beta}\}^{k+1} \max\left\{ \|d_{0, 0}^{true}\|, \tfrac{\tilde{a}_2}{\tilde{\nu}}\right\},
        \end{align*}
        where $\tilde{a}_2 = \tilde{a}\left[\tilde{\omega} + \mu_H^{-1}\left(\tfrac{(\epsilon_G + \kappa_g)}{\kappa_{\sigma}} + \tilde{\gamma}\right) \right]$ and $\tilde{\nu} > 0$  such that $\tilde{a}_1 + \tilde{\nu} < 1$.  
    \end{enumerate}
\end{theorem}
\begin{proof}
    For the finite-sum problem \eqref{eq:intro_deter_error_obj}, the error bound \eqref{eq:ineq_main_error_bound} can be further refined as,
    \begin{align*}
        \|d_{k, N_k}^{true}\|
        &\leq \gamma \|d_{k-1, N_{k-1}}^{true}\| +  \omega \left\|\nabla f(x_{k, 0}) - g_{S_k}(x_{k, 0})\right\| + \hat{\omega}\beta^k \\
        &\quad + \mu_H^{-1}\left\|\nabla f(x_{k, N_k}) - g_{S_k}(x_{k, N_k})\right\| + \gamma\mu_H^{-1}\|\nabla f(x_{k, 0}) - g_{S_k}(x_{k, 0})\|\\
        &\leq  \gamma \|d_{k-1, N_{k-1}}^{true}\| + \omega \left(\theta \|d_{k-1, N_{k-1}}^{true}\| + a \beta^{k}\right)  + \hat{\omega}\beta^k \\
        &\quad + \mu_H^{-1}\hat{\theta}(\theta \|d_{k-1, N_{k-1}}^{true}\| + a \beta^{k}) + \mu_H^{-1}\gamma \left(\theta \|d_{k-1, N_{k-1}}^{true}\| + a \beta^{k}\right)\\
        &=  a_1 \|d_{k-1, N_{k-1}}^{true}\| + a_2 \beta^{k},
    \end{align*}
    where the second inequality follows from \cref{cond:ineq_norm_test} and the final equality follows from the defined constants. 
    Using the above bound, applying \cref{lem:linear_convergence_induction} with $Z_k = \|d_{k, N_k}^{true}\|$, $\rho_1 = a_1$, $\rho_2 = \beta$ and $b = a_2$ completes the proof.
    
    For the expectation problem \eqref{eq:intro_stoch_error_obj}, the expectation error bound \eqref{eq:ineq_main_error_bound_expec} can be further refined as,
    \begin{align*}
        \Embb\left[\|d_{k, N_k}^{true}\| | \Fcal_k \right]
        &\leq \tilde{\gamma} \|d_{k-1, N_{k-1}}^{true}\| + \tilde{\omega}\sqrt{\tfrac{Var(\nabla F(x_{k, 0}) | \Fcal_k)}{|S_k|}} \\
        &\quad + \mu_H^{-1}\Embb\left[\left\|\nabla f(x_{k, N_k}) - g_{S_k}(x_{k, N_k})\right\| | \Fcal_k \right] \\
        &\quad + \tilde{\gamma}\mu_H^{-1}\Embb\left[\|\nabla f(x_{k, 0}) - g_{S_k}(x_{k, 0})\| | \Fcal_k \right]\\
        &\leq \tilde{\gamma} \|d_{k-1, N_{k-1}}^{true}\| + \tilde{\omega}\sqrt{\tfrac{Var(\nabla F(x_{k, 0}) | \Fcal_k)}{|S_k|}} \\
        &\quad + \mu_H^{-1}\kappa_G\tfrac{(\epsilon_G + \kappa_g)}{\sqrt{|S_k|}} + \tilde{\gamma} \mu_H^{-1} \Embb\left[\|\nabla f(x_{k, 0}) - g_{S_k}(x_{k, 0})\| | \Fcal_k \right] \\
        &\leq \tilde{\gamma} \|d_{k-1, N_{k-1}}^{true}\| + \tilde{\omega}\left[\tilde{\theta} \|d_{k-1, N_{k-1}}^{true}\| + \tilde{a}\tilde{\beta}^{k}\right] \\
        &\quad + \mu_H^{-1}\tfrac{(\epsilon_G + \kappa_g)}{\kappa_{\sigma}} \left[\tilde{\theta} \|d_{k-1, N_{k-1}}^{true}\| + \tilde{a}\tilde{\beta}^{k}\right] + \tilde{\gamma} \mu_H^{-1}\left[\tilde{\theta} \|d_{k-1, N_{k-1}}^{true}\| + \tilde{a}\tilde{\beta}^{k}\right]\\
        &= \tilde{a}_1 \|d_{k-1, N_{k-1}}^{true}\|  + \tilde{a}_2\tilde{\beta}^k,
    \end{align*}
    where the second inequality follows from \eqref{eq:biased_gradient_error_bound} and \cref{ass:Variance_regularity}, the third inequality follows by substituting $\kappa_G \leq \tfrac{\sqrt{\Var(\nabla F(x_{k, 0}) | \Fcal_k)}}{\kappa_{\sigma}}$ and \cref{cond:ineq_norm_test} and the last equality follows from the defined constants. Taking the total expectation of the bound yields,
    \begin{align*}
        \Embb\left[\|d_{k, N_k}^{true}\| \right]
        \leq \tilde{a}_1 \Embb\left[\|d_{k-1, N_{k-1}}^{true}\|\right]  + \tilde{a}_2\tilde{\beta}^k.
    \end{align*}
    Applying \cref{lem:linear_convergence_induction} with $Z_k = \Embb\left[\|d_{k, N_k}^{true}\| \right]$, $\rho_1 = \tilde{a}_1 $, $\rho_2 = \tilde{\beta}$ and $b = \tilde{a}_2$ completes the proof.
\end{proof}

\cref{th:outer_complexity_ineq} provides results analogous to those in \cref{th:EQ_outer_iter_complexity_step_norm}, and the parameter trade-offs discussed in \cref{th:EQ_outer_iter_complexity} remain applicable.
Unlike \cref{th:EQ_outer_iter_complexity}, while \cref{th:outer_complexity_ineq} establishes the convergence of the norm of the true search direction at a linear rate, this does not necessarily result in linear convergence of the error in KKT conditions \eqref{eq:KKT_conditions}. Although $d^{true}_{k, N_k} = 0$ guarantees that a KKT point or an infeasible stationary point has been reached, we are not aware of any result that bounds the KKT error in terms of the true search direction norm, as is possible in the setting with only equality constraints.

We now present a corollary for the simpler case of geometrically increasing batch sizes, similar to \cref{cor:Eq_geometric_batch_increase_outer_iter_complexity}. This results from setting $\theta$ (or $\tilde{\theta}$) to zero in  \cref{cond:ineq_norm_test}, removing all the parameter selection conditions in \cref{th:outer_complexity_ineq} and the variance lower bound in \cref{ass:Variance_regularity}.

\begin{corollary} \label{cor:ineq_geometric_batch_increase_outer_iter_complexity}
    Suppose the conditions of \cref{th:outer_complexity_ineq} hold.
    \begin{enumerate}
        \item For the finite-sum problem \eqref{eq:intro_deter_error_obj}: If the batch size is selected as $|S_{k}| =\lceil (1 - \beta^k)|\Scal| \rceil$ with $\beta \in (0, 1)$ and termination criterion parameters are chosen as $0 \leq \{\gamma_k\} \leq \gamma < 1$, $\epsilon_k = \omega \left(1 - \tfrac{|S_k|}{|\Scal|}\right)$ with $\omega \geq 0$, then the true search direction norm converges to zero at a linear rate across outer iterations, as expressed in \cref{th:outer_complexity_ineq} with $a_1 = \gamma$ and $a_2 = \omega + 2\mu_H^{-1}(1 + \gamma)(\omega_1 \kappa_g + \omega_2)$.
        
        \item For the expectation problem \eqref{eq:intro_stoch_error_obj}: If \cref{ass:Variance_regularity} is satisfied, the sample set size is chosen as $|S_{k+1}| = \left\lceil\tfrac{|S_k|}{\tilde{\beta}^2}\right\rceil$ with $\tilde{\beta} \in (0, 1)$ and the termination criterion parameters are chosen such that $0 \leq \{\gamma_k\} \leq \tilde{\gamma} < 1$, $\epsilon_k = \tfrac{\tilde{\omega}}{\sqrt{|S_k|}}$ where $\tilde{\omega} \geq 0$, then the true search direction norm converges to zero in expectation at a linear rate across outer iterations, as expressed in \cref{th:outer_complexity_ineq} with $\tilde{a}_1 = \tilde{\gamma}$ and $\tilde{a}_2 = \tfrac{\tilde{\omega} + \kappa_G\mu_H^{-1}(\epsilon_G + \kappa_g) + \tilde{\gamma}\mu_H^{-1} (\tilde{\omega}_1 \kappa_g + \tilde{\omega}_2)} {\sqrt{|S_0|}}$, where $|S_0|$ is the initial batch size.
    \end{enumerate}
\end{corollary}
\begin{proof}
    The proof follows the same procedure as \cref{cor:Eq_geometric_batch_increase_outer_iter_complexity} with adjustments to account for changes from \cref{lem:main_error_bound} to \cref{lem:ineq_main_error_bound} by incorporating the scaling factor $\mu_H^{-1}$, as done in \cref{th:outer_complexity_ineq}.
    The complete proof is presented in \cref{appendix:proofs}.
\end{proof}

\cref{cor:ineq_geometric_batch_increase_outer_iter_complexity} presents a geometrically increasing batch size scheme that satisfies \cref{cond:ineq_norm_test}. One can also select batch sizes sufficiently large with $\theta$ (or $\tilde{\theta}$) non-zero to ensure \cref{cond:ineq_norm_test} is satisfied, similar to \cref{lem:eq_large_enough_batch_sizes}.

\begin{lemma} \label{lem:ineq_large_enough_batch_sizes}
    Suppose Assumptions \ref{ass:gradient_errors}, \ref{ass:SQP_ineq_assumptions} and \ref{ass:MFCQ_ineq} hold in \cref{alg:Inequality_Constrained_RA_SQP}.
    \begin{enumerate}
        \item For the finite-sum problem \eqref{eq:intro_deter_error_obj}: For $k \geq 0$, \cref{cond:ineq_norm_test} is satisfied if
        \begin{equation*}
            |S_k| \geq |\Scal|\left(1 - \sqrt{\tfrac{\theta^2 \|d_{k, 0}^{true}\|^2 + a^2 \beta^{2k}}{4(\omega_1^2 \kappa_g^2 + \omega_2^2)}}\right) \quad \text{with} \quad \hat{\theta} = 1.
        \end{equation*}
        \item For the expectation problem \eqref{eq:intro_stoch_error_obj}:
        For $k \geq 0$, \cref{cond:ineq_norm_test} is satisfied if
        \begin{equation*}
            |S_k| \geq  \tfrac{\tilde{\omega}_1^2 \kappa_g^2 + \tilde{\omega}_2^2}{\tilde{\theta}^2 \|d_{k, 0}^{true}\|^2 + \tilde{a}^2 \tilde{\beta}^{2k}}.
        \end{equation*}
    \end{enumerate}
\end{lemma}
\begin{proof}
    The proof follows the same procedure as \cref{lem:eq_large_enough_batch_sizes} with suitable adjustments for \cref{cond:ineq_norm_test}.
    The complete proof can be found in \cref{appendix:proofs}.
\end{proof}

\cref{lem:ineq_large_enough_batch_sizes}, while somewhat pessimistic, demonstrates that there exists a sequence of batch sizes ${|S_k|}$ that satisfies \cref{cond:ineq_norm_test}. Although our analysis focuses on \cref{alg:RA_base} with a specific deterministic solver for inequality constraints, it can be readily extended to other deterministic solvers used within the inner loop, provided an appropriate termination criterion is defined. However, in contrast to the case with only equality constraints, the empirical performance in the presence of inequality constraints is significantly more sensitive to these choices.
\section{Numerical Experiments} \label{sec:Numerical_experiments}

In this section, we present numerical results that illustrate the empirical performance of the proposed algorithms.
All methods were implemented by the authors in Python\footnote{Our code will be made publicly available upon publication of the manuscript within the GitHub repository: \url{https://github.com/SANDOPT/Gradient-Tracking-Algorithmic-Framework.git}. All our experiments were conducted on a PowerEdge R760 server equipped with two Intel Xeon Gold 16-core processors and 256 GB of RDIMM memory}.
We begin by introducing the test problems, followed by a discussion of key implementation considerations for the adaptive sampling strategy. We report numerical results, first for problems with only equality constraints (\cref{sec:Equality_constrained_numerical_exps}), and then for problems involving general nonlinear constraints (\cref{sec:Inequality_constrained_numerical_exps}). Implementation details for all methods are discussed in the respective subsections.

We evaluate the algorithms on two types of problems: regularized multi-class logistic regression tasks, and  problems from the S2MPJ CUTEst problem set \cite{gratton2024s2mpj}. For the classification tasks, the objective function is defined using a logistic regression model with a cross-entropy loss over a set of classes $\Kcal$, and is given by
\begin{equation} \label{eq:logreg_obj}
    f(x) = \frac{1}{|\Scal|} \sum_{(y, t) \in \Scal} \sum_{i \in \Kcal} t^i \log\left(\frac{1}{1 + \exp(-y^T x^i)}\right),
\end{equation}
where $\Scal$ denotes the set of samples, and each sample $(y, t) \in \Scal$ consists of a feature vector $y \in \Rmbb^{n_f}$ (where $n_f$ is the number of features including the bias term) and a label vector $t \in \{0, 1\}^{|\Kcal|}$, with $t^i = 1$ if the sample belongs to class $i \in \Kcal$, and $t^i = 0$ otherwise. Each $x^i \in \Rmbb^{n_f}$ represents the classifier corresponding to class $i \in \Kcal$, and the variable $x \in \Rmbb^{n}$ is the concatenation of classifiers for all classes with $n = n_f |\Kcal|$. We report results on the covetype dataset ($n_f = 55$, $|\Kcal| = 7$, $|\Scal| = 581,012$) and the mnist dataset ($n_f = 781$, $|\Kcal| = 10$, $|\Scal| = 60,000$) \cite{CC01a}. Constraints for these problems are described later.

The objective function for problems from the S2MPJ CUTEst set \cite{gratton2024s2mpj} are augmented as,
\begin{equation} \label{eq:cutest_problem}
    F(x, \xi) = f(x) + \xi \|x - x_{init} - e_n\|^2,
\end{equation}
where $f(x)$ is the deterministic objective function of the true problem, $x_{init}$ 
is the starting point and $\xi$ is a uniformly distributed random variable in $[-0.1, 0.1]$ independent of $x$, ensuring $\Embb[F(x, \xi)] = f(x)$. This augmentation differs from those considered in \cite{berahas2021sequential,curtis2023sequential,o2024two,berahas2025sequential}, as the noise in the objective function now depends on the current iterate and increases as the algorithm moves farther from the initial point. This results in an adversarial setting that demands more accurate function estimates to achieve good empirical performance. Additionally, the offset in $\|x - x_{init} - e_n\|^2$ prevents the algorithms from receiving perfect information at the initial point. The constraints remain the deterministic constraints specified in the problem set.

\subsection{Adaptive Sampling Implementation}

This section outlines the practical implementation of our adaptive sampling strategy 
for Algorithms \ref{alg:Equality_Constrained_RA}, \ref{alg:Equality_Constrained_RA_SQP} and \ref{alg:Inequality_Constrained_RA_SQP}.
In practice, the finite-sum problem \eqref{eq:intro_deter_error_obj} is treated as an expectation problem \eqref{eq:intro_stoch_error_obj} within the adaptive sampling procedure. The adaptive sampling conditions at the start of outer iteration $k\geq0$ can be expressed in general as $\Embb[\|g_{S_k}(x_{k, 0}) - \nabla f (x_{k, 0})\|^2] \leq \tilde{\theta}^2 Z_k^2 + \tilde{a}^2 \tilde{\beta}^{2k}$, where $Z_k$ represents the true problem quantity in the adaptive sampling condition, determined by the termination criterion $\Tcal_k$ used in the inner loop. Specifically, for \cref{cond:eq_norm_test}, $Z_k = \|T(x_{k, 0}, \lambda_{k-1, N_{k-1}})\|$; for \cref{cond:eq_norm_test_d}, $Z_k = \|d_{k, 0}^{true}\|$; for \cref{cond:eq_norm_test_delta_l}, $Z_k^2 = \Delta l (x_{k, 0}, \tau_{k, 0}^{true}, d_{k, 0}^{true})$; and for \cref{cond:ineq_norm_test}, $Z_k = \|d_{k, 0}^{true}\|$.
The sampling condition can be rearranged as
\begin{align*}
    \tfrac{Var(\nabla F(x_{k, 0}, \xi) | x_{k, 0})}{|S_k|} &\leq \tilde{\theta}^2 Z_k^2 + \tilde{a}^2 \tilde{\beta}^{2k} \quad \Rightarrow \quad
    |S_k| \geq \tfrac{Var(\nabla F(x_{k, 0}, \xi) | x_{k, 0})}{\tilde{\theta}^2 Z_k^2 + \tilde{a}^2 \tilde{\beta}^{2k}},
\end{align*}
where sample estimates are used to approximate the terms in the condition, similar to \cite{berahas2022adaptive,o2024fast,bollapragada2018adaptive,byrd2012sample}. 

Unlike earlier works, we cannot use estimates from the previous outer iteration to approximate the condition. This is because multiple inner iterations in outer iteration $k-1$ introduce a strong bias in the new starting point $x_{k, 0}$ toward the sample set $S_{k-1}$. To mitigate this, we instead draw a new sample set $\tilde{S}_{k-1}$, independent of $x_{k, 0}$, with $|\tilde{S}_{k-1}| = |S_{k-1}|$. We then use the sample variance over $\tilde{S}_{k-1}$ and  $\tilde{Z}_k$ as an estimate for $Z_k$ to approximate the condition as
\begin{align*}
    |S_k| &\geq \tfrac{Var_{\xi \in \tilde{S}_{k-1}}(\nabla F(x_{k, 0}, \xi) | x_{k, 0})}{\tilde{\theta}^2 \tilde{Z}_{k}^2 + \tilde{a}^2 \tilde{\beta}^{2k}}.
\end{align*}
For \cref{cond:eq_norm_test}, $\tilde{Z}_k = \|T_{\tilde{S}_{k-1}}(x_{k, 0}, \lambda_{k-1, N_{k-1}})\|$ can be readily computed as an estimate of $Z_k$. For Conditions \ref{cond:eq_norm_test_d}, \ref{cond:eq_norm_test_delta_l}, and \ref{cond:ineq_norm_test}, we form the subsampled problem $F_{\tilde{S}_{k-1}}$ and perform a single inner loop iteration initialized at $x_{k, 0}$ to compute the corresponding value of $\tilde{Z}_k$. No updates are made during this iteration. This adds the computational cost of one inner iteration, but incurs no additional gradient cost, as the computed gradients can be reused in the first inner iteration of outer iteration $k$. To enable this, the sample set $S_k$ is constructed by augmenting $\tilde{S}_{k-1}$ with additional samples, as needed, ensuring that $\tilde{S}_{k-1} \subseteq S_k$. 

The geometric component in the adaptive sampling condition is intended to prevent large, erratic increases in batch sizes that may result from noisy sample estimates. In practice, we replace this component with a hard geometric increase limit, restricting batch size growth to at most a factor of $\hat{\beta}>1$ between successive outer iterations.
Additionally, batch sizes are enforced to be non-decreasing and are bounded above by the total number of samples in the finite-sum problem \eqref{eq:intro_deter_error_obj}. Therefore, the batch size at outer iteration $k$ is selected as
\begin{equation*}
    |S_k| = \min\left\{|\Scal|, \hat{\beta}|S_{k-1}|, \max\left\{|S_{k-1}|, \left\lceil \tfrac{Var_{\xi \in \tilde{S}_{k-1}}(\nabla F(x_{k, 0}, \xi) | x_{k, 0})}{\tilde{\theta}^2 \tilde{Z}_{k}^2}\right\rceil \right\}\right\}.
\end{equation*}
We set $\tilde{\theta} = 0.5$ and $\hat{\beta} = 5$ for the proposed algorithms and the batch size is initialized at 32. While prior works such as \cite{berahas2022adaptive, bollapragada2018adaptive, o2024fast} often use larger values of $\tilde{\theta}$, we prefer a smaller value to allow for more significant adjustments to the batch size,
given that multiple updates are performed between successive evaluations of the sampling condition.

\subsection{Equality Constrained Problems} \label{sec:Equality_constrained_numerical_exps}
We first consider problems with only general nonlinear equality constraints. We evaluate the performance of five variants of \cref{alg:Equality_Constrained_RA_SQP}, each referred to as ``RA-SQP" with additional descriptors, differing across three key characteristics. 
First, the termination criterion $\Tcal_k$ is either based on the norm of the search direction \eqref{eq:termination_criterion_d}, labeled ``$\|d\|$" (with $\gamma_k = \gamma = 0.5$), or on the decrease in the merit function model \eqref{eq:termination_criterion_delta_l}, labeled ``$\Delta l$" (with $\gamma_k = \gamma = 0.1$). 
Second, the Hessian of the Lagrangian is approximated either by the identity matrix (default) or using an L-BFGS update (with memory size as $\min\{n, 10\}$), labeled ``L-BFGS". 
Third, the linear system in the \texttt{SQP} subproblem \eqref{eq:eq_sqp_system} is solved via MINRES either exactly (default option to a relative error tolerance of $10^{-6}$) or approximately, using inexact solutions that satisfy \eqref{eq:eq_inexactness_condition_1} or \eqref{eq:eq_inexactness_condition_2} with $\kappa_{T} = 10^{-1}$, $\epsilon_{feas} = 10^{-4}$ and $\epsilon_{opt} = 10^{-4}$, as discussed in \cref{sec:SQP_based_algorithm_eq}, and labeled ``Inexact". For the termination criterion based on $\|d\|$ \eqref{eq:termination_criterion_d}, we only report results with the identity Hessian and exact linear system solutions, as the results for other configurations are not competitive. To avoid solving subsampled subproblems to unnecessary accuracy, the additional tolerance in the termination criterion is set to $\epsilon_k = 10^{-6}$.
Furthermore, to prevent excessive computation on subproblems, we limit the number of inner iterations per outer iteration to a maximum of $500$ for reasons discussed in \cref{sec:SQP_based_algorithm_eq_theory}.

We implement and compare against the stochastic \texttt{SQP} method \cite{berahas2021sequential}, labeled ``S-SQP", and the adaptive sampling \texttt{SQP} method \cite{berahas2022adaptive}, labeled ``AS-SQP". We also include comparisons with the deterministic \texttt{SQP} method \cite{nocedal2006numerical,berahas2021sequential}, labeled ``SQP", applied to the finite-sum logistic regression problems. In all three baselines, ``SQP", ``S-SQP", and ``AS-SQP", we use the identity matrix as the Lagrangian Hessian approximation and solve the \texttt{SQP} linear system exactly. For ``S-SQP", we use a batch size of $1024$, which is larger than commonly used, to ensure the method reaches solutions with accuracy comparable to the proposed \texttt{RA}-based algorithms. For ``AS-SQP", the batch size is initialized at $32$ and increases according to the adaptive sampling conditions of the algorithm.

We employ the MINRES algorithm to solve the \texttt{SQP} linear system using its scipy implementation \cite{2020SciPy-NMeth}. To evaluate algorithm performance, we track two metrics: the constraint violation $\|c(x)\|_{\infty}$ and the norm of the Lagrangian gradient $\|\nabla_x \Lcal(x, \lambda^*)\|_{\infty}$, where $\lambda^* \in \Rmbb^m$ is the dual variable that minimizes the KKT error for the true problem at $x$. These metrics are plotted against two cost measures: the number of gradient evaluations and the total number of MINRES iterations used to solve the linear systems. For the RA-based algorithms, the MINRES count also includes the cost of solving the subproblem used to estimate quantities in the sampling condition.

\subsubsection{Logistic Regression}

We consider the logistic regression multi-class classification problem \eqref{eq:logreg_obj} with equality regularization constraints, i.e., 
$\|x^i\|^2 = 1$ for all $i \in \Kcal$.
In Figures \ref{fig:covtype_equality} and \ref{fig:mnist_equality}, we present comparisons against existing methods on the covtype and mnist datasets, respectively. For a fair comparison, all algorithms use the identity matrix as the approximation for the Lagrangian Hessian. When comparing ``RA-SQP $\|d\|$" and ``RA-SQP $\Delta l$" with algorithms from the literature, all of which use exact solutions to the \texttt{SQP} linear system, our algorithms achieve better performance in terms of the number of gradient evaluations and are comparable to deterministic \texttt{SQP} in terms of MINRES iterations. Furthermore, the flexibility of our framework is evident from the improved performance of ``RA-SQP $\Delta l$ Inexact", which uses inexact solutions to the \texttt{SQP} linear system. This variant significantly reduces the number of MINRES iterations compared to ``RA-SQP $\Delta l$", demonstrating how performance can be enhanced by modifying the underlying deterministic \texttt{SQP} method.

\begin{figure}[H]
    \centering
    \includegraphics[width=\textwidth]{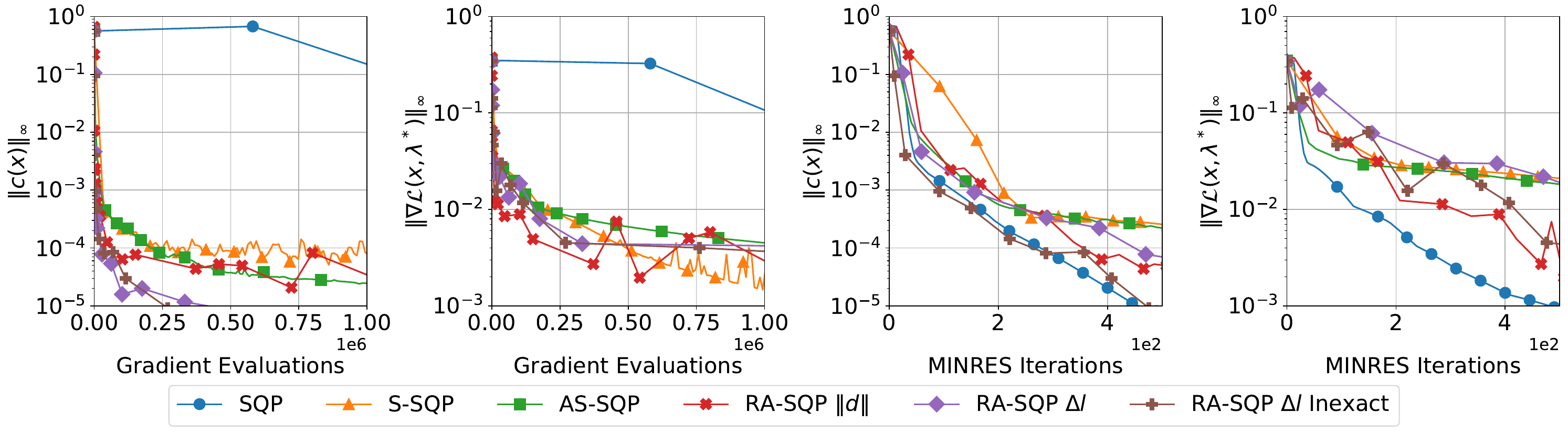}
    \caption{Constraint violation ($\|c(x)\|_{\infty}$) and Lagrangian gradient norm ($\|\nabla_x \Lcal(x,\lambda^*)\|_{\infty}$) with optimized dual variable $\lambda^*$, with respect to number of gradient evaluations and number of MINRES iterations for stochastic \texttt{SQP} ( ``S-SQP" \cite{berahas2021sequential}), adaptive sampling \texttt{SQP} (``AS-SQP" \cite{berahas2022adaptive}), deterministic \texttt{SQP} (``SQP" \cite{nocedal2006numerical,berahas2021sequential}) and our proposed algorithms ``RA-SQP $\|d\|$", ``RA-SQP $\Delta l$", and ``RA-SQP $\Delta l$ Inexact" over the multi-class logistic regression problem \eqref{eq:logreg_obj} with equality regularization constraints for the covtype dataset ($n_f = 55$, $|\Kcal| = 7$, $|\Scal| = 581,012$, \cite{CC01a}).}
    \label{fig:covtype_equality}
\end{figure}

\begin{figure}[H]
    \centering
    \includegraphics[width=\textwidth]{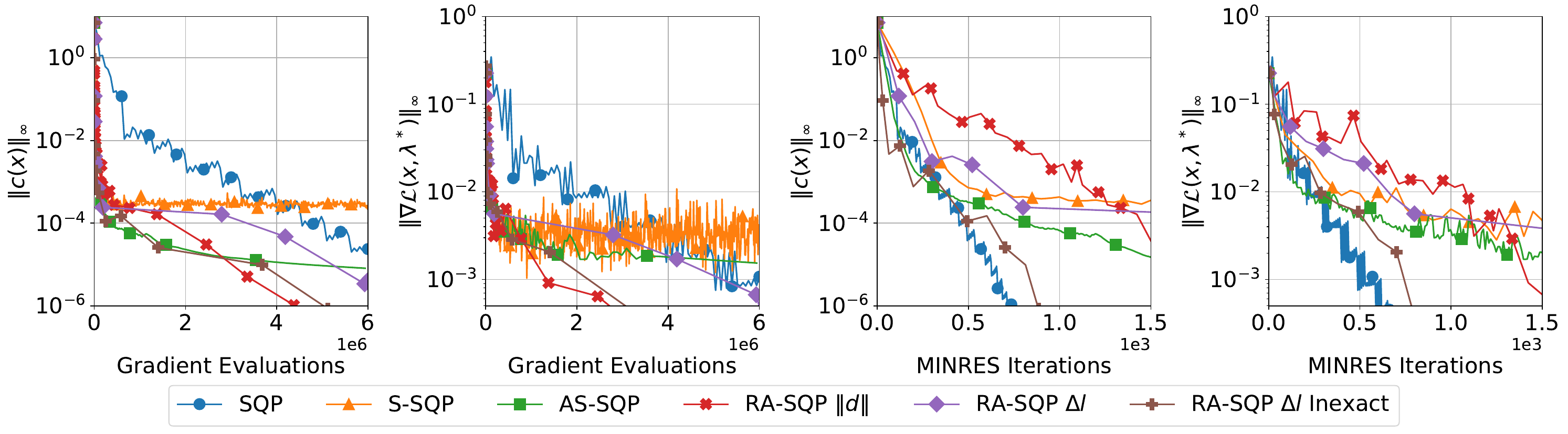}
    \caption{Constraint violation ($\|c(x)\|_{\infty}$) and Lagrangian gradient norm ($\|\nabla_x \Lcal(x,\lambda^*)\|_{\infty}$) with optimized dual variable $\lambda^*$, with respect to number of gradient evaluations and number of MINRES iterations for stochastic \texttt{SQP} ( ``S-SQP" \cite{berahas2021sequential}), adaptive sampling \texttt{SQP} (``AS-SQP" \cite{berahas2022adaptive}), deterministic \texttt{SQP} (``SQP" \cite{nocedal2006numerical,berahas2021sequential}) and our proposed algorithms ``RA-SQP $\|d\|$", ``RA-SQP $\Delta l$", and ``RA-SQP $\Delta l$ Inexact" over the multi-class logistic regression problem \eqref{eq:logreg_obj} with equality regularization constraints for the mnist dataset ($n_f = 781$, $|\Kcal| = 10$, $|\Scal| = 60,000$, \cite{CC01a}).}
    \label{fig:mnist_equality}
\end{figure}

In Figure \ref{fig:mnist_equality_L-BFGS}, we present results for our proposed algorithms on the mnist dataset, highlighting the flexibility of the \texttt{RA} framework and the performance gains achievable through enhancements such as the L-BFGS Hessian approximation and inexact solutions to the \texttt{SQP} linear system. 
Incorporating the L-BFGS approximation improves performance in terms of gradient evaluations compared to the identity Hessian counterparts. 
However, adopting the L-BFGS approximation results in harder to solve \texttt{SQP} linear systems for the higher quality search directions as depicted with the increase in MINRES iterations when calculating exact solutions. Adopting inexact solutions to the linear system with the L-BFGS approximation reduces this cost.
These results demonstrate how techniques from deterministic \texttt{SQP} methods can be seamlessly integrated into the \texttt{RA} framework to enhance its effectiveness on stochastic optimization problems.

\begin{figure}[H]
    \centering
    \includegraphics[width=\textwidth]{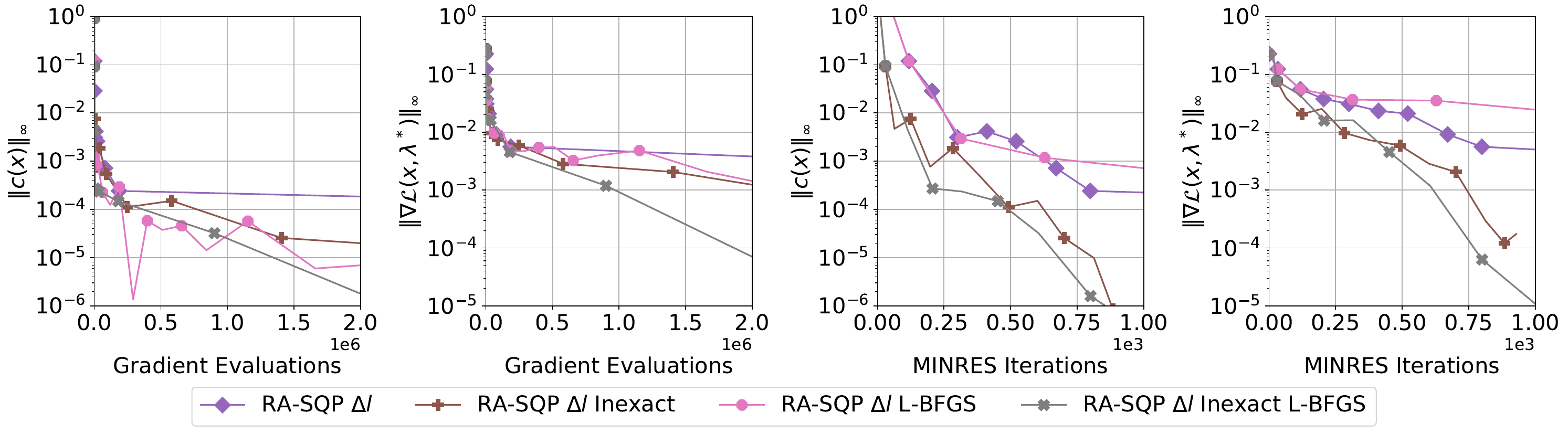}
    \caption{Constraint violation ($\|c(x)\|_{\infty}$) and Lagrangian gradient norm ($\|\nabla_x \Lcal(x,\lambda^*)\|_{\infty}$) with optimized dual variable $\lambda^*$, with respect to number of gradient evaluations and number of MINRES iterations for ``RA-SQP $\Delta l$" with and without L-BFGS approximations and exact and inexact 
    \texttt{SQP} linear system solutions over the multi-class logistic regression problem \eqref{eq:logreg_obj} with equality regularization constraints for the mnist dataset ($n_f = 781$, $|\Kcal| = 10$, $|\Scal| = 60,000$, \cite{CC01a}).}
    \label{fig:mnist_equality_L-BFGS}
\end{figure}

In Figure \ref{fig:mnist_equality_L-BFGS_inner_iterations}, we present how the number of inner iterations $(N_k)$ and the batch size ($|S_k|$) change across outer iterations and with respect to the number of gradient evaluations for the results presented in Figure \ref{fig:mnist_equality_L-BFGS}. First, the number of inner iterations increases during the initial phase for all methods and subsequently stabilizes, never reaching the upper limit of 500 inner iterations. This behavior is desirable, as it aligns with the growth of the batch size.
Second, while the rate of increase for the batch sizes is different for each method across outer iterations depending on the progress made, it is relatively consistent when measured against the number of gradient evaluations.
This adaptive behavior is one of the main reasons for the good empirical performance of the framework with respect to the number of gradient evaluations.
Third, the methods that use an L-BFGS approximation perform fewer inner iterations, as expected.
Finally, when the L-BFGS approximation is used, the batch size increases slower than when the identity matrix is used across outer iterations when exact solutions to the \texttt{SQP} linear system are employed. Conversely, the increase is faster when inexact solutions are employed.
As discussed in \cref{sec:General_Deterministic_Solver_theory}, ``RA-SQP $\Delta l$ L-BFGS" frequently solves the subsampled problems to a higher accuracy than required by the inner loop termination test $\Tcal_k$, introducing additional bias in the solutions. In such cases, a slower increase in the batch size works better. 
Alternatively, a faster batch size increase can be employed by decreasing $\Tilde{\theta}$ to take full advantage of this behavior. We note again that we do not tune the value of $\tilde{\theta}$ across methods to display the robustness of the framework.

\begin{figure}[H]
    \centering
    \includegraphics[width=\textwidth]{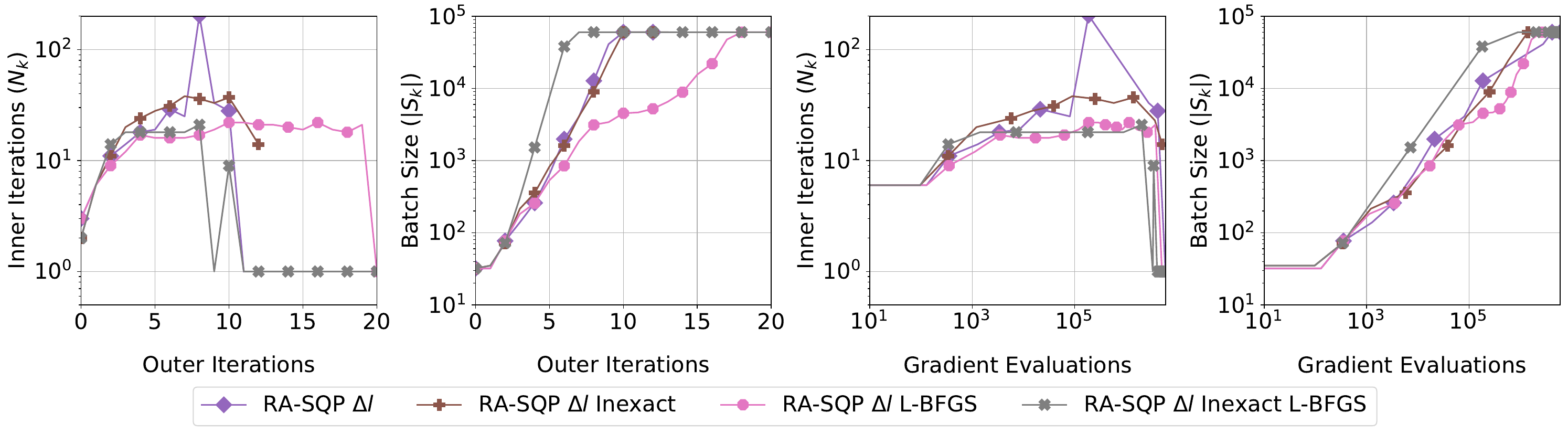}
    \caption{Number of Inner Iterations ($N_k$) and Batch Size ($|S_k|$) with respect to outer iterations and number of gradient evaluations for ``RA-SQP $\Delta l$" with and without L-BFGS approximations and exact and inexact \texttt{SQP} linear system solutions over the multi-class logistic regression problem \eqref{eq:logreg_obj} with equality regularization constraints for the mnist dataset ($n_f = 781$, $|\Kcal| = 10$, $|\Scal| = 60,000$, \cite{CC01a}).}
    \label{fig:mnist_equality_L-BFGS_inner_iterations}
\end{figure}

\subsubsection{CUTEst Problems}

In this section, we evaluate the empirical performance of \cref{alg:Equality_Constrained_RA_SQP} on the CUTEst problem set, using the augmented formulation in \eqref{eq:cutest_problem}, and compare our results with methods from the literature. We focus on problems that have only equality constraints, satisfy the constraint qualifications (verified via solutions obtained using deterministic \texttt{SQP}), possess a non-constant objective, and have total dimension $n + m < 1000$. This filtering yields a benchmark suite of $88$ problems. For each problem, we perform $10$ runs with different random seeds, resulting in a total of $880$ problem instances per algorithm. Each run is allotted a gradient evaluation budget of $10^6$.

In Figure \ref{fig:CUTEST_equality}, we present the results using Dolan-Mor\'e performance profiles \cite{dolan2002benchmarking} with respect to the number of gradient evaluations and MINRES iterations, evaluating both feasibility and stationarity errors. The performance profiles are constructed as follows. For each cost metric, an algorithm is considered to have successfully solved a problem instance in a given seed run at iterate $x_{out}$,  starting from initial point $x_{init}$, if it satisfies a prescribed tolerance $\epsilon_{tol} \in (0, 1)$. Specifically, the algorithm is deemed to have solved a given problem with respect to feasibility at $x_{out}$ if $\|c(x_{out})\|_{\infty} \leq \epsilon_{tol} \max\{1, \|c(x_{init})\|_{\infty}\}$. The problem is further considered solved with respect to stationarity at $x_{out}$ if it satisfies the feasibility criterion and, in addition, $\|\nabla_x \Lcal(x_{out},\lambda_{out}^*)\|_{\infty} \leq \epsilon_{tol} \max\{1, \|\nabla_x \Lcal(x_{init}, \lambda^*_{init})\|_{\infty}\}$, where $\lambda_{out}^*$ and $\lambda^*_{init}$ denote the dual variables that minimize the true problem's KKT error at $x_{out}$ and $x_{init}$ respectively. Overall, across all tolerance levels, the \texttt{RA}-based algorithms demonstrate strong initial progress and exhibit greater robustness. The results also highlight the performance gains from incorporating the L-BFGS Hessian approximation, particularly at tighter tolerance levels where higher solution accuracy is required. 

\begin{figure}[H]
    \centering
    \includegraphics[width=\textwidth]{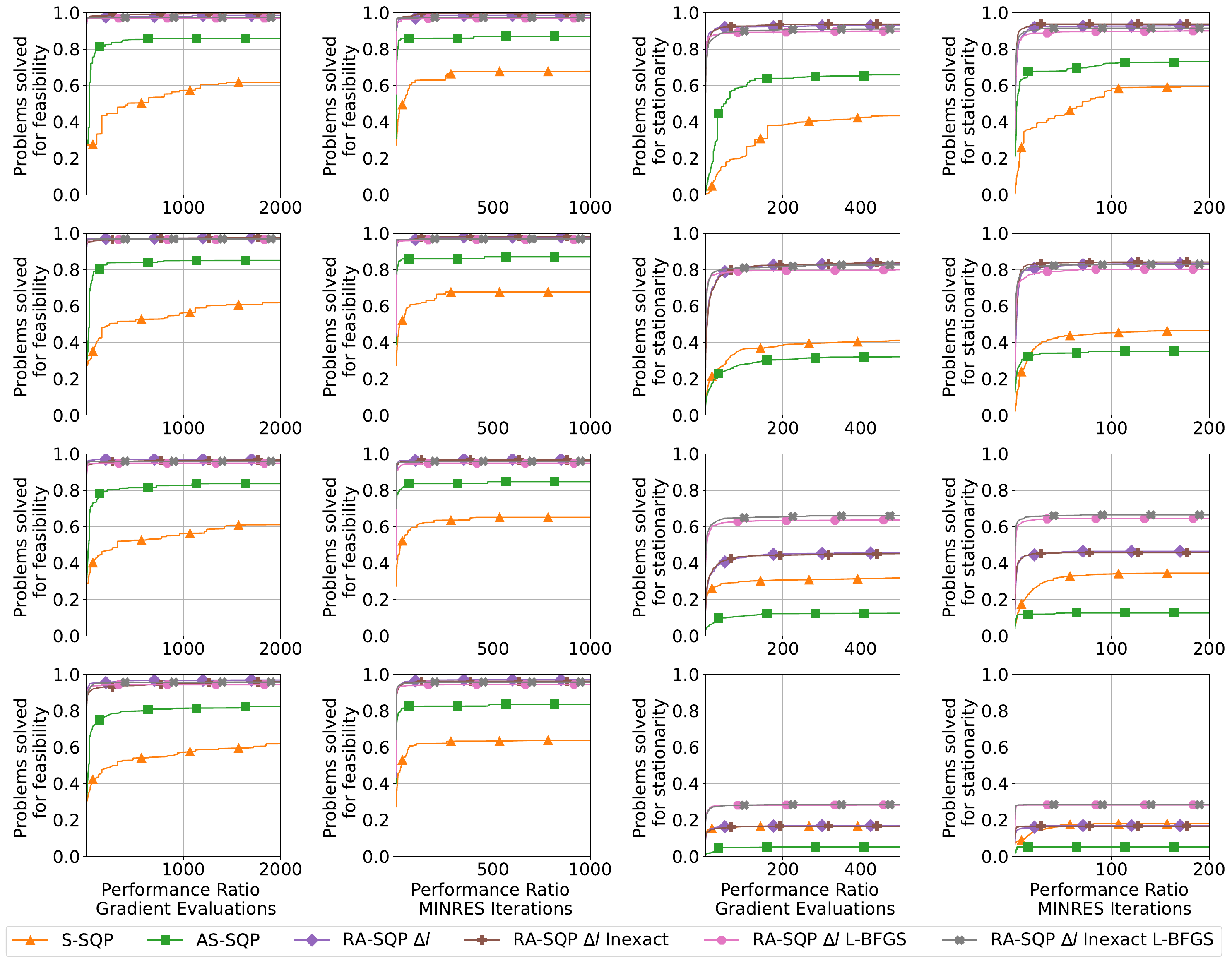}
    \caption{Performance profiles for feaibility and stationarity errors with respect to number of gradient evaluations and number of MINRES iterations for stochastic \texttt{SQP} ( ``S-SQP" \cite{berahas2021sequential}), adaptive sampling \texttt{SQP} (``AS-SQP" \cite{berahas2022adaptive}) and our proposed algorithm ``RA-SQP $\Delta l$" with and without L-BFGS Hessian approximations and exact and inexact \texttt{SQP} linear system solutions over the 
    CUTEst problem set \cite{gratton2024s2mpj} 
    to accuracy levels $\epsilon_{tol} \in \{10^{-1}, 10^{-2}, 10^{-3}, 10^{-4}\}$ that are decreasing going down the rows.}
    \label{fig:CUTEST_equality}
\end{figure}

\subsection{Inequality Constrained Problems} \label{sec:Inequality_constrained_numerical_exps}
In this section, we illustrate the empirical performance of \cref{alg:Inequality_Constrained_RA_SQP} for solving problems with general nonlinear constraints. We consider regularized logistic regression multi-class classification problems as well as 
the CUTEst problem set, and compare our methods with algorithms from the literature. We also empirically evaluate the stability of the active set in the \texttt{RA} framework across outer iterations.

We implement and compare our algorithm with stochastic \texttt{SQP} \cite{curtis2023sequential} (labeled ``S-SQP") and stochastic robust \texttt{SQP} \cite{qiu2023sequential} (labeled ``Robust-S-SQP"). We also include comparisons with deterministic robust \texttt{SQP} methods \cite{burke1989robust,qiu2023sequential}, using the $l_\infty$ norm (labeled ``Robust-SQP-$l_\infty$") and the $l_1$ norm (labeled ``Robust-SQP-$l_1$") for the finite-sum logistic regression problems. For fair comparison, all baseline algorithms use the identity matrix as the approximation to the Lagrangian Hessian. The batch size for both ``S-SQP" and ``Robust-SQP" is set to 1024, larger than typically used, to allow these methods to attain solution accuracies comparable to those achieved by the proposed \texttt{RA}-based algorithms.

We present two versions of \cref{alg:Inequality_Constrained_RA_SQP}, each differing in the robust \texttt{SQP} method employed within the inner loop. The first uses the $l_\infty$ norm and is labeled ``RA-SQP-$l_\infty$", while the second employs the $l_1$ norm and is labeled ``RA-SQP-$l_1$". Both versions use the termination criterion \eqref{eq:termination_criterion_ineq}, based on the search direction norm, with $\gamma_k = \gamma = 0.5$, and use the identity matrix to approximate the Lagrangian Hessian. Following the setup from the equality-constrained case, to prevent excessive computation on subsampled problems, we set $\epsilon_k = 10^{-6}$ and 
cap the number of inner iterations per outer iteration at 500 for reasons discussed in \cref{sec:Inquality_constrained_analysis}.
The bound on the search direction norm in the \texttt{SQP} subproblems is set as $\sigma_d = 2\sigma_p$, where $\sigma_p$ is proportional to the constraint violation, similar to the strategy in \cite{qiu2023sequential}. Specifically, for the $l_{\infty}$ norm, we set $\sigma_p = 10\left\|\begin{matrix}c_E(x) \\ [c_I(x)]_+ \end{matrix}\right\|_{\infty}$, while ensuring $\sigma_p \in [10^2, 10^4]$. For the $l_1$ norm, we set  $\sigma_p = 10\left\|\begin{matrix}c_E(x) \\ [c_I(x)]_+ \end{matrix}\right\|_{1}$, constrained to the range $[n \cdot 10^2, n \cdot 10^4]$ scaling the range by the dimension of the problem to account for the scaling difference in $l_1$ and $l_\infty$ norms.

We employ GUROBI \cite{gurobi} to solve the linear and quadratic programs that arise in the \texttt{SQP} subproblems. The progress of the algorithms is evaluated using two primary metrics: the constraint violation, defined as $ \|c(x)\|_{\infty}= \left\|\begin{matrix}c_E(x) \\ [c_I(x)]_+ \end{matrix}\right\|_{\infty}$, and the KKT residual, which is computed at the point $x$ as the objective value of the linear program
\begin{align*}
    KKT(x) = \min_{\substack{t \in \Rmbb, \lambda_E \in \mathbb{R}^{m_E}, \lambda_I \in \mathbb{R}^{m_I}}} & t \numberthis \label{eq:KKT_residual_ineq} \\
    s.t. \quad \qquad & \|\nabla f(x) + \nabla c_E(x) \lambda_E + \nabla c_I(x) \lambda_I\|_{\infty} \leq t, \\
    & \|\lambda_I \odot c_I(x)\|_{\infty} \leq t, \quad \lambda_I \geq 0,
\end{align*}
that measures violation of the general KKT conditions \eqref{eq:KKT_conditions} using the best dual variables.
These metrics are evaluated with respect to two computational costs: the number of gradient evaluations and the number of Barrier method iterations performed by GUROBI in solving the associated linear and quadratic programs. For the \texttt{RA}-based algorithms, the Barrier iterations include those required to solve subproblems for estimating quantities in the sampling condition.

\subsubsection{Logistic Regression}
We evaluate the empirical performance of \cref{alg:Inequality_Constrained_RA_SQP} on the logistic regression multi-class classification problem \eqref{eq:logreg_obj} with inequality regularization constraints, i.e., $\|x^i\|^2 \leq 1$ for all $i \in \Kcal$.
In Figures \ref{fig:covtype_inequality} and \ref{fig:mnist_inequality}, we compare our algorithm with methods from the literature on the covtype and MNIST datasets, respectively. When comparing ``RA-SQP-$l_\infty$" and ``RA-SQP-$l_1$" with existing methods, our algorithms outperform the others in terms of the number of gradient evaluations and achieve performance close to deterministic \texttt{SQP} methods with respect to Barrier iterations.
Furthermore, we observe a difference in Barrier iteration performance between ``RA-SQP-$l_1$" and ``RA-SQP-$l_\infty$", with the $l_\infty$ variant performing fewer Barrier iterations. This is likely due to the smaller size of the \texttt{SQP} subproblems when using the $l_\infty$ norm, which reduces solver overhead without affecting performance with respect to gradient evaluations.

\begin{figure}[H]
    \centering
    \includegraphics[width=\textwidth]{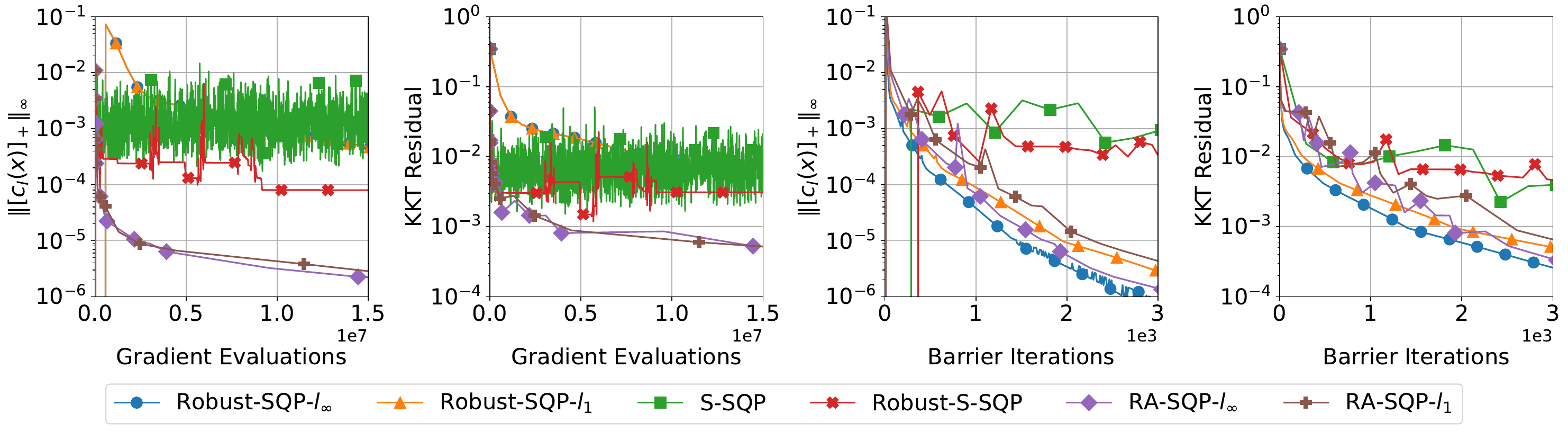}
    \caption{Constraint violation ($\|[c_I(x)]_+\|_{\infty}$) and KKT residual \eqref{eq:KKT_residual_ineq} with respect to number of gradient evaluations and number of Barrier iterations for stochastic \texttt{SQP} ( ``S-SQP" \cite{curtis2023sequential}), robust stochastic \texttt{SQP} ( ``Robust-S-SQP" \cite{qiu2023sequential}), deterministic robust \texttt{SQP} (``Robust-SQP-$l_\infty$" and ``Robust-SQP-$l_1$" \cite{burke1989robust,qiu2023sequential}) and our proposed algorithms ``RA-SQP-$l_\infty$" and ``RA-SQP-$l_1$" over the multi-class logistic regression problem \eqref{eq:logreg_obj} with inequality regularization constraints for the covetype dataset ($n_f = 55$, $|\Kcal| = 7$, $|\Scal| = 581,012$, \cite{CC01a}).}
    \label{fig:covtype_inequality}
\end{figure}

\begin{figure}[H]
    \centering
    \includegraphics[width=\textwidth]{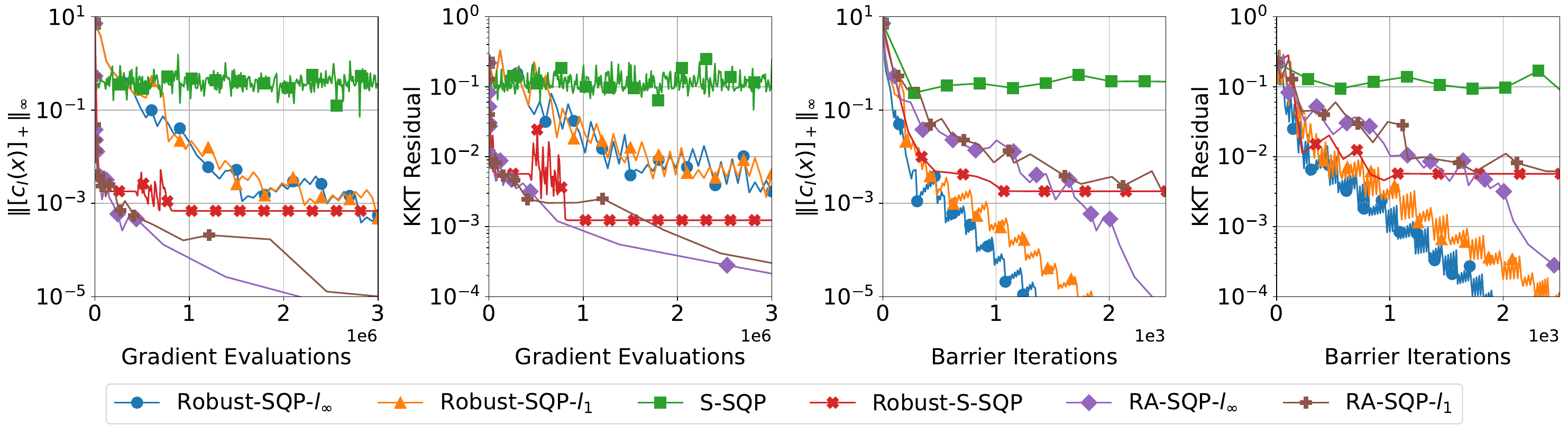}
    \caption{Constraint violation ($\|[c_I(x)]_+\|_{\infty}$) and KKT residual \eqref{eq:KKT_residual_ineq} with respect to number of gradient evaluations and number of Barrier iterations for stochastic \texttt{SQP} ( ``S-SQP" \cite{curtis2023sequential}), robust stochastic \texttt{SQP} ( ``Robust-S-SQP" \cite{qiu2023sequential}), deterministic robust \texttt{SQP} (``Robust-SQP-$l_\infty$" and ``Robust-SQP-$l_1$" \cite{burke1989robust,qiu2023sequential}) and our proposed algorithms ``RA-SQP-$l_\infty$" and ``RA-SQP-$l_1$" over the multi-class logistic regression problem \eqref{eq:logreg_obj} with inequality regularization constraints for the mnist dataset ($n_f = 781$, $|\Kcal| = 10$, $|\Scal| = 60,000$, \cite{CC01a}).}
    \label{fig:mnist_inequality}
\end{figure}

\subsubsection{CUTEst Problems} 

In this section, we evaluate the empirical performance of \cref{alg:Inequality_Constrained_RA_SQP} on the CUTEst problem set in the form \eqref{eq:cutest_problem}, and compare our algorithms with methods from the literature. We consider problems that contain general inequality constraints (excluding problems with only bound constraints), satisfy constraint qualifications (verified using a deterministic \texttt{SQP} solver), and have a non-constant objective function. Additionally, we restrict the problem size to $n + m_E + m_I < 2000$, resulting in a final set of $248$ problems. For each problem, we perform $10$ independent runs with different random seeds, totaling $2480$ runs per algorithm. A gradient evaluation budget of $10^6$ is allocated for each problem instance and seed run.

In Figure \ref{fig:CUTEST_inequality}, we present results using Dolan-Mor\'e performance profiles \cite{dolan2002benchmarking}, evaluating both feasibility and stationarity errors with respect to the number of gradient evaluations and Barrier iterations. The profiles are constructed similarly to those in \cref{fig:CUTEST_equality}, as follows. For each cost metric, an algorithm is considered to have successfully solved a problem instance in a given seed run at iterate $x_{out}$, starting from initial iterate $x_{init}$, if it satisfies a prescribed tolerance $\epsilon_{tol} \in (0, 1)$. Specifically, the algorithm is deemed to have solved a problem with respect to feasibility if $\|c(x_{out})\|_{\infty} \leq \epsilon_{tol} \max\{1, \|c(x_{init})\|_{\infty}\}$. The algorithm is further considered to have solved the problem with respect to stationarity if it satisfies the feasibility criterion and $KKT(x_{out}) \leq \epsilon_{tol} \max\{1, KKT(x_{init})\}$, where the KKT residual is defined in \eqref{eq:KKT_residual_ineq}.

Overall, across all tolerance levels, the \texttt{RA}-based algorithms demonstrate good performance in the early stages, are robust and achieve high accuracy. While we observe that ``S-SQP" slightly outperforms the \texttt{RA}-based methods at the highest accuracy level with respect to stationarity, the \texttt{RA}-based algorithms remain efficient in reducing constraint violation.

\begin{figure}[H]
    \centering
    \includegraphics[width=\textwidth]{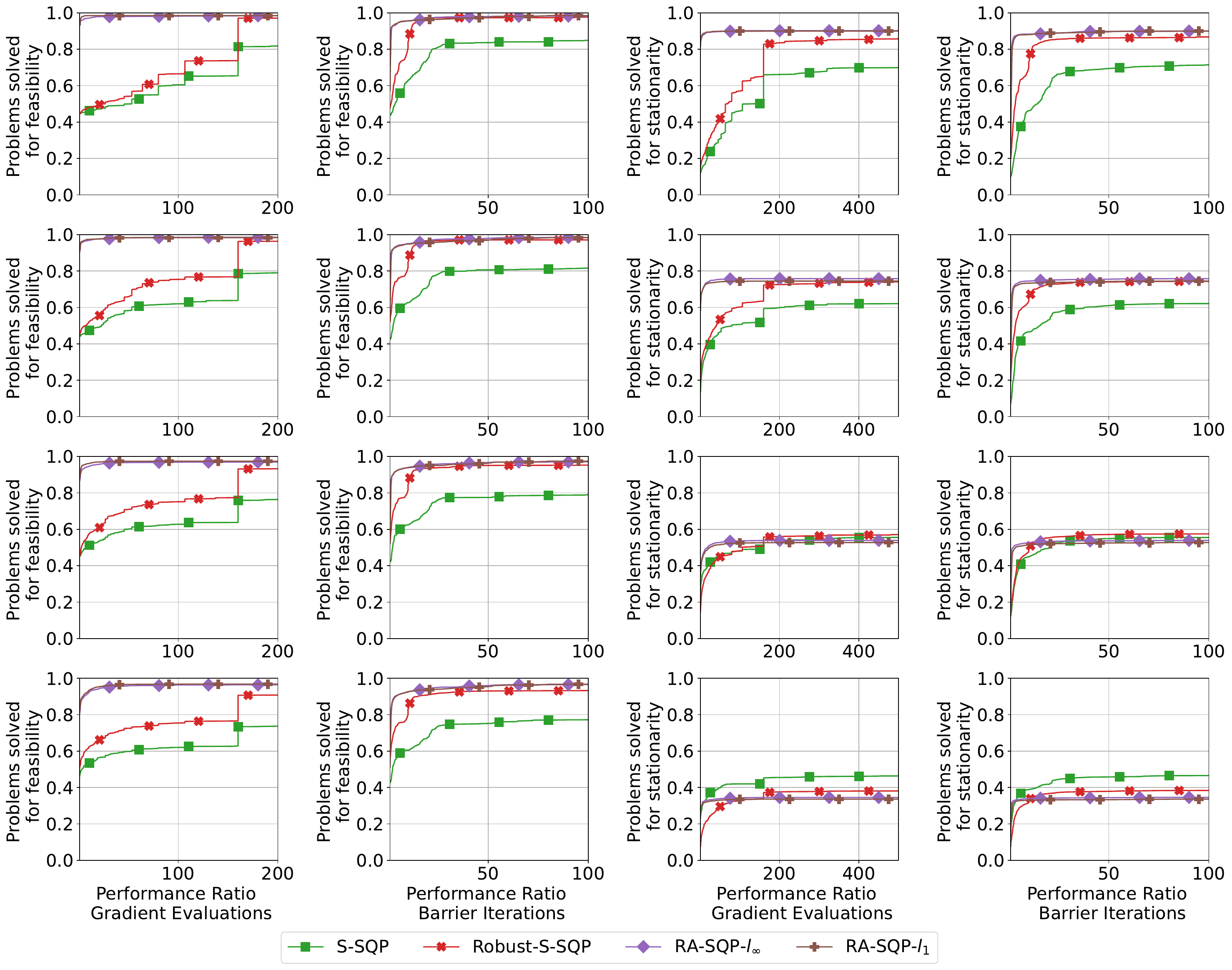}
    \caption{Performance profiles for feasibility and stationarity errors with respect to number of gradient evaluations and number of Barrier iterations for stochastic \texttt{SQP} ( ``S-SQP" \cite{curtis2023sequential}), robust stochastic \texttt{SQP} ( ``Robust-S-SQP" \cite{qiu2023sequential}) and our proposed algorithms ``RA-SQP-$l_\infty$" and ``RA-SQP-$l_1$" over the CUTEst problem set \cite{gratton2024s2mpj} to accuracy levels $\epsilon_{tol} \in \{10^{-1}, 10^{-2}, 10^{-3}, 10^{-4}\}$ decreasing going down the rows.}
    \label{fig:CUTEST_inequality}
\end{figure}

\subsubsection{Active Set Stability}
In this section, we empirically investigate the stability of the active set for problems with general nonlinear constraints across the outer iterations of \cref{alg:Inequality_Constrained_RA_SQP}. Since the constraints are deterministic while the objective function varies between outer iterations due to subsampling, it is desirable for the algorithm to identify and maintain a stable active set after a few outer iterations, particularly once the batch size in the subsampled problem becomes sufficiently large. To study this behavior, we select problems from the 
CUTEst problem set \cite{gratton2024s2mpj} that feature a large number of inequality constraints ($m_I$), where the number of active inequality constraints at the optimal solution ($m_I^*$) constitutes only a small subset.

To measure this, we first define the active set of inequality constraints at a point $x$ as $A(x) = \{i : [c_I(x)]_i \geq -10^{-6}\}$, where the threshold $10^{-6}$ corresponds to the feasibility tolerance used by GUROBI when solving the \texttt{SQP} subproblems. Let $A^*$ denote the active set at the optimal solution obtained using a deterministic solver. We employ two metrics to assess active set stability: 
``optimal set similarity", defined as the Jaccard similarity between the current active set and the optimal active set, $\tfrac{|A(x) \cap A^*|}{|A(x) \cup A^*|}$, and the constraint violation.
These metrics are evaluated over outer iterations for \cref{alg:Inequality_Constrained_RA_SQP} with the $l_1$ norm, and with respect to the number of gradient evaluations to facilitate comparison with ``S-SQP" and ``Robust-S-SQP" methods.
While we have used a large batch size of $1024$ for both ``S-SQP" and ``Robust-S-SQP" in our previous experiments to enable comparable accuracy, we also include results for both with a smaller batch size of $128$ in this analysis, labeled ``S-SQP (small)" and ``Robust-S-SQP (small)", respectively.

The results are summarized in Figure \ref{fig:Active_set_stability}. Achieving the optimal active set is only meaningful if the corresponding solution is feasible. For the BATCH problem, both batch sizes of ``S-SQP" fail to produce feasible solutions.
While ``RA-SQP-$l_1$", ``Robust-S-SQP" and ``Robust-S-SQP (small)" converge to the optimal active set, the feasibility error remains too large to claim successful recovery.
For the GPP problem, all algorithms eventually achieve feasibility.
However, both ``S-SQP" variants exhibit significant volatility in feasibility error and active set identification, especially the small batch size variant, and fail to stabilize on a consistent active set.
Similar instability is observed in the behavior of ``Robust-S-SQP", with ``Robust-S-SQP (small)" failing to identify the optimal active set. In contrast, ``RA-SQP-$l_1$", while slower, eventually recovers a feasible solution along with the optimal active set and maintains it consistently.
For the OET1 problem, all algorithms rapidly achieve and maintain feasibility. 
However, with the exception of ``RA-SQP-$l_1$", all other methods exhibit erratic changes in their active sets across iterations or fail to identify the optimal active set.
While this is an empirical discussion, the observations motivate further theoretical investigation into active set stability within the \texttt{RA}-based framework, potentially paving the way for an active set method tailored to inequality constrained problems.

\begin{figure}[H]
    \centering
    \includegraphics[width=\textwidth]{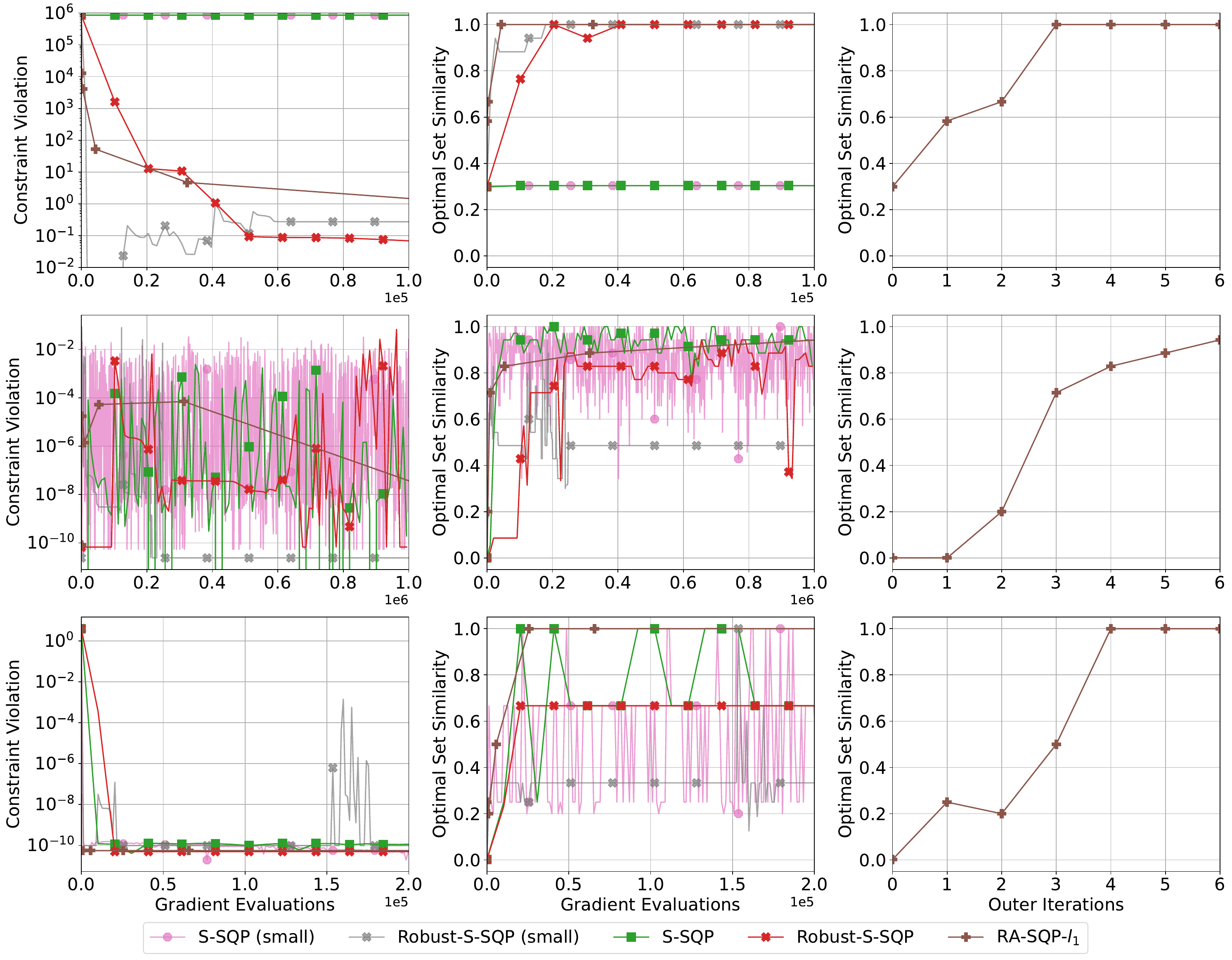}
    \caption{Constraint violation and optimal set similarity with respect to number of outer iterations and gradient evaluations for ``RobustS-SQP (small)", ``S-SQP (small)", ``S-SQP", ``Robust-S-SQP" and ``RA-SQP-$l_1$" across selected problems from the CUTEst problem set \cite{gratton2024s2mpj}. Each row corresponds to a different problem, listed from top to bottom as follows: BATCH($m_I=61$, $m_I^*=17$), GPP($m_I=198$, $m_I^*=35$) and OET1($m_I=1002$, $m_I^*=3$).  
    }
    \label{fig:Active_set_stability}
\end{figure}
\section{Final Remarks} \label{sec:Conclusion}

In this paper, we have proposed an algorithmic framework based on the Retrospective Approximation (\texttt{RA}) paradigm to solve stochastic optimization problems with deterministic constraints using deterministic optimization methods. The use of deterministic optimization methods has addressed challenges such as small step sizes and slow convergence that arise in recent stochastic approaches. They also have enabled the use of techniques from deterministic optimization, such as line search and quasi-Newton Hessian approximations, to improve empirical performance, techniques that are not readily applicable to the stochastic setting. We have presented a framework designed for problems with general nonlinear equality constraints that can leverage any deterministic solver for a stochastic problem. We have followed that with an instance of the framework that employs a deterministic line search Sequential Quadratic Programming (\texttt{SQP}) method. This instance is endowed with optimal complexity with respect to both gradient evaluations and \texttt{SQP} linear system solves for stochastic problems with equality constraints, and has allowed for the effective use of advanced techniques such as quasi-Newton Hessian approximations and inexact search direction computations, as demonstrated through our numerical experiments. For problems with general nonlinear equality and inequality constraints, we have proposed an \texttt{RA}-based algorithm that uses the deterministic robust-\texttt{SQP} method as the deterministic solver. The method ensures the consistency of the \texttt{SQP} subproblem in the presence of inequality constraints while identifying infeasible stationary points and is efficient as demonstrated in our numerical experiments.
We present empirical results over multiple problems and compare with algorithms from literature to showcase the competitive performance of the proposed framework.

\bibliographystyle{plain}
\bibliography{references.bib}

\begin{thebibliography}{10}

\bibitem{berahas2022adaptive}
Albert~S Berahas, Raghu Bollapragada, and Baoyu Zhou.
\newblock An adaptive sampling sequential quadratic programming method for equality constrained stochastic optimization.
\newblock {\em arXiv preprint arXiv:2206.00712}, 2022.

\bibitem{berahas2021sequential}
Albert~S Berahas, Frank~E Curtis, Daniel Robinson, and Baoyu Zhou.
\newblock Sequential quadratic optimization for nonlinear equality constrained stochastic optimization.
\newblock {\em SIAM Journal on Optimization}, 31(2):1352--1379, 2021.

\bibitem{berahas2025sequential}
Albert~S Berahas, Miaolan Xie, and Baoyu Zhou.
\newblock A sequential quadratic programming method with high-probability complexity bounds for nonlinear equality-constrained stochastic optimization.
\newblock {\em SIAM Journal on Optimization}, 35(1):240--269, 2025.

\bibitem{bertsekas2009convex}
Dimitri Bertsekas.
\newblock {\em Convex optimization theory}, volume~1.
\newblock Athena Scientific, Belmont, Massachusetts, 2009.

\bibitem{bollapragada2018adaptive}
Raghu Bollapragada, Richard Byrd, and Jorge Nocedal.
\newblock Adaptive sampling strategies for stochastic optimization.
\newblock {\em SIAM Journal on Optimization}, 28(4):3312--3343, 2018.

\bibitem{burke1989robust}
James~V Burke and Shih-Ping Han.
\newblock A robust sequential quadratic programming method.
\newblock {\em Mathematical Programming}, 43(1):277--303, 1989.

\bibitem{byrd2012sample}
Richard~H Byrd, Gillian~M Chin, Jorge Nocedal, and Yuchen Wu.
\newblock Sample size selection in optimization methods for machine learning.
\newblock {\em Mathematical programming}, 134(1):127--155, 2012.

\bibitem{byrd2008inexact}
Richard~H Byrd, Frank~E Curtis, and Jorge Nocedal.
\newblock An inexact sqp method for equality constrained optimization.
\newblock {\em SIAM Journal on Optimization}, 19(1):351--369, 2008.

\bibitem{byrd2010inexact}
Richard~H Byrd, Frank~E Curtis, and Jorge Nocedal.
\newblock An inexact newton method for nonconvex equality constrained optimization.
\newblock {\em Mathematical programming}, 122(2):273--299, 2010.

\bibitem{carter1991global}
Richard~G Carter.
\newblock On the global convergence of trust region algorithms using inexact gradient information.
\newblock {\em SIAM Journal on Numerical Analysis}, 28(1):251--265, 1991.

\bibitem{CC01a}
Chih-Chung Chang and Chih-Jen Lin.
\newblock {LIBSVM}: A library for support vector machines.
\newblock {\em ACM Transactions on Intelligent Systems and Technology}, 2:27:1--27:27, 2011.
\newblock Software available at \url{http://www.csie.ntu.edu.tw/~cjlin/libsvm}.

\bibitem{chen2018constraint}
Changan Chen, Frederick Tung, Naveen Vedula, and Greg Mori.
\newblock Constraint-aware deep neural network compression.
\newblock In {\em Proceedings of the European Conference on Computer Vision (ECCV)}, pages 400--415, 2018.

\bibitem{chen1994retrospective}
Huifen Chen and Bruce~W Schmeiser.
\newblock Retrospective approximation algorithms for stochastic root finding.
\newblock In {\em Proceedings of Winter Simulation Conference}, pages 255--261. IEEE, 1994.

\bibitem{chen2001stochastic}
Huifen Chen and Bruce~W Schmeiser.
\newblock Stochastic root finding via retrospective approximation.
\newblock {\em IIE Transactions}, 33(3):259--275, 2001.

\bibitem{curtis2025interior}
Frank~E Curtis, Shima Dezfulian, and Andreas Waechter.
\newblock An interior-point algorithm for continuous nonlinearly constrained optimization with noisy function and derivative evaluations.
\newblock {\em arXiv preprint arXiv:2502.11302}, 2025.

\bibitem{curtis2024single}
Frank~E Curtis, Xin Jiang, and Qi~Wang.
\newblock Single-loop deterministic and stochastic interior-point algorithms for nonlinearly constrained optimization.
\newblock {\em arXiv preprint arXiv:2408.16186}, 2024.

\bibitem{curtis2023stochastic}
Frank~E Curtis, Vyacheslav Kungurtsev, Daniel~P Robinson, and Qi~Wang.
\newblock A stochastic-gradient-based interior-point algorithm for solving smooth bound-constrained optimization problems.
\newblock {\em arXiv preprint arXiv:2304.14907}, 2023.

\bibitem{curtis2010matrix}
Frank~E Curtis, Jorge Nocedal, and Andreas W{\"a}chter.
\newblock A matrix-free algorithm for equality constrained optimization problems with rank-deficient jacobians.
\newblock {\em SIAM Journal on Optimization}, 20(3):1224--1249, 2010.

\bibitem{curtis2024worst}
Frank~E Curtis, Michael~J O’Neill, and Daniel~P Robinson.
\newblock Worst-case complexity of an sqp method for nonlinear equality constrained stochastic optimization.
\newblock {\em Mathematical Programming}, 205(1):431--483, 2024.

\bibitem{curtis2021inexact}
Frank~E Curtis, Daniel~P Robinson, and Baoyu Zhou.
\newblock Inexact sequential quadratic optimization for minimizing a stochastic objective function subject to deterministic nonlinear equality constraints.
\newblock {\em arXiv preprint arXiv:2107.03512}, 2021.

\bibitem{curtis2023sequential}
Frank~E Curtis, Daniel~P Robinson, and Baoyu Zhou.
\newblock Sequential quadratic optimization for stochastic optimization with deterministic nonlinear inequality and equality constraints.
\newblock {\em SIAM Journal on Optimization}, 34(4):3592--3622, 2024.

\bibitem{deng2009variable}
Geng Deng and Michael~C Ferris.
\newblock Variable-number sample-path optimization.
\newblock {\em Mathematical Programming}, 117(1-2):81--109, 2009.

\bibitem{dezfulian2024convergence}
Shima Dezfulian and Andreas W{\"a}chter.
\newblock On the convergence of interior-point methods for bound-constrained nonlinear optimization problems with noise.
\newblock {\em arXiv preprint arXiv:2405.11400}, 2024.

\bibitem{dolan2002benchmarking}
Elizabeth~D Dolan and Jorge~J Mor{\'e}.
\newblock Benchmarking optimization software with performance profiles.
\newblock {\em Mathematical programming}, 91:201--213, 2002.

\bibitem{fang2024fully}
Yuchen Fang, Sen Na, Michael~W Mahoney, and Mladen Kolar.
\newblock Fully stochastic trust-region sequential quadratic programming for equality-constrained optimization problems.
\newblock {\em SIAM Journal on Optimization}, 34(2):2007--2037, 2024.

\bibitem{friedlander2012hybrid}
Michael~P Friedlander and Mark Schmidt.
\newblock Hybrid deterministic-stochastic methods for data fitting.
\newblock {\em SIAM Journal on Scientific Computing}, 34(3):A1380--A1405, 2012.

\bibitem{ghadimi2016mini}
Saeed Ghadimi, Guanghui Lan, and Hongchao Zhang.
\newblock Mini-batch stochastic approximation methods for nonconvex stochastic composite optimization.
\newblock {\em Mathematical Programming}, 155(1-2):267--305, 2016.

\bibitem{gratton2024s2mpj}
Serge Gratton and Philippe~L Toint.
\newblock S2mpj and cutest optimization problems for matlab, python and julia.
\newblock {\em arXiv preprint arXiv:2407.07812}, 2024.

\bibitem{gurobi}
{Gurobi Optimization, LLC}.
\newblock {Gurobi Optimizer Reference Manual}, 2024.

\bibitem{han1977globally}
Shih-Ping Han.
\newblock A globally convergent method for nonlinear programming.
\newblock {\em Journal of optimization theory and applications}, 22(3):297--309, 1977.

\bibitem{jalilzadeh2016eg}
Afrooz Jalilzadeh and Uday~V Shanbhag.
\newblock eg-vssa: An extragradient variable sample-size stochastic approximation scheme: Error analysis and complexity trade-offs.
\newblock In {\em 2016 Winter Simulation Conference (WSC)}, pages 690--701. IEEE, 2016.

\bibitem{kotary2021end}
James Kotary, Ferdinando Fioretto, Pascal Van~Hentenryck, and Bryan Wilder.
\newblock End-to-end constrained optimization learning: A survey.
\newblock {\em arXiv preprint arXiv:2103.16378}, 2021.

\bibitem{lan2012optimal}
Guanghui Lan.
\newblock An optimal method for stochastic composite optimization.
\newblock {\em Mathematical Programming}, 133(1-2):365--397, 2012.

\bibitem{lan2020first}
Guanghui Lan.
\newblock {\em First-order and stochastic optimization methods for machine learning}, volume~1.
\newblock Springer, Atlanta, USA, 2020.

\bibitem{le2022survey}
Tai Le~Quy, Arjun Roy, Vasileios Iosifidis, Wenbin Zhang, and Eirini Ntoutsi.
\newblock A survey on datasets for fairness-aware machine learning.
\newblock {\em Wiley Interdisciplinary Reviews: Data Mining and Knowledge Discovery}, 12(3):e1452, 2022.

\bibitem{mehrabi2021survey}
Ninareh Mehrabi, Fred Morstatter, Nripsuta Saxena, Kristina Lerman, and Aram Galstyan.
\newblock A survey on bias and fairness in machine learning.
\newblock {\em ACM computing surveys (CSUR)}, 54(6):1--35, 2021.

\bibitem{na2023adaptive}
Sen Na, Mihai Anitescu, and Mladen Kolar.
\newblock An adaptive stochastic sequential quadratic programming with differentiable exact augmented lagrangians.
\newblock {\em Mathematical Programming}, 199(1-2):721--791, 2023.

\bibitem{na2023inequality}
Sen Na, Mihai Anitescu, and Mladen Kolar.
\newblock Inequality constrained stochastic nonlinear optimization via active-set sequential quadratic programming.
\newblock {\em Mathematical Programming}, 202(1):279--353, 2023.

\bibitem{nandwani2019primal}
Yatin Nandwani, Abhishek Pathak, and Parag Singla.
\newblock A primal dual formulation for deep learning with constraints.
\newblock {\em Advances in Neural Information Processing Systems}, 32, 2019.

\bibitem{newton2023retrospective}
David Newton.
\newblock {\em Retrospective approximation for smooth stochastic optimization}.
\newblock PhD thesis, Purdue University, 2023.

\bibitem{newtonretrospective}
David Newton, Raghu Bollapragada, Raghu Pasupathy, and Nung~Kwan Yip.
\newblock A retrospective approximation approach for smooth stochastic optimization.
\newblock {\em Mathematics of Operations Research}, 2024.

\bibitem{nocedal2006numerical}
J~Nocedal and SJ~Wright.
\newblock Numerical optimization 2nd edition springer.
\newblock {\em New York}, 2006.

\bibitem{o2024fast}
Thomas O'Leary-Roseberry and Raghu Bollapragada.
\newblock Fast unconstrained optimization via hessian averaging and adaptive gradient sampling methods.
\newblock {\em arXiv preprint arXiv:2408.07268}, 2024.

\bibitem{omojokun1989trust}
EO~Omojokun.
\newblock {\em Trust region algorithm for optimization with equalities and inequalities constraints}.
\newblock PhD thesis, Ph. D Thesis, University of Cororado at Boulder, 1989.

\bibitem{o2024two}
Michael~J O'Neill.
\newblock A two stepsize sqp method for nonlinear equality constrained stochastic optimization.
\newblock {\em arXiv preprint arXiv:2408.16656}, 2024.

\bibitem{pasupathy2010choosing}
Raghu Pasupathy.
\newblock On choosing parameters in retrospective-approximation algorithms for stochastic root finding and simulation optimization.
\newblock {\em Operations Research}, 58(4-part-1):889--901, 2010.

\bibitem{pasupathy2011introspective}
Raghu Pasupathy.
\newblock An introspective on the retrospective-approximation paradigm.
\newblock In {\em Proceedings of the 2011 Winter Simulation Conference (WSC)}, pages 412--421. IEEE, 2011.

\bibitem{pasupathy2021adaptive}
Raghu Pasupathy and Yongjia Song.
\newblock Adaptive sequential sample average approximation for solving two-stage stochastic linear programs.
\newblock {\em SIAM Journal on Optimization}, 31(1):1017--1048, 2021.

\bibitem{perold1984large}
Andre~F Perold.
\newblock Large-scale portfolio optimization.
\newblock {\em Management science}, 30(10):1143--1160, 1984.

\bibitem{phelps2016optimal}
Chris Phelps, Johannes~O Royset, and Qi~Gong.
\newblock Optimal control of uncertain systems using sample average approximations.
\newblock {\em SIAM Journal on Control and Optimization}, 54(1):1--29, 2016.

\bibitem{polak2008efficient}
E~Polak and JO~Royset.
\newblock Efficient sample sizes in stochastic nonlinear programming.
\newblock {\em Journal of Computational and Applied Mathematics}, 217(2):301--310, 2008.

\bibitem{powell2006fast}
Michael~JD Powell.
\newblock A fast algorithm for nonlinearly constrained optimization calculations.
\newblock In {\em Numerical Analysis: Proceedings of the Biennial Conference Held at Dundee, June 28--July 1, 1977}, pages 144--157. Springer, 2006.

\bibitem{powell1986recursive}
Michael~JD Powell and YJMP Yuan.
\newblock A recursive quadratic programming algorithm that uses differentiable exact penalty functions.
\newblock {\em Mathematical programming}, 35:265--278, 1986.

\bibitem{qiu2023sequential}
Songqiang Qiu and Vyacheslav Kungurtsev.
\newblock A sequential quadratic programming method for optimization with stochastic objective functions, deterministic inequality constraints and robust subproblems.
\newblock {\em arXiv preprint arXiv:2302.07947}, 2023.

\bibitem{ravi2019explicitly}
Sathya~N Ravi, Tuan Dinh, Vishnu~Suresh Lokhande, and Vikas Singh.
\newblock Explicitly imposing constraints in deep networks via conditional gradients gives improved generalization and faster convergence.
\newblock In {\em Proceedings of the AAAI Conference on Artificial Intelligence}, volume~33, pages 4772--4779, 2019.

\bibitem{roy2018geometry}
Soumava~Kumar Roy, Zakaria Mhammedi, and Mehrtash Harandi.
\newblock Geometry aware constrained optimization techniques for deep learning.
\newblock In {\em Proceedings of the IEEE conference on computer vision and pattern recognition}, pages 4460--4469, 2018.

\bibitem{royset2006optimal}
Johannes~O Royset, Armen Der~Kiureghian, and Elijah Polak.
\newblock Optimal design with probabilistic objective and constraints.
\newblock {\em Journal of Engineering Mechanics}, 132(1):107--118, 2006.

\bibitem{shapiro2021lectures}
Alexander Shapiro, Darinka Dentcheva, and Andrzej Ruszczynski.
\newblock {\em Lectures on stochastic programming: modeling and theory}.
\newblock SIAM, Philadelphia, PA, 2021.

\bibitem{summers2015stochastic}
Tyler Summers, Joseph Warrington, Manfred Morari, and John Lygeros.
\newblock Stochastic optimal power flow based on conditional value at risk and distributional robustness.
\newblock {\em International Journal of Electrical Power \& Energy Systems}, 72:116--125, 2015.

\bibitem{2020SciPy-NMeth}
Pauli Virtanen, Ralf Gommers, Travis~E. Oliphant, Matt Haberland, Tyler Reddy, David Cournapeau, Evgeni Burovski, Pearu Peterson, Warren Weckesser, Jonathan Bright, St{\'e}fan~J. {van der Walt}, Matthew Brett, Joshua Wilson, K.~Jarrod Millman, Nikolay Mayorov, Andrew R.~J. Nelson, Eric Jones, Robert Kern, Eric Larson, C~J Carey, {\.I}lhan Polat, Yu~Feng, Eric~W. Moore, Jake {VanderPlas}, Denis Laxalde, Josef Perktold, Robert Cimrman, Ian Henriksen, E.~A. Quintero, Charles~R. Harris, Anne~M. Archibald, Ant{\^o}nio~H. Ribeiro, Fabian Pedregosa, Paul {van Mulbregt}, and {SciPy 1.0 Contributors}.
\newblock {{SciPy} 1.0: Fundamental Algorithms for Scientific Computing in Python}.
\newblock {\em Nature Methods}, 17:261--272, 2020.

\bibitem{vondrak2009adaptive}
David~A Vondrak and R~Kevin Wood.
\newblock {\em Adaptive Selections of Sample Size and Solver Iterations in Stochastic Optimization with Application to Nonlinear Commodity Flow Problems}.
\newblock PhD thesis, Naval Postgraduate School, 2009.

\bibitem{wang2017penalty}
Xiao Wang, Shiqian Ma, and Ya-xiang Yuan.
\newblock Penalty methods with stochastic approximation for stochastic nonlinear programming.
\newblock {\em Mathematics of computation}, 86(306):1793--1820, 2017.

\bibitem{zhao2019physics}
Wen~Li Zhao, Pierre Gentine, Markus Reichstein, Yao Zhang, Sha Zhou, Yeqiang Wen, Changjie Lin, Xi~Li, and Guo~Yu Qiu.
\newblock Physics-constrained machine learning of evapotranspiration.
\newblock {\em Geophysical Research Letters}, 46(24):14496--14507, 2019.

\end{thebibliography}

\appendix

\section{Technical Results}
We present some technical results that have been used throughout the paper.
\begin{lemma} \label{lem:series_for_convergence}
    Given a non-negative sequence $\{b_k\} \rightarrow 0$ and a constant $\gamma \in [0, 1)$,
    \begin{equation*}
        \lim_{K \rightarrow \infty} \sum_{k = 0}^K b_k \gamma^{K - k} = 0.
    \end{equation*}
\end{lemma}
\begin{proof}
    Consider a finite index $\Bar{K}$ and a corresponding value $\bar{b} \geq 0$ such that $a_k \leq \bar{b}$ for all $k \geq \bar{K}$, which exist as $\{b_k\}$ is a non-negative sequence that converges to zero. For $K > \bar{K}$, the sum can be bounded as
    \begin{align*}
        \sum_{k=0}^{K}b_k \gamma^{k - k} &= \sum_{k=0}^{\bar{K}}b_k \gamma^{K - k} + \sum_{k=\bar{K}+1}^{K}b_k \gamma^{K - k} \leq b_0 \sum_{k=0}^{\bar{K}}\gamma^{K - k} + \bar{b} \sum_{k=\bar{K}+1}^{K}\gamma^{K - k} \leq b_0  \frac{\gamma^{K- \bar{K}}}{1 - \gamma} + \frac{\bar{b} \gamma}{1 - \gamma}.
    \end{align*}
    Given $\epsilon > 0$, if one selects $\bar{K}$ large enough such that $\bar{b} \leq \frac{\epsilon(1-\gamma)}{2\gamma}$ and $K$ large enough such that $b_0  \frac{\gamma^{K- \bar{K}}}{1 - \gamma} \leq \frac{\epsilon}{2}$, then $\sum_{k=0}^{K}b_k \gamma^{K - k} \leq \epsilon$. As one can always select $K$ and $\bar{K}$ large enough to make the non-negative sum arbitrarily small, it converges to zero.
\end{proof}

\begin{lemma} \label{lem:linear_convergence_induction}
    Given a non-negative sequence $\{Z_k\}$ such that $Z_{k+1} \leq \rho_1 Z_k + b \rho_2^k$ where $\rho_1, \rho_2 \in [0, 1)$ and $ 0 \leq b < \infty$, the sequence $\{Z_k\} \rightarrow 0$ at a linear rate as,
    \begin{equation*}
        Z_{k} \leq \max\{\rho_1 + \nu, \rho_2\}^{k+1} \max\left\{Z_0, \frac{b}{\nu}\right\},
    \end{equation*}
    where $\nu > 0$ such that $\rho_1 + \nu < 1$.
\end{lemma}
\bproof
The proof follows by induction. For $Z_0$, the result trivially holds. Then, if the result holds for $Z_k$ for $k \geq 0$,
\begin{align*}
    Z_{k+1} \leq \rho_1 Z_k + b \rho_2^k &\leq \rho_1 \max\{\rho_1 + \nu, \rho_2\}^k \max\left\{Z_0, \frac{b}{\nu}\right\} + b \rho_2^k \\
    &\leq \max\{\rho_1 + \nu, \rho_2\}^k \max\left\{Z_0, \frac{b}{\nu}\right\} \left(\rho_1  + \nu\right) \\
    &\leq \max\{\rho_1 + \nu, \rho_2\}^{k+1} \max\left\{Z_0, \frac{b}{\nu}\right\},
\end{align*}
thus completing the proof.
\eproof

\section{Additional Proofs} \label{appendix:proofs}

In this section, we present certain proofs that have omitted from the paper for brevity.

\newtheorem*{customtheorem1}{\cref{th:EQ_outer_iter_complexity_step_norm}}
\begin{customtheorem1} 
    Suppose Assumptions \ref{ass:gradient_errors} and \ref{ass:SQP_eq_assumptions} hold and that the batch size sequence $\{|S_k|\}$ is chosen to satisfy \cref{cond:eq_norm_test_d} in \cref{alg:Equality_Constrained_RA_SQP} with termination criterion \eqref{eq:termination_criterion_d}.
    \begin{enumerate}
        \item For the finite-sum problem \eqref{eq:intro_deter_error_obj}: For all $k \geq 0$, if termination criterion parameters are chosen as $0 \leq \gamma_k \leq \gamma < \tfrac{1 - \eta_d}{1 + \eta_d}$, $\epsilon_k = \omega \|\nabla f(x_{k, 0}) - g_{S_k}(x_{k, 0})\| + \hat{\omega}\beta^k$ with $\omega, \hat{\omega} \geq 0$ and \cref{cond:eq_norm_test_d} parameters are chosen such that $a_1 = \tfrac{1}{1 - \eta_d}\left[ \gamma(1 + \eta_d) + \theta\left(\omega + (\gamma + \hat{\theta})\tfrac{1 + \eta_d}{\kappa_{HJ}}\right) \right] < 1$, then, the true search direction norm converges at a linear rate across outer iterations, i.e., 
        \begin{align*}
            \|d_{k, N_k}^{true}\|
            \leq \max \{a_1 + \nu, \beta\}^{k+1} \max\left\{\|d_{0, 0}^{true}\|, \tfrac{a_2}{\nu}\right\},
        \end{align*}
        where $a_2 = \tfrac{a}{1 - \eta_d}\left(\omega + (\gamma + \hat{\theta})\tfrac{1 + \eta_d}{\kappa_{HJ}}\right) + \tfrac{\hat{\omega}}{1 - \eta_d}$ and $\nu > 0$ such that $a_1 + \nu < 1$.
        \item For the expectation problem \eqref{eq:intro_stoch_error_obj}: For all $k \geq 0$, if \cref{ass:Variance_regularity} holds, the termination criterion parameters are chosen as $0 \leq \gamma_k \leq \tilde{\gamma} < \tfrac{1 - \eta_d}{1 + \eta_d}$, $\epsilon_k = \tilde{\omega}\sqrt{\tfrac{\Var(\nabla F(x_{k, 0}) | \Fcal_k)}{|S_k|}}$ where $\tilde{\omega} \geq 0$ and \cref{cond:eq_norm_test_d} parameters are chosen such that $\tilde{a}_1 = \tfrac{1}{1- \eta_d}\left[\tilde{\gamma} (1 + \eta_d) + \tilde{\theta} \left(\tilde{\omega} + \tfrac{1 + \eta_d}{\kappa_{HJ}}\left(\tfrac{\epsilon_G + \kappa_g}{\kappa_{\sigma}} + \tilde{\gamma}\right)\right) \right] < 1$, then, the expected true search direction norm converges at a linear rate across outer iterations, i.e., 
        \begin{align*}
            \Embb\left[\|d_{k, N_k}^{true}\| \right]
            \leq \max \{\tilde{a}_1 + \tilde{\nu}, \tilde{\beta}\}^{k+1} \max \left\{\|d_{0, 0}^{true}\|, \tfrac{\tilde{a}_2}{\tilde{\nu}}\right\},
        \end{align*}
        where $\tilde{a}_2 = \tfrac{\tilde{a}}{1 - \eta_d} \left[\tilde{\omega} + \tfrac{1 + \eta_d}{\kappa_{HJ}}\left(\tfrac{\epsilon_G + \kappa_g}{\kappa_{\sigma}} + \tilde{\gamma}\right) \right]$ and $\tilde{\nu} > 0$ such that $\tilde{a}_1 + \tilde{\nu} < 1$.
    \end{enumerate}
\end{customtheorem1}
\begin{proof}
    The proof follows a similar procedure to \cref{th:EQ_outer_iter_complexity}. For the finite-sum problem \eqref{eq:intro_deter_error_obj}, the result from \cref{lem:step_norm_error_bound} can be further refined as,
    \begin{align*}
        \|d_{k, N_k}^{true}\|
        &\leq \gamma \tfrac{1 + \eta_d}{1 - \eta_d} \|d_{k-1, N_{k-1}}^{true}\| + \tfrac{1}{1 - \eta_d}\left( \omega \|\nabla f(x_{k, 0}) - g_{S_k}(x_{k, 0})\| + \hat{\omega}\beta^k\right) \\
        &\quad + \kappa_{HJ}^{-1} \tfrac{1 + \eta_d}{1 - \eta_d}\|\nabla f(x_{k, N_k}) - g_{S_k}(x_{k, N_k})\| + \gamma \kappa_{HJ}^{-1} \tfrac{1 + \eta_d}{1 - \eta_d}\|\nabla f(x_{k, 0}) - g_{S_k}(x_{k, 0})\| \\
        &\leq \gamma \tfrac{1 + \eta_d}{1 - \eta_d} \|d_{k-1, N_{k-1}}^{true}\| + \tfrac{1}{1 - \eta_d}\left( \omega \left(\theta \|d_{k-1, N_{k-1}}^{true}\| + a \beta^{k}\right) + \hat{\omega}\beta^k\right) \\
        &\quad + \kappa_{HJ}^{-1}\hat{\theta} \tfrac{1 + \eta_d}{1 - \eta_d}\left(\theta \|d_{k-1, N_{k-1}}^{true}\| + a \beta^{k}\right) + \gamma \kappa_{HJ}^{-1} \tfrac{1 + \eta_d}{1 - \eta_d}\left(\theta \|d_{k-1, N_{k-1}}^{true}\| + a \beta^{k}\right) \\
        &=  a_1 \|d_{k-1, N_{k-1}}^{true}\| + a_2 \beta^{k},
\end{align*}
where the second inequality follows from \cref{cond:eq_norm_test_d} and the final equality follows from the defined constants. Using the above bound, applying \cref{lem:linear_convergence_induction} with $Z_k = \|d_{k, N_k}^{true}\|$, $\rho_1 = a_1$, $\rho_2 = \beta$ and $b = a_2$ completes the proof.

For the expectation problem \eqref{eq:intro_stoch_error_obj}, the expectation result from \eqref{eq:ineq_main_error_bound_expec} can be further refined as,
\begin{align*}
    \Embb[\|d_{k, N_k}^{true}\| | \Fcal_k]
    &\leq  \tilde{\gamma} \tfrac{1 + \eta_d}{1 - \eta_d} \|d_{k-1, N_{k-1}}^{true}\| + \tfrac{\tilde{\omega}}{1 - \eta_d} \sqrt{\tfrac{\Var(\nabla F(x_{k, 0}) | \Fcal_k)}{|S_k|}} \\
    &\quad + \kappa_{HJ}^{-1} \tfrac{1 + \eta_d}{1 - \eta_d} \Embb[\|\nabla f(x_{k, N_k}) - g_{S_k}(x_{k, N_k})\| | \Fcal_k] \\
    &\quad + \tilde{\gamma} \kappa_{HJ}^{-1} \tfrac{1 + \eta_d}{1 - \eta_d} \Embb[\|\nabla f(x_{k, 0}) - g_{S_k}(x_{k, 0})\| | \Fcal_k] \\
    &\leq  \tilde{\gamma} \tfrac{1 + \eta_d}{1 - \eta_d} \|d_{k-1, N_{k-1}}^{true}\| + \tfrac{\tilde{\omega}}{1 - \eta_d} \sqrt{\tfrac{\Var(\nabla F(x_{k, 0}) | \Fcal_k)}{|S_k|}} \\
    &\quad + \kappa_{HJ}^{-1} \kappa_G \tfrac{1 + \eta_d}{1 - \eta_d} \tfrac{(\epsilon_G + \kappa_g)}{\sqrt{|S_k|}} + \tilde{\gamma} \kappa_{HJ}^{-1} \tfrac{1 + \eta_d}{1 - \eta_d} \Embb[\|\nabla f(x_{k, 0}) - g_{S_k}(x_{k, 0})\| | \Fcal_k] \\
    &\leq  \tilde{\gamma} \tfrac{1 + \eta_d}{1 - \eta_d} \|d_{k-1, N_{k-1}}^{true}\| + \tfrac{\tilde{\omega}}{1 - \eta_d} \left(\tilde{\theta} \|d_{k-1, N_{k-1}}^{true}\| + \tilde{a}\tilde{\beta}^{k}\right) \\
    &\quad + \kappa_{HJ}^{-1} \tfrac{1 + \eta_d}{1 - \eta_d} \tfrac{(\epsilon_G + \kappa_g)}{\kappa_{\sigma}}\left(\tilde{\theta} \|d_{k-1, N_{k-1}}^{true}\| + \tilde{a}\tilde{\beta}^{k}\right) + \tilde{\gamma} \kappa_{HJ}^{-1} \tfrac{1 + \eta_d}{1 - \eta_d} \left(\tilde{\theta} \|d_{k-1, N_{k-1}}^{true}\| + \tilde{a}\tilde{\beta}^{k}\right) \\
    &= \tilde{a}_1 \|d_{k-1, N_{k-1}}^{true}\|  + \tilde{a}_2\tilde{\beta}^k,
\end{align*}
where the second inequality follows from \eqref{eq:biased_gradient_error_bound} and \cref{ass:Variance_regularity}, the third inequality follows by substituting $\kappa_G \leq \tfrac{\sqrt{\Var(\nabla F(x_{k, 0}) | \Fcal_k)}}{\kappa_{\sigma}}$ and \cref{cond:eq_norm_test_d} and the last equality follows from the defined constants. Taking the total expectation of the bound yields,
\begin{align*}
    \Embb\left[\|d_{k, N_k}^{true}\| \right]
    \leq \tilde{a}_1 \Embb\left[\|d_{k-1, N_{k-1}}^{true}\|\right]  + \tilde{a}_2\tilde{\beta}^k.
\end{align*}
Applying \cref{lem:linear_convergence_induction} with $Z_k = \Embb\left[\|d_{k, N_k}^{true}\| \right]$, $\rho_1 = \tilde{a}_1 $, $\rho_2 = \tilde{\beta}$ and $b = \tilde{a}_2$ completes the proof.
\end{proof}

\newtheorem*{customcorollary1}{\cref{cor:ineq_geometric_batch_increase_outer_iter_complexity}}
\begin{customcorollary1} 
    Suppose the conditions of \cref{th:outer_complexity_ineq} hold.
    \begin{enumerate}
        \item For the finite-sum problem \eqref{eq:intro_deter_error_obj}: If the batch size is selected as $|S_{k}| =\lceil (1 - \beta^k)|\Scal| \rceil$ with $\beta \in (0, 1)$ and termination criterion parameters are chosen as $0 \leq \{\gamma_k\} \leq \gamma < 1$, $\epsilon_k = \omega \left(1 - \tfrac{|S_k|}{|\Scal|}\right)$ with $\omega \geq 0$, then the true search direction norm converges to zero at a linear rate across outer iterations, as expressed in \cref{th:outer_complexity_ineq} with $a_1 = \gamma$ and $a_2 = \omega + 2\mu_H^{-1}(1 + \gamma)(\omega_1 \kappa_g + \omega_2)$.
        
        \item For the expectation problem \eqref{eq:intro_stoch_error_obj}: If \cref{ass:Variance_regularity} is satisfied, the sample set size is chosen as $|S_{k+1}| = \left\lceil\tfrac{|S_k|}{\tilde{\beta}^2}\right\rceil$ with $\tilde{\beta} \in (0, 1)$ and the termination criterion parameters are chosen such that $0 \leq \{\gamma_k\} \leq \tilde{\gamma} < 1$, $\epsilon_k = \tfrac{\tilde{\omega}}{\sqrt{|S_k|}}$ where $\tilde{\omega} \geq 0$, then the true search direction norm converges to zero in expectation at a linear rate across outer iterations, as expressed in \cref{th:outer_complexity_ineq} with $\tilde{a}_1 = \tilde{\gamma}$ and $\tilde{a}_2 = \tfrac{\tilde{\omega} + \kappa_G\mu_H^{-1}(\epsilon_G + \kappa_g) + \tilde{\gamma}\mu_H^{-1} (\tilde{\omega}_1 \kappa_g + \tilde{\omega}_2)} {\sqrt{|S_0|}}$, where $|S_0|$ is the initial batch size.
    \end{enumerate}
\end{customcorollary1}
\begin{proof}
    The proof follows from the same procedure as \cref{cor:Eq_geometric_batch_increase_outer_iter_complexity}.
    For the finite-sum problem \eqref{eq:intro_deter_error_obj}, substituting \eqref{eq:deter_sampled_gradient_error} and the defined parameters into error bound \eqref{eq:ineq_main_error_bound},
    \begin{align*}
        \|d_{k, N_k}^{true}\|
        &\leq  \gamma \|d_{k-1, N_{k-1}}^{true}\| + \omega \left(1 - \tfrac{|S_k|}{|\Scal|}\right)  \\
        &\quad + \mu_H^{-1}\left\|\nabla f(x_{k, N_k}) - g_{S_k}(x_{k, N_k})\right\| + \gamma\mu_H^{-1} \|\nabla f(x_{k, 0}) - g_{S_k}(x_{k, 0})\|\\
        &\leq  \gamma \|d_{k-1, N_{k-1}}^{true}\| + \omega \left(\tfrac{|\Scal| - (1 - \beta^k)|\Scal|}{|\Scal|}\right)  \\
        &\quad + 2\mu_H^{-1}\left(\tfrac{|\Scal| - (1 - \beta^k)|\Scal|}{|\Scal|}\right)\left(\omega_1 \|\nabla f(x_{k, N_k})\| + \omega_2 + \gamma (\omega_1 \|\nabla f(x_{k, 0})\| + \omega_2)\right)\\
        &\leq  \gamma \|d_{k-1, N_{k-1}}^{true}\| + (\omega + 2\mu_H^{-1}(1 + \gamma)(\omega_1 \kappa_g + \omega_2))\beta^k \\
        &=  \gamma \|d_{k-1, N_{k-1}}^{true}\| + a_2 \beta^k,
    \end{align*}
    where the third inequality follows from \cref{ass:SQP_ineq_assumptions} and the final equality from the defined constants. 
    Using the above bound, applying \cref{lem:linear_convergence_induction} with $Z_k = \|d_{k, N_k}^{true}\|$, $\rho_1 = \gamma$, $\rho_2 = \beta$ and $b = a_2$ completes the proof.
    
    For the expectation problem \eqref{eq:intro_stoch_error_obj}, the expectation error bound \eqref{eq:ineq_main_error_bound_expec} can be further refined as,
    \begin{align*}
        \Embb\left[\|d_{k, N_k}^{true}\| | \Fcal_k \right]
        &\leq \tilde{\gamma} \|d_{k-1, N_{k-1}}^{true}\| + \tfrac{\tilde{\omega}}{\sqrt{|S_k|}} +\mu_H^{-1} \Embb\left[\left\|\nabla f(x_{k, N_k}) - g_{S_k}(x_{k, N_k})\right\| | \Fcal_k \right] \\
        &\quad + \tilde{\gamma}\mu_H^{-1} \Embb\left[\|\nabla f(x_{k, 0}) - g_{S_k}(x_{k, 0})\| | \Fcal_k \right]\\
        &\leq \tilde{\gamma} \|d_{k-1, N_{k-1}}^{true}\| + \tfrac{\tilde{\omega}}{\sqrt{|S_k|}} + \kappa_G\mu_H^{-1}\tfrac{(\epsilon_G + \kappa_g)}{\sqrt{|S_k|}} + \tilde{\gamma}\mu_H^{-1} \tfrac{\tilde{\omega}_1 \|\nabla f(x_{k, 0})\| + \tilde{\omega}_2}{\sqrt{|S_k|}}\\
        &\leq \tilde{\gamma} \|d_{k-1, N_{k-1}}^{true}\| + \tfrac{\tilde{\beta}^k}{\sqrt{|S_0|}}\left[\tilde{\omega} + \kappa_G\mu_H^{-1}(\epsilon_G + \kappa_g) + \tilde{\gamma}\mu_H^{-1} (\tilde{\omega}_1 \kappa_g + \tilde{\omega}_2) \right] \\
        &= \tilde{\gamma} \|d_{k-1, N_{k-1}}^{true}\| + \tilde{a}_2 \tilde{\beta}^k,
    \end{align*}
    where the second inequality follows from \eqref{eq:biased_gradient_error_bound} with \cref{ass:Variance_regularity} and \eqref{eq:stoch_sampled_gradient_error}, the third inequality follows from the form of $S_k$ and \cref{ass:SQP_ineq_assumptions} and the equality follows from the defined constants. Taking the total expectation of the above bound yields,
    \begin{align*}
        \Embb\left[\|d_{k, N_k}^{true}\|\right]
        \leq \tilde{\gamma} \Embb\left[\|d_{k-1, N_{k-1}}^{true}\|\right] + \tilde{a}_2 \tilde{\beta}^k.
    \end{align*}
    Applying \cref{lem:linear_convergence_induction} with $Z_k = \Embb\left[\|d_{k, N_k}^{true}\|\right]$, $\rho_1 = \tilde{\gamma} $, $\rho_2 = \tilde{\beta}$ and $b = \tilde{a}_2$ completes the proof.
\end{proof}

\newtheorem*{customlemma1}{\cref{lem:ineq_large_enough_batch_sizes}}
\begin{customlemma1}
    Suppose Assumptions \ref{ass:gradient_errors}, \ref{ass:SQP_ineq_assumptions} and \ref{ass:MFCQ_ineq} hold in \cref{alg:Inequality_Constrained_RA_SQP}.
    \begin{enumerate}
        \item For the finite-sum problem \eqref{eq:intro_deter_error_obj}: For $k \geq 0$, \cref{cond:ineq_norm_test} is satisfied if
        \begin{equation*}
            |S_k| \geq |\Scal|\left(1 - \sqrt{\tfrac{\theta^2 \|d_{k, 0}^{true}\|^2 + a^2 \beta^{2k}}{4(\omega_1^2 \kappa_g^2 + \omega_2^2)}}\right) \quad \text{with} \quad \hat{\theta} = 1.
        \end{equation*}
        \item For the expectation problem \eqref{eq:intro_stoch_error_obj}:
        For $k \geq 0$, \cref{cond:ineq_norm_test} is satisfied if
        \begin{equation*}
            |S_k| \geq  \tfrac{\tilde{\omega}_1^2 \kappa_g^2 + \tilde{\omega}_2^2}{\tilde{\theta}^2 \|d_{k, 0}^{true}\|^2 + \tilde{a}^2 \tilde{\beta}^{2k}}.
        \end{equation*}
    \end{enumerate}
\end{customlemma1}
\begin{proof}
    Similar to the procedure employed in \cref{lem:eq_large_enough_batch_sizes}, for the finite-sum problem \eqref{eq:intro_deter_error_obj}, from \cref{ass:gradient_errors}
    \begin{align*}
        \|\nabla f(x_k, 0) - g_{S_k} (x_k, 0)\|^2 &\leq 4\left(1 - \tfrac{|S_k|}{|\Scal|}\right)^2 (\omega_1^2 \|\nabla f (x_{k, 0})\|^2 + \omega_2)
        \\
        &\leq 4 \left(\sqrt{\tfrac{\theta^2 \|d_{k, 0}^{true}\|^2 + a^2 \beta^{2k}}{4(\omega_1^2 \kappa_g^2 + \omega_2^2)}}\right)^2 (\omega_1^2 \kappa_g^2 + \omega_2)
    \end{align*}
    where the first inequality follows from \eqref{eq:deter_sampled_gradient_error} and second inequality from \cref{ass:SQP_ineq_assumptions} and the suggested lower bound on $|S_k|$, thereby satisfying \cref{cond:ineq_norm_test}. The gradient error condition on $x_{k, N_k}$ with $\hat{\theta} = 1$ can be verified in the same way. 
    
    For the expectation problem \eqref{eq:intro_stoch_error_obj},
    \begin{align*}
        \Embb[\|\nabla f(x_k, 0) - g_{S_k} (x_k, 0)\|^2 | \Fcal_k] &\leq \tfrac{\tilde{\omega}_1^2 \|\nabla f(x_{k, 0})\|^2 + \tilde{\omega}_2^2}{|S_k|}         \\   
        &\leq (\tilde{\omega}_1^2 \kappa_g^2 + \tilde{\omega}_2^2)\tfrac{\tilde{\theta}^2 \|d_{k, 0}^{true}\|^2 + \tilde{a}^2 \tilde{\beta}^{2k}}{\tilde{\omega}_1^2 \kappa_g^2 + \tilde{\omega}_2^2}
    \end{align*}
    where the first inequality follows from \eqref{eq:stoch_sampled_gradient_error} and second inequality from \cref{ass:SQP_ineq_assumptions} and the suggested lower bound on $|S_k|$, thereby satisfying \cref{cond:ineq_norm_test}, thus completing the proof.
\end{proof}

\end{document}